\documentclass[noinfoline]{imsart}

\RequirePackage[OT1]{fontenc}
\RequirePackage{amsthm,amsmath}
\RequirePackage[numbers]{natbib}
\RequirePackage[colorlinks,citecolor=blue,urlcolor=blue]{hyperref}

\usepackage{verbatim}
\usepackage{algorithm}
\usepackage{booktabs}

\allowdisplaybreaks

\usepackage{fullpage}

\newcommand {\R}{\mathbb{R}}
\newcommand {\N}{\mathbb{N}}
\renewcommand {\O}{\mathcal{O}}
\newcommand {\Prob}{\mathbb{P}}

\newcommand{\E}{\mathbb{E}}

\newcommand {\cov}{\textrm{Cov}}

\newcommand {\var}{\textrm{Var}}
\newcommand {\1}{\textrm{\textbf{1}}}

\newcommand{\verti}[1]{{\left\vert\kern-0.25ex\left\vert\kern-0.25ex\left\vert #1 
    \right\vert\kern-0.25ex\right\vert\kern-0.25ex\right\vert}}
\newcommand{\ct}{\,\tilde\otimes\,}

\renewcommand{\t}[1]{#1^\top}
\newcommand{\argmin}{\operatornamewithlimits{arg\,min}}

\newcommand{\mat}[1]{\mathbf{#1}}
\newcommand{\spa}[1]{\mathcal{#1}}
\newcommand{\ope}[1]{#1}
\newcommand{\fun}[1]{\lowercase{#1}}
\newcommand{\set}[1]{{#1}}
\newcommand{\Chat}{\widehat{\ope C}}
\newcommand{\ip}[1]{{\left\langle\kern-0.5ex\left\langle\kern-0.5ex\left\langle #1 
    \right\rangle\kern-0.5ex\right\rangle\kern-0.5ex\right\rangle}}
\newcommand{\vertj}{\vert\kern-0.25ex\vert\kern-0.25ex\vert}

\usepackage{subfigure}
\usepackage{layout}

\usepackage{dsfont}
\usepackage[mathscr]{eucal}
\usepackage[toc,page]{appendix}
\usepackage{mathrsfs}
\usepackage{color}
\usepackage{pifont}
\usepackage{bm}
\usepackage{latexsym}
\usepackage{amsfonts}
\usepackage{amssymb}
\usepackage{epsfig}
\usepackage{graphicx}
\usepackage{multirow}
\usepackage{caption}
\newtheorem{theorem}{Theorem}
\newtheorem{proposition}{Proposition}
\newtheorem{corollary}{Corollary}
\newtheorem{lemma}{Lemma}
\newtheorem{definition}{Definition}
\newtheorem{remark}{Remark}
\newtheorem{example}{Example}

\startlocaldefs
\numberwithin{equation}{section}
\theoremstyle{plain}
\endlocaldefs

\begin{document}

\begin{frontmatter}

\title{{\large Separable Expansions for Covariance Estimation}}

\runtitle{Separable Expansions for Covariance Estimation}

\begin{aug}
\author{\fnms{Tomas} \snm{Masak}\ead[label=e1]{tomas.masak@epfl.ch}},
\author{\fnms{Soham} \snm{Sarkar}\ead[label=e2]{soham.sarkar@epfl.ch}} \and
\author{\fnms{Victor M.} \snm{Panaretos}\ead[label=e3]{victor.panaretos@epfl.ch}}

\thankstext{t1}{Research supported by a Swiss National Science Foundation grant.}

\runauthor{T.~Masak \& S.~Sarkar \& V.M.~Panaretos}

\affiliation{Ecole Polytechnique F\'ed\'erale de Lausanne}

\address{Institut de Math\'ematiques\\
Ecole Polytechnique F\'ed\'erale de Lausanne\\
e-mail: \href{mailto:tomas.masak@epfl.ch}{tomas.masak@epfl.ch},
\href{mailto:soham.sarkar@epfl.ch}{soham.sarkar@epfl.ch}, \href{mailto:victor.panaretos@epfl.ch}{victor.panaretos@epfl.ch}
}

\end{aug}

\begin{abstract} 
The non-parametric estimation of covariance lies at the heart of functional data analysis, whether for curve or surface-valued data. The case of a two-dimensional domain poses both statistical and computational challenges, which are typically alleviated by assuming separability. However, separability is often questionable, sometimes even demonstrably inadequate. We propose a framework for the analysis of covariance operators of random surfaces that generalises separability, while retaining its major advantages. Our approach is based on the expansion of the covariance into a series of separable terms. The expansion is valid for any covariance over a two-dimensional domain. Leveraging the key notion of the partial inner product, we extend the power iteration method to general Hilbert spaces and show how the aforementioned expansion can be efficiently constructed in practice. Truncation of the expansion and retention of the leading terms automatically induces a non-parametric estimator of the covariance, whose parsimony is dictated by the truncation level. The resulting estimator can be calculated, stored and manipulated with little computational overhead relative to separability. Consistency and rates of convergence are derived under mild regularity assumptions, illustrating the trade-off between bias and variance regulated by the truncation level. The merits and practical performance of the proposed methodology are demonstrated in a comprehensive simulation study and on classification of EEG signals.
\end{abstract}



\end{frontmatter}

\tableofcontents

\newpage
\section{Introduction}

We consider the interlinked problems of parsimonous representation, efficient estimation, and tractable manipulation of a random surface's covariance, i.e.\ the covariance of a random process on a two-dimensional domain. We operate in the framework of \emph{functional data analysis} \citep[see][]{ramsay2002,hsing2015}, which treats the process' realizations as elements of a separable Hilbert space, and assumes the availability of replicates thereof, thus allowing for nonparametric estimation. As a specific example, consider a spatio-temporal process $\ope X = (\fun X(t,s), t~\in \set T, s \in \set S)$ taking values in $\spa L_2(\set T \times \set S)$ with covariance kernel $c(t,s,t',s') = \cov(\fun X(t,s),\fun X(t',s'))$, and induced covariance operator $\ope C: \spa L_2(\set T \times \set S) \to \spa L_2(\set T \times \set S)$. We assume that we have access to (potentially discretised) i.i.d.\ realisations  $\ope X_1, \ldots, \ope X_N$ of $\ope X$, and wish to estimate $c$ nonparametrically and computationally feasibly, ideally via a parsimonious representation allowing for tractable further computational manipulations (e.g.\ inversion) required in key tasks involving $c$ (regression, prediction, classification).

Although the nonparametric estimation of covariance kernels is still an active field of research in functional data analysis, it is safe to say that the problem is well understood for curve data (i.e.~functional observations on one-dimensional domains), see \cite{wang2016} for a complete overview. The same cannot be said about surface-valued data (i.e.\ functional observations on bivariate domains). Even though most, if not all, univariate procedures can be in theory taken and adapted to the case of surface-valued data, in practice one quickly runs into computational and statistical limitations associated with the dimensionality of the problem. For instance, suppose that we observe the surfaces densely at $K_1$ temporal and $K_2$ spatial locations. Estimation of the empirical covariance then requires evaluations at $K_1^2 K_2^2$ points. Even storage of the empirical covariance is then prohibitive for grid sizes as small as $K_1 \approx K_2 \approx 100$. Moreover, statistical constraints -- accompanying the necessity to reliably estimate $K_1^2 K_2^2$ unknown parameters from only $N K_1 K_2$ observations -- are usually even tighter \citep{aston2017}. %An abundant source of such data sets is bio-medical imaging (such as EEG or fMRI), where data are typically noiseless and sampled on regular and dense grids. \tomas{I would suggest to delete the last sentence here and use the bronze paragraph below instead.}
Examples of random surfaces observed on a~grid densely enough to impede the usage of an unstructured covariance were given by \citet{wang2016}, and labeled ``next-generation functional data''. For instance, they arise abundantly in neuroimaging (e.g.\ EEG or fMRI studies).

Due to the aforementioned challenges associated with higher dimensionality, additional structure is often imposed on the spatio-temporal covariance as a modeling assmption. Perhaps the most prevalent assumption is that of \emph{separability}, factorizing the covariance kernel $c$~into a~purely spatial part and a purely temporal part:
$c(t,s,t',s') = a(t,t')\, b(s,s')$, $t,t' \in \set T$, \mbox{$s,s' \in \set S$}.
When the data are observed on a grid, separability entails that the 4-way covariance tensor $\mat C \in \R^{K_1 \times K_2 \times K_1 \times K_2}$ simplifies into an outer product of two matrices, say $\mat A \in \R^{K_1 \times K_1}$ and $\mat B \in \R^{K_2 \times K_2}$. This reduces the number of parameters to be estimated from $\O(K_1^2 K_2^2)$ to $\O(K_1^2 + K_2^2)$, simplifying estimation from both computational and statistical viewpoints. Moreover, subsequent manipulation of the estimated covariance becomes much simpler. However, assuming separability often encompasses oversimplification and has undesirable practical implications for real data \citep[see][]{rougier2017}. A number of tests for separability of space-time functional data have been recently developed \citep{aston2017,bagchi2020,constantinou2017}, and their application demonstrates that separability is distinctly violated for several data sets previously modeled as separable.

If separability of the covariance is rejected for a data set by one of the aforementioned tests, an alternative model with similar computational advantages to separability is needed. But there is a clear deficiency in this regard in the present literature. While many authors focus on testing separability \citep[see][for a list of up-to-date references]{chen2021} or assessing the departures from it \citep{huang2019,dette2020}, little work has been done to offer non-parametric alternatives to covariance estimation beyond separability. This becomes evident from the fact that separability is often assumed in practice, not because it is believed to hold, but merely due to the computational gains it offers \citep[see][]{gneiting2006,pigoli2018}.

In this article, we aspire to bridge this gap and go beyond separability in covariance estimation. To this effect, we introduce and study a decomposition allowing for representation, estimation and manipulation of a covariance on a two-dimensional domain. This decomposition can be viewed as a generalization of separability, and applies to \emph{any} covariance. By truncating this decomposition, we are able to obtain the sought computational and statistical efficiency. Our approach is motivated by the fact that any Hilbert-Schmidt operator on a product Hilbert space can be additively decomposed in a~separable manner. For example, the covariance kernel $c$ can be decomposed as
\begin{equation}\label{eq:separable_decomposition_space_time}
c(t,s,t^\prime,s^\prime) = \sum_{r=1}^\infty \sigma_r\, a_r(t,t^\prime)\,b_r(s,s^\prime),
\end{equation}
where $(\sigma_r)_{r \ge 1}$ is the non-increasing and non-negative sequence of {\em scores} and $(a_r)_{r \ge 1}$, resp. $(b_r)_{r \ge 1}$, is an orthonormal basis of $\spa L_2(\set T \times \set T)$, resp. $\spa L_2(\set S \times \set S)$. We call \eqref{eq:separable_decomposition_space_time} the {\em separable expansion} of $c$, because the spatial and temporal dimensions of $c$ are separated in each term. This fact distinguishes the separable expansion from the eigendecomposition of $c$, i.e.
$
c(t,s,t^\prime,s^\prime) = \sum_{r=1}^\infty \lambda_r\, e_r(t,s)\,e_r(t^\prime,s^\prime)
$,
which contains no purely temporal or purely spatial components.

The separable expansion \eqref{eq:separable_decomposition_space_time} effectively analyzes $c$ into its principal \emph{separable components}, much in the same way its eigendecompostion analyzes $c$ into its  \emph{principal components}. The retention of a few leading terms in \eqref{eq:separable_decomposition_space_time} represents a parsimonious reduction generalizing separability, and affording similar advantages. The \emph{nearest Kronecker product} \citep{vanloan1993}, or equivalently the \emph{best separable approximation} \citep{genton2007}, are directly related to the leading term in \eqref{eq:separable_decomposition_space_time}.
The subsequent terms capture departures from separability, yet still in a~separable manner.
We study estimators of $c$ obtained by truncating its separable expansion, i.e.~estimators of the form
\begin{equation}\label{eq:R_separable_space_time}
\widehat c(t,s,t^\prime,s^\prime) = \sum_{r=1}^R \hat\sigma_r\, \hat a_r(t,t^\prime)\,\hat b_r(s,s^\prime)
\end{equation}
for an appropriate choice of the \emph{degree-of-separability} $R \in \N$. Our approach should not be viewed as a model: any covariance can be approximated in this way to arbitrary precision, when $R$ is chosen sufficiently large. On the other hand, {as in principal component analysis}, the expansion is most useful for a given covariance when $R$ can be chosen to be relatively low, and we explore the role of $R$ in estimation quality and computational tractability. Our asymptotic theory illustrates how the bias/variance of the resulting estimator is dictated by its degree-of-separability.

To use \eqref{eq:separable_decomposition_space_time}, we must of course be able to construct it, and our starting point is an iterative (provably convergent) algorithm for calculating the separable expansion of a general operator. Conceptually, an $R$-separable estimator of the form \eqref{eq:R_separable_space_time}  would be obtained by applying the algorithm to the empirical covariance estimator. We show, however, that the estimator can be constructed directly at the level of the 2D surface data, without the need to ever compute or store the 4D empirical covariance. The algorithm consists of a~generalization of the power iteration method to arbitrary Hilbert spaces, and an operation called the \emph{partial inner product} lies at its foundation.

The partial inner product was used implicitly by \cite{bagchi2020} and explicitly under the name `partial product' in the follow-up work by \cite{dette2020}, both in the context of separability tests. These were possibly the first to consider this operation over infinite-dimensional spaces. In another avenue of research, generalizations of separability similar to \eqref{eq:R_separable_space_time} have been considered before, in a finite-dimensional context. Based on the groundwork on Kronecker approximations \citep{vanloan1993, vanloan2000}, \cite{munck2005} used Lagrange multipliers to estimate a covariance model similar to \eqref{eq:R_separable_space_time} with $R=2$ only, and a convex relaxation of an optimization problem with a non-convex constraint leading to a similar estimator was considered in \cite{hero2013a}. In an infinite-dimensional context, \cite{lynch2018} introduced the \emph{weakly separable model}, which bears some similarity to \eqref{eq:R_separable_space_time}, but places additional structural assumptions. These lead to some useful properties, but also restricts generality -- it is in effect a \emph{model}, not possessing optimal reduction properties.

Our approach differs from the aforementioned references in several important aspects, which add up to the main contributions of the present paper:
\begin{enumerate}%[leftmargin=0.5cm,itemindent=0cm]

\item The goal of this paper is estimation beyond the separability assumption. We consider the separable expansion \eqref{eq:separable_decomposition_space_time}, which is valid at the level of operators on general (possibly infinite-dimensional) Hilbert spaces. To construct this expansion, we provide in Section~\ref{sec:PIP} a generalization of the power iteration method based on a broader definition of the partial inner product, which yields a natural framework for generalizations of separability such as \eqref{eq:R_separable_space_time}.
    
\item We place emphasis on the computational aspect, considering the mere evaluation of the empirical covariance estimator as too expensive. We make the computationally consequential observation that the estimator \eqref{eq:R_separable_space_time} can be calculated directly at the level of data, significantly reducing the number of operations needed to construct the estimator (c.f.\ Section~\ref{sec:estimation}). This puts our approach in a sharp contrast to the aforementioned prior works, where the empirical covariance needs to be fully computed (even though sometimes not stored) and is possible due to the availability of replicated observations.

\item We provide an efficient inversion algorithm for covariances of the form \eqref{eq:R_separable_space_time}, which is based on the preconditioned conjugate gradient. This allows for subsequent manipulation of the estimated covariance (e.g.\ using it for prediction) with ease comparable to that of a separable model. Moreover, the inversion algorithm can also be utilized for other estimators with a~similar structure, e.g.\ for the weakly separable estimator of \cite{lynch2018}, who left the question of efficient inversion up for further research.

\item We develop asymptotic statistical theory, including convergence rates, for the truncation estimator that apply both to the case of fully observed data and to the case of functional data observed on a grid with noise contamination. The latter is important especially when relating statistical and computational efficiency. These results are genuinely nonparametric, in the sense that we make only very mild assumptions on the underlying covariance $\ope C$.

\item Finally, we examine the effect of the degree-of-separability $R$ on the performance of the estimation procedure. We provide both theoretical guidance for the choice of $R$, as well as practical means in the form of a cross-validation strategy to automatically select $R$ based on the data. This is an essential point as truncating the expansion elicits a form parsimony, and the choice of $R$ has a genuine regularization effect, trading off between bias and variance.
\end{enumerate}

The advantages of estimation beyond separability using the proposed estimator \eqref{eq:R_separable_space_time} are demonstrated in a comprehensive simulation study and on the problem of classification of EEG signals. The latter task additionally illustrates the importance of a suitable covariance estimator, in terms of both quality and ease of computational manipulation, which is offered by the proposed methodology.

\section{Mathematical Background}\label{sec:background}

Let $\spa H$ be a real separable Hilbert space. The Banach space of bounded linear transformations (operators, in short) on $\spa H$ is denoted by $\spa S_\infty(\spa H)$ and equipped with the operator norm $\verti{\ope F}_\infty = \sup_{\| x \|=1} \| \ope F x \|$. An operator $\ope F \in \spa S_{\infty}(\spa H)$ is compact if it can be written as
$\ope F x = \sum_{j=1}^\infty \sigma_j \langle e_j, x \rangle_{\spa H} f_j$ for a sequence of non-negative numbers $\{\sigma_j\}$ decreasing to $0$, and some orthonormal bases $\{ e_j \}_{j=1}^\infty$ and $\{ f_j \}_{j=1}^\infty$ of $\spa H$.
For $p \geq 1$, a compact operator $\mathcal{F}$ on $\spa H$ belongs to the Schatten class $\spa S_p(\spa H)$ if $\verti{\ope F}_p^p = \sum_{j=1}^\infty \sigma_j^p < \infty$. When equipped with the norm $\verti{\cdot}_p$, the space $\spa S_p(\spa H)$ is a Banach space. In particular, $\spa S_2(\spa H)$ is the space of Hilbert-Schmidt operators. It is itself a Hilbert space with the inner product of $\ope F,\ope G \in \spa S_2(\spa H)$ defined as
$\langle \ope F, \ope G \rangle_{\spa S_2(\spa H)} = \sum_{j=1}^\infty \langle \ope F e_j , \ope G e_j \rangle_{\spa H}$,
where $\{ e_j\}_{j=1}^\infty$ is an arbitrary orthonormal basis of $\mathcal \spa H$.

In the case of $\spa H=\spa L_2(E,\spa E,\mu)$, where $(E,\spa E,\mu)$ is a measure space, every $\ope G \in \spa S_2(\spa H)$ is an integral operator associated with a kernel $g \in \spa L_2(E \times E, \spa E \otimes \spa E, \nu)$, where $\spa E \otimes \spa E$ is the product $\sigma$-algebra and $\nu$ the product measure. Hence the space of operators and the space of kernels are isometrically isomorphic, which we write as $\spa L_2(E \times E, \spa E \otimes \spa E, \nu) \simeq \spa S_2(\spa H)$. For $y=\ope G x$, it holds
\[
y(t) = \int_E g(t,t') x(t') d \mu(t').
\]

The tensor product of two Hilbert spaces $\spa H_1$ and $\spa H_2$, denoted by $\spa H := \spa H_1 \otimes \spa H_2$, is the completion of the following set of finite linear combinations of abstract tensor products \mbox{(c.f.~\citealp{weidmann2012}):}
\begin{equation}\label{eq:setOfTensors}
\bigg\{ \sum_{j=1}^m x_j \otimes y_j \,:\, x_j \in \spa H_1, y_j \in \spa H_2, m \in \N \bigg\},
\end{equation}
under the inner product $\langle x_1 \otimes y_1, x_2 \otimes y_2 \rangle_{\spa H} = \langle x_1, x_2\rangle_{\spa H_1} \langle y_1, y_2\rangle_{\spa H_2}$, for all $x_1,x_2 \in \spa H_1$ and $y_1,y_2 \in \spa H_2$. If $\{ e_j \}_{j=1}^{\infty}$ and $\{f_j\}_{j=1}^\infty$ are orthonormal bases of $\spa H_1$ and $\spa H_2$, then $\{ e_i \otimes f_j \}_{i,j=1}^\infty$ is an orthonormal basis of $\spa H$. This construction of product Hilbert spaces can be generalized to Banach spaces $\spa B_1$ and $\spa B_2$. That is, one can define $\spa B = \spa B_1 \otimes \spa B_2$ in a similar way, the only difference being that the completion is done under the norm $\| x \otimes y \|_{\spa B} := \| x \|_{\spa B_1} \| y \|_{\spa B_2}$, for $x \in \spa B_1$ and $y \in \spa B_2$.

Following \cite{aston2017}, for $\ope A \in \spa S_p(\spa H_1)$ and $\ope B \in \spa S_p(\spa H_2)$, we define $\ope A \ct \ope B$ as the unique operator on $\spa S_p(\spa H_1) \otimes \spa S_p(\spa H_2)$ satisfying
\begin{equation}\label{eq:ct_symbol}
(\ope A \ct \ope B)(x \otimes y) = \ope A x \otimes \ope B y \,, \qquad \forall x \in \spa H_1, y \in \spa H_2 \,.
\end{equation}
By the abstract construction above, we have $\verti{\ope A \ct \ope B}_p = \verti{\ope A}_p \verti{\ope B}_p$. In this paper, we are mostly interested in products of spaces of Hilbert-Schmidt operators: for $\ope A_1,\ope A_2 \in \spa S_2(\spa H_1)$ and $\ope B_1,\ope B_2 \in \spa S_2(\spa H_2)$, we have $\langle \ope A_1 \ct \ope B_1 , \ope A_2 \ct \ope B_2 \rangle_{\spa S_2(\spa H)} = \langle \ope A_1,\ope A_2 \rangle_{\spa S_2(\spa H_1)} \langle \ope B_1,\ope B_2 \rangle_{\spa S_2(\spa H_2)} $.

With these background concepts in place, the definition of separability can be phrased as follows.

\begin{definition}[Separability]\label{def:separability}
Let $\spa H_1$ and $\spa H_2$ be separable Hilbert spaces and $\spa H := \spa H_1 \otimes \spa H_2$. An operator $\ope C \in \spa{S}_p(\spa H)$ is called separable if $\ope C = \ope A \ct \ope B$ for some $\ope A \in \spa{S}_p(\spa H_1)$ and $\ope B \in \spa{S}_p(\spa H_2)$.
\end{definition}

When $\ope A$ and $\ope B$ are integral operators with kernels $a=a(t,t')$ and $b=b(s,s')$, respectively, the kernel of $\ope C = \ope A \ct \ope B$ is given by $c(t,s,t',s') = a(t,t')b(s,s')$.

Let $\ope X$ be a random element of a separable Hilbert space $\spa H$ with $\E \|\ope X \|^2 < \infty$. Then, the mean $m = \E \ope X$ and the covariance $\ope C=\E[ (\ope X - m) \otimes (\ope X-m) ]$ are well defined \mbox{\citep{hsing2015}}. The covariance operator $\ope C \in \spa{S}_2(\spa{H})$ is positive semi-definite. In the case of $\spa H=\spa L_2([0,1]^2)$, the covariance operator is related to the covariance kernel \mbox{$c=c(t,s,t',s')$~via}
\[
(\ope C f)(t,s) = \int_0^1 \int_0^1 c(t,s,t',s') f(t',s') d t' d s' \,.
\]
The kernel $c$ is continuous, for example, if $\ope X = (x(t,s):t,s \in [0,1])$ is a mean-square continuous process with continuous sample paths (the latter is needed for $\ope X$ to be a random element of $\spa L_2([0,1]^2)$). In that case, $c(t,s,t',s') = \E[\fun X(t,s) - \E \fun X(t,s)][\fun X(t',s') - \E \fun X(t',s')]$.

\section{Methodology}\label{sec:methodology}

\subsection{The Separable Expansion}\label{sec:SCD}

The following four spaces are isometrically isomorphic, and thus the covariance $\ope C$ of a random element $\ope X \in \spa H_1 \otimes \spa H_2$ can be regarded as an element of any of these, each of which leads to different perspectives on potential decompositions:
\[
\spa H_1 \otimes \spa H_2 \otimes \spa H_1 \otimes \spa H_2 \simeq \spa{S}_2(\spa H_1 \otimes \spa H_2) \simeq \spa{S}_2(\spa H_2 \otimes \spa H_2, \spa H_1 \otimes \spa H_1) \simeq \spa{S}_2(\spa H_1) \otimes \spa{S}_2(\spa H_2)  .
\]
If we consider $\ope C \in \spa{S}_2(\spa H_1 \otimes \spa H_2)$, we can write its eigendecomposition as
\mbox{$
\ope C = \sum_{j=1}^\infty \lambda_j g_j \otimes g_j
$,}
where $\lambda_1 \geq \lambda_2 \geq \ldots$ are the eigenvalues and $\{g_j\} \subset \spa H_1 \otimes \spa H_2$ are the eigenvectors, forming an orthonormal basis of $\spa H_1 \otimes \spa H_2$.
On the other hand, if we consider $\ope C \in \spa{S}_2(\spa H_2 \otimes \spa H_2, \spa H_1 \otimes \spa H_1)$, we can write its singular value decomposition as
$
\ope C = \sum_{j=1}^\infty \sigma_j e_j \otimes f_j,
$
where $\sigma_1 \geq \sigma_2 \geq  \ldots \geq 0$ are the singular values, and $\{e_j\} \subset \spa H_1 \otimes \spa H_1$ and $\{f_j\} \subset \spa H_2 \otimes \spa H_2$ are the (left and right) singular vectors, forming orthonormal bases of $\spa H_1 \otimes \spa H_1$ and $\spa H_2 \otimes \spa H_2$, respectively. Note that $\ope C$ is not self-adjoint in this case; $\{ e_j \}$ are the eigenvectors of $\ope C \ope C^\ast$, $\{ f_j\}$ are the eigenvectors of $\ope C^\ast \ope C$, and $\{ \sigma_j^2 \}$ are eigenvalues of both $\ope C \ope C^\ast$ and $\ope C^\ast \ope C$.
If we consider $\ope C \in \spa{S}_2(\spa H_1) \otimes \spa{S}_2(\spa H_2)$, the previous singular value decomposition can be re-expressed as
$
\ope C = \sum_{j=1}^\infty \sigma_j \ope A_j \otimes \ope B_j,
$
where $\sigma_1 \geq \sigma_2\geq  \ldots \geq 0$ are the same as before, and $\{\ope A_j\} \subset \spa{S}_2(\spa H_1)$ and $\{\ope B_j\} \subset \spa{S}_2(\spa H_2)$ are isomorphic to $\{e_j\}$ and $\{f_j\}$, respectively. Finally, the singular value decomposition can be written down for (the self-adjoint version of) $\ope C \in \spa{S}_2(\spa H_1 \otimes \spa H_2)$ using the symbol defined in \eqref{eq:ct_symbol} as
\begin{equation}\label{eq:SVD_of_C}
\ope C = \sum_{j=1}^\infty \sigma_j \ope A_j \ct \ope B_j.
\end{equation}

We refer to \eqref{eq:SVD_of_C} as the \emph{separable expansion} of $\ope C$. At the level of kernels, it corresponds to \eqref{eq:separable_decomposition_space_time}. We also refer to the $\sigma_j$'s as the \emph{scores} and $\ope A_j$'s and $\ope B_j$'s to \emph{left} and \emph{right factors}, respectively.

\begin{remark}
The separable expansion generalizes separability on an arbitrary domain in a conceptually similar way to how the nearest Kronecker product \citep{vanloan1993} generalizes to the \emph{sum of Kronecker products} (e.g., the PhD thesis of N. Pitsianis, 1997, Cornell University) or the \emph{Kronecker product singular value decomposition} \cite[Section 6]{vanloan2000} or the \emph{Kronecker sum decomposition} \citep{hero2013a} on the domain of matrices. When working with general Hilbert spaces, however, the Kronecker product and the associated matricizations and vectorizations of multi-dimensional objects (such as the covariance) do not apply and must be circumvented. Moreover, the general framework offers more simplicity and versatility.
\end{remark}

The eigendecomposition and the separable expansion \eqref{eq:SVD_of_C} are two different decompositions of the same element $\ope C \in \spa{S}_2(\spa H_1 \otimes \spa H_2)$. If all but $R \in \N$ scores in \eqref{eq:SVD_of_C} are zero, we say that the \emph{degree-of-separability} of $\ope C$ is $R$ and write $\text{DoS}(\ope C)=R$. In this case, we also say in short that $\ope C$ is $R$-separable. If $\ope C$ does not necessarily have a finite degree-of-separability, but we truncate the series at level $R$ yielding $\ope C_R = \sum_{j=1}^R \sigma_j \ope A_j \ct \ope B_j$ for some $R \in \N$, we call $\ope C_R$ the best \mbox{$R$-separable} approximation of $\ope C$, because
\begin{equation}\label{eq:best_Rsep}
\ope C_R = \argmin_{\ope G} \verti{\ope C - \ope G}_2^2 \quad \textrm{s.t.} \quad \text{DoS}(\ope G) \leq R.
\end{equation}

It may be tempting to find an analogy between the degree-of-separability and the rank of an operator, which equals the number of non-zero eigenvalues. But there is no simple relationship between these two. In particular, $\ope C$ can be of infinite rank even if it is $1$-separable. On the other hand, $\ope C$ has degree $R=1$ if and only if $\ope C$ is separable according to Definition~\ref{def:separability}. In that case, we simply call $\ope C$ separable instead of 1-separable.

We believe that the separable expansion can be used to induce an interesting form of parsimony because, for many operators, the (approximate) degree-of-separability is substantially smaller than the (approximate) rank. The following example provides an illustration.

\begin{figure}[!t]
   \centering
   \begin{tabular}{cc}
   \includegraphics[width=0.45\textwidth]{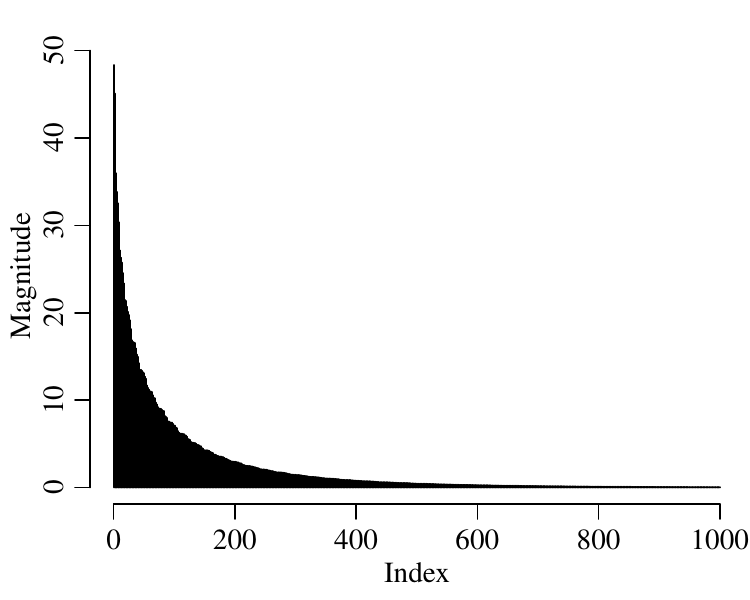} &
   \includegraphics[width=0.45\textwidth]{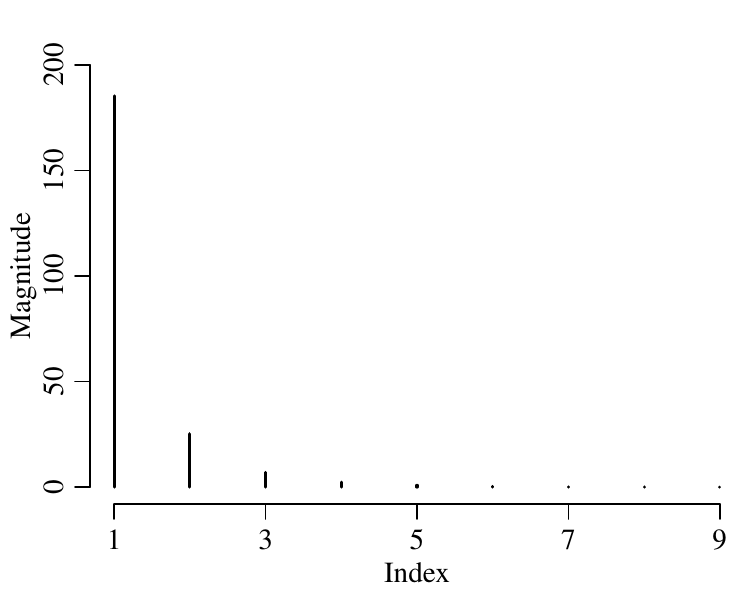}
   \end{tabular}  
   \caption{\mbox{Eigenvalues (left) and separable expansion scores (right) of the covariance from Example~\ref{ex:gneiting}.}}
    \label{fig:Irish_Wind} 
\end{figure}

\begin{example}\label{ex:gneiting}
Consider the Irish Wind data set of \cite{haslett1989}. The data set was modeled with a separable covariance structure at first, before \cite{gneiting2002,gneiting2006} argued that separability has hardly justifiable practical consequences for this data set. Later, separability was formally rejected by a statistical test in \cite{bagchi2020}. A non-separable parametric covariance model was developed specifically for the Irish Wind data in \cite{gneiting2002}. We consider the fitted parametric covariance model as the ground truth $\ope C$ (see Appendix~\ref{sec:setup_for_simulations} for a full specification). We evaluate $\ope C$ on a $50 \times 50$ grid in the domain $[0,20]^2$ to obtain $\mat C \in \R^{50 \times 50}$ and plot the eigenvalues and the separable expansion scores of $\mat C$ (evaluated on the same grid) in Figure~\ref{fig:Irish_Wind}. While $\mat C$ is clearly not low-rank, it is approximately of very low degree-of-separability. We will return to this particular covariance in Section~\ref{sec:simulations} and show that the choices $R=2$ or $R=3$ lead to very good approximations of $\mat C$.
\end{example}

\subsection{Partial Inner Product and Generalized Power Iteration}\label{sec:PIP}

For $x \in \spa H_1$ and $y \in \spa H_2$, the \emph{tensor product operators} $(x \otimes_1 y) : \spa H_1 \to \spa H_2$ and $(x \otimes_2 y) : \spa H_2 \to \spa H_1$ are defined as $(x \otimes_1 y)\, e = \langle x,e \rangle_{\spa H_1} y$ for all $e \in \spa H_1$, and $(x \otimes_2 y) f = \langle y,f \rangle_{\spa H_2} x$ for all $f \in \spa H_2$, respectively \citep[][p. 76]{hsing2015}. The symbols $\otimes_1$ and $\otimes_2$ themselves can be understood as mappings
$
\otimes_1: [\spa H_1 \times \spa H_2] \times \spa H_1 \to \spa H_2$
and
$
\otimes_2: [\spa H_1 \times \spa H_2] \times \spa H_2 \to \spa H_1
$.
We develop the partial inner products $T_1$ and $T_2$ by extending these mappings from the Cartesian product space $\spa H_1 \times \spa H_2$ to the richer outer product space $\spa H_1 \otimes \spa H_2$. This is straightforward in principle, because the finite linear combinations of the elements of $\spa H_1 \times \spa H_2$ are by definition dense in $\spa H_1 \otimes \spa H_2$.

\begin{definition}\label{def:PIP}
Let $\spa H=\spa H_1 \otimes \spa H_2$. The partial inner products are the two unique bi-linear operators $T_1: \spa H \times \spa H_1 \to \spa H_2$ and $T_2: \spa H \times \spa H_2 \to \spa H_1$ defined by
\[
\begin{split}
T_1(x \otimes y, e) &= (x \otimes_1 y) e , \quad x,e \in \spa H_1, \, y \in \spa H_2,\\
T_2(x \otimes y, f) &= (x \otimes_2 y) f , \quad x \in \spa H_1, \, y,f \in \spa H_2.
\end{split}
\]
\end{definition}

The definition includes the claim of uniqueness of the partial inner products, which is proven in the appendix. The partial inner product has an explicit integral representation on any $\spa L_2$ space, given by the following proposition. From the computational perspective, the four-dimensional discrete case is the most important one. Hence we state it as the subsequent corollary.

\begin{proposition}\label{prop:integral_rep_PIP}
Let $\spa H_1=\spa{L}_2(E_1,\spa E_1, \mu_1)$, $\spa H_2=\spa{L}_2(E_2,\spa E_2, \mu_2)$, and $g \in \spa H_1 \otimes \spa H_2$, $v \in \spa H_2$. If we denote $u=T_2(g,v)$, then
\begin{equation}\label{eq:integral_representation}
   u(t) = \int_{E_2} g(t,s) v(s) d \mu_2(s) . 
\end{equation}
\end{proposition}

\begin{corollary}\label{lem:discrete_PIP}
Let $K_1, K_2 \in \N$, $\spa H_1 = \R^{K_1 \times K_1}$, $\spa H_2 = \R^{K_2 \times K_2}$. Let $\mat G \in \spa H_1 \otimes \spa H_2$ and $\mat V \in \spa H_2$. If we denote $\mat U = T_2(\mat G,\mat V)$, then
\begin{equation}\label{eq:discrete_PIP}
\mat U[i,j] = \sum_{k=1}^{K_2} \sum_{l=1}^{K_2} \mat G[i,j,k,l] \mat V[k,l] , \quad \forall i,j = 1, \ldots, K_1 .
\end{equation}
\end{corollary}

\begin{remark}
Definition~\ref{def:PIP} is more general than the respective definitions provided by \cite{bagchi2020,dette2020}. Still, the partial inner product can be defined for product spaces of more than two Hilbert spaces. See Appendix~\ref{appendix:PIP_general}  for more details.
\end{remark}

The following lemma explores some basic properties of the partial inner product.

\begin{lemma}\label{lem:PIP_basic_properties}
Let $\ope C \in \spa{S}_2(\spa H_1 \otimes \spa H_2)$, $\ope W_1 \in \spa{S}_2(\spa H_1)$ and $\ope W_2 \in \spa{S}_2(\spa H_2)$. Then,
\begin{enumerate}
    \item If $\ope C$ is separable, i.e.\ $\ope C = \ope A \ct \ope B$, then $T_2(\ope C,\ope W_2) = \langle \ope B, \ope W_2 \rangle \ope A$ and $T_1(C,W_1) = \langle \ope A, \ope W_1 \rangle \ope B$.
    \item If $\ope C, \ope W_1$ and $\ope W_2$ are positive semi-definite, then $T_2(\ope C,\ope W_2)$ and $T_1(\ope C,\ope W_1)$ are also positive semi-definite.
\end{enumerate}
\end{lemma}

The first part of Lemma~\ref{lem:PIP_basic_properties} exemplifies why the operators $T_1$ and $T_2$ are called partial inner products. If $\ope C \in \spa S_2(\spa H_1 \otimes \spa H_2)$ is not exactly separable, the partial inner product is at the core of the algorithm for finding a separable proxy to it. The necessity to choose (correctly scaled) weights $\ope W_1$ and $\ope W_2$ is bypassed via an iterative procedure.

\begin{proposition}\label{prop:convergence}
For $\ope C \in \spa H_1 \otimes \spa H_2$, let $\ope C=\sum_{j=1}^\infty \sigma_j \ope A_j \otimes \ope B_j$ be a decomposition such that $|\sigma_1| > |\sigma_2| \geq |\sigma_3| \geq \ldots$, $\{\ope A_j\}$ is an orthonormal basis of $\spa H_1$ and $\{\ope B_j\}$ is an orthonormal basis of $\spa H_2$.
Let $\ope V^{(0)} \in \spa H_2$ be such that $\langle \ope B_1, \ope V^{(0)} \rangle_{\spa H_2} > 0$. Then the sequences $\{ \ope U^{(k)} \}$ and $\{ \ope V^{(k)} \}$ formed via the recurrence relation
\[
\ope U^{(k+1)} = \frac{T_2(\ope C,\ope V^{(k)})}{\| T_2(\ope C,\ope V^{(k)}) \|} , \quad \ope V^{(k+1)} = \frac{T_1(\ope C,\ope U^{(k+1)})}{\| T_1(\ope C,\ope U^{(k+1)}) \|}
\]
converge to $\ope A_1$ and $\ope B_1$, respectively. The convergence speed is linear with the rate given by the spacing between $\sigma_1$ and $\sigma_2$.
\end{proposition}

The assumption $\langle \ope B_1, \ope V^{(0)} \rangle_{\spa H_2} > 0$ in the previous proposition can be weakened to $\langle \ope B_1, \ope V^{(0)} \rangle_{\spa H_2} \neq 0$. In that case, the sequences do not necessarily converge, but all limit points span the appropriate spaces. The proof is similar, except for some care taken at the level of the signs. The sign ambiguity is caused by the fact that the factors $\ope A_j$ and $\ope B_j$ are (even in the case of non-zero spacing between the scores $\{ \sigma_j \}$) unique only up to the sign.

\begin{example}
Let $\ope C$ be the covariance of $\ope X \in \spa{H}_1^\prime \otimes \spa{H}_2^\prime$. Choosing $\spa H_1 = \spa S_2(\spa H_1^\prime)$ and \mbox{$\spa H_2 = \spa S_2(\spa H_2^\prime)$}, the previous proposition shows that the best separable approximation of the covariance, i.e.\ a solution to \eqref{eq:best_Rsep} with $R=1$, can be found via the power iteration method, consisting of a series of partial inner products. Choosing instead $\spa H_1 = \spa H_2 = \spa H_1'\otimes \spa H_2'$, the singular value decomposition in Proposition~\ref{prop:convergence} becomes the eigendecomposition. The previous proposition shows that the leading eigenvalue-eigenvector pair can also be found via the power iteration method.
\end{example}

In the present paper, we expand our attention beyond separability, to the solution of \eqref{eq:best_Rsep} with $R \in \N$, i.e.\ searching for the best $R$-separable approximation.
This optimisation problem can be solved via Algorithm~\ref{alg1}, which is a
generalization of the power iteration method \citep[for finding the singular value decomposition, cf.][]{vanloan1983} to general Hilbert spaces. It can be used to find the best separable approximation and also the best $R$-separable approximation of a~covariance operator $\ope C$, as shown in Algorithm~\ref{alg1}.

\begin{algorithm}[b]
\caption{Power iteration method for separable expansion on a general Hilbert space.}\label{alg1}
\begin{description}
\item[Input] $\ope C \in \spa{S}_2(\spa H_1 \otimes \spa H_2)$, initial guesses $\ope A_1,\ldots,\ope A_R \in \spa H_1 \otimes \spa H_1$
\item[for] r=1,\ldots,R

$\widetilde{\ope C} = \ope C - \sum_{j=1}^{r-1} \sigma_j \ope A_j \ct \ope B_j$
\begin{description}
\item[repeat]\mbox{}
  \begin{description}
  \item[] $ \ope B_r = T_1(\widetilde{\ope C}, \ope A_r)$
  \item[] $ \ope B_r = \ope B_r/\|\ope B_r \|$
  \item[] $ \ope A_r = T_2(\widetilde{\ope C},\ope B_r)$
  \item[] $\sigma_r \,= \| \ope A_r \|$
  \item[] $\ope A_r = \ope A_r/\sigma_r$
  \end{description}
\item[until convergence]
\end{description}
\item[end for]
\item[Output] $\sigma_1,\ldots,\sigma_R$, $\ope A_1,\ldots,\ope A_R$, $\ope B_1,\ldots \ope B_R$
\end{description}
\end{algorithm}

\subsection{Estimation} \label{sec:estimation}

Let $\ope X_1, \ldots, \ope X_N$ be i.i.d.\ elements in  $\spa H_1 \otimes \spa H_2$, and $R \in \mathbb{N}$ be given (the choice of $R$ will be discussed in Section~\ref{sec:rank}). We propose the following estimator for the covariance operator:
\[
\widehat{\ope C}_{R,N} = \argmin_{\ope G} \verti{\widehat{\ope C}_N - \ope G}_2^2 \quad \textrm{s.t.} \quad \text{DoS}(\ope G) \leq R,
\]
where $\widehat{\ope C}_N = N^{-1} \sum_{n=1}^N (\ope X_n - \bar{\ope X}) \otimes (\ope X_n - \bar{\ope X})$ is the empirical covariance. The estimator $\widehat{\ope C}_{R,N}$ is the best $R$-separable approximation to the empirical covariance $\widehat{\ope C}_N$, and has the form
\begin{equation}\label{eq:estimator}
\widehat{\ope C}_{R,N} = \sum_{r=1}^R \widehat{\sigma}_r \widehat{\ope A}_r \ct \widehat{\ope B}_r,
\end{equation}
where $\widehat \sigma_1,\ldots,\widehat \sigma_R$, $\widehat{\ope A}_1,\ldots,\widehat{\ope A}_R$, $\widehat {\ope B}_1,\ldots \widehat{\ope B}_R$ are the output of Algorithm~\ref{alg1} applied to the empirical covariance estimator $\widehat{\ope C}_N$ as the input.

The key observation here is that the partial inner product operations required to construct \eqref{eq:estimator} by means of Algorithm~\ref{alg1} can be carried out directly at the level of the data, without the need to explicitly form or store the empirical covariance estimator $\widehat{\ope C}_N$. In the discrete case, for example, the data arrive as matrices $\mat X_1, \ldots, \mat X_N \in \R^{K_1 \times K_2}$. Then the covariance and the estimator \eqref{eq:estimator} are also discrete, namely $\mat C, \widehat{\mat C}_{R,N} \in \R^{K_1 \times K_2 \times K_1 \times K_2}$. In the theoretical development of Section~\ref{sec:analysis}, we differentiate between multivariate data and functional data observed on a dense grid (of size $K_1 \times K_2$), the latter being our primary interest. But for now, the distinction is immaterial: we assume we are in the discrete case to exemplify the computational benefits of the partial inner product. It is straightforward to verify from Corollary~\ref{lem:discrete_PIP} that
\begin{equation}\label{eq:PIP_on_data}
\begin{split}
T_1(\widehat{\mat C}_N,\mat A) &= \frac{1}{N}\sum_{n=1}^N T_1\Big( (\mat X_n - \bar{\mat X}) \otimes (\mat X_n - \bar{\mat X}),\mat A\Big) = \frac{1}{N}\sum_{n=1}^N \t{(\mat X_n - \bar{\mat X})} \mat A (\mat X_n - \bar{\mat X}),\\
T_2(\widehat{\mat C}_N,\mat B) &= \frac{1}{N}\sum_{n=1}^N T_2\Big((\mat X_n - \bar{\mat X}) \otimes (\mat X_n - \bar{\mat X}),\mat B\Big) = \frac{1}{N}\sum_{n=1}^N (\mat X_n - \bar{\mat X}) \mat B \t{(\mat X_n - \bar{\mat X})}.
\end{split}
\end{equation}
Utilizing these formulas together with Lemma~\ref{lem:PIP_basic_properties}, one obtains Algorithm~\ref{alg2}, working efficiently at the level of data.

The immense popularity of the separability assumption stems from the computational savings it entails. These are captured in Table~\ref{tab:complexities}, assuming $K_1 = K_2 = K$ for simplicity. Notice the reduced estimation complexity of the separable model. Since the convergence rate of the power iteration method is linear, and a single iteration can be evaluated efficiently at the level of data due to \eqref{eq:PIP_on_data}, our approach can also be used to estimate the $R$-separable model with the same efficiency (times $R$, of course). Moreover, we will show in the next section that it is possible to apply and (numerically) invert an $R$-separable covariance \eqref{eq:estimator} with the same computational costs as for a separable model. Hence the complexities reported in Table~\ref{tab:complexities} for a separable model are also valid for an $R$-separable approximation.

\begin{algorithm}[b]
\caption{Constructing the $R$-separable estimator from discrete measurements.}\label{alg2}
\begin{description}
\item[Input] $\mat X_1,\ldots,X_N \in \R^{K_1 \times K_2}$, initial guesses $A_1,\ldots,A_R \in \R^{K_1 \times K_2}$
\item[for] r=1,\ldots,R
\begin{description}
\item[repeat]\mbox{}
  \begin{description}
  \item[] $B_r = \frac{1}{N}\sum_{n=1}^N \t{(\mat X_n - \bar{\mat X})} \mat A_r (\mat X_n - \bar{\mat X}) - \sum_{j=1}^{r-1} \sigma_j \langle \mat A_r, \mat A_j\rangle \mat A_j$
  \item[] $B_r = B_r/\|B_r \|$
  \item[] $A_r = \frac{1}{N}\sum_{n=1}^N \t{(\mat X_n - \bar{\mat X})} \mat B_r (\mat X_n - \bar{\mat X}) - \sum_{j=1}^{r-1} \sigma_j \langle \mat B_r, \mat B_j\rangle \mat B_j$
  \item[] $\sigma_r \,= \|A_r \|$
  \item[] $A_r = A_r/\sigma_r$
  \end{description}
\item[until convergence]
\end{description}
\item[end for]
\item[Output] $\sigma_1,\ldots,\sigma_R$, $A_1,\ldots,A_R$, $B_1,\ldots B_R$
\end{description}
\end{algorithm}

\begin{remark}
To the best of our knowledge, the only approach to efficient estimation of a separable model, which reduces the estimation complexity to that reported in Table~\ref{tab:complexities}, is partial tracing of \cite{aston2017}. Partial tracing achieves this complexity by considering only some of the available raw covariances (i.e.\ the cross-products of two sampled entries on the same surface). In contrast, our approach uses all the available raw covariances, composing the empirical estimator. The computational savings are achieved by reducing the number of necessary operations via the formulas in \eqref{eq:PIP_on_data}. Moreover, our estimation procedure facilitates the search for an approximation (either $R$-separable or entirely separable), while the partial tracing estimator lacks any optimality properties, and assumes separability as a model. This can be said also for other approaches built upon partial tracing, such as the weakly separable model \citep{lynch2018}.
\end{remark}

\begin{table}[!t]
    \centering
    \caption{Time and memory complexities of computing the empirical covariance estimator and a separable estimator when $N$ surfaces are observed discretely on a $K \times K$ grid.}
    \label{tab:complexities}
    \begin{tabular}{ccccccc}
    \toprule
      && Memory complexity && & Time complexity&  \\
      \cmidrule(lr){4-7}
      Method && && Estimation & Application & Inversion \\
      \midrule
      empirical && $\mathcal{O}(K^4)$ && $\mathcal{O}(N K^4)$ & $\mathcal{O}(K^4)$ & $\mathcal{O}(K^6 )$  \\
      separable && $\mathcal{O}(K^2)$ && $\mathcal{O}(N K^3)$ & $\mathcal{O}(K^3)$ & $\mathcal{O}(K^3)$   \\
     \bottomrule
    \end{tabular}
\end{table}

\subsection{Inversion and Prediction}\label{sec:inversion}

Here, we consider the inversion of the estimated covariance. In particular, we are interested in solving a linear system
\begin{equation}\label{eq:system}
\widehat{\mat C}_{R,N} \mat X = \mat Y,
\end{equation}
where $\widehat{\mat C}_{R,N} \in \R^{K \times K \times K \times K}$ is our $R$-separable estimator of equation \eqref{eq:estimator}. The inversion is needed for the purposes of prediction or kriging, among other tasks. We develop an efficient iterative procedure to solve \eqref{eq:system}. The crucial observation is that  $\widehat{\mat C}_{R,N}$ can be applied in $\O(K^3)$ operations:
\begin{equation}\label{eq:fast_application}
\widehat{\mat C}_{R,N} \mat  X = \sum_{r=1}^R \widehat{\sigma}_r \, (\widehat{\mat A}_r \ct \widehat{\mat B}_r) \mat X = \sum_{r=1}^R \widehat{\sigma}_r \widehat{\mat A}_r \mat  X \widehat{\mat B}_r .
\end{equation}
The previous equation follows from the fact that $(\mat A \ct \mat B) \mat X = \mat A \mat X \mat B$ for $\mat B$ self-adjoint, which can be verified directly from the definitions. Due to \eqref{eq:fast_application}, we can rewrite the linear system \eqref{eq:system} as
$
\widehat{\sigma}_1 \widehat{\mat A}_1 \mat X \widehat{\mat B}_1 + \ldots + \widehat{\sigma}_R \widehat{\mat A}_R \mat X \widehat{\mat B}_R = \mat Y
$, which is a \emph{linear matrix equation}. Even though there exist provably convergent specialized solvers for equations of this type, the simple preconditioned conjugate gradient \citep{vanloan1983} is the method of choice when $\widehat{\mat C}_{R,N}$ is symmetric and positive semi-definite. In our case, $\mat P = \widehat{\sigma}_1 \widehat{\mat A}_1 \ct \widehat{\mat B}_1$ is a natural preconditioner, whose square-root $\mat V$ can be both obtained and applied easily \citep{vanloan1983}.
The more dominant the leading term $\widehat{\sigma}_1 \widehat{\mat A}_1 \ct \widehat{\mat B}_1$ is in $\widehat{\mat C}_{R,N}$, the fewer conjugate gradient iterations will be needed. This is a manifestation of a certain statistical-computational trade-off. The $R$-separable model can in theory fit any covariance $C$, when $R$ is taken large enough. However, the more dominant the leading (separable) term is, the better computational properties we have, see Appendix~\ref{sec:further_simulations}.

It remains to address the existence of a solution to the linear system \eqref{eq:system}. Note that we cannot guarantee that $\widehat{\mat C}_{R,N}$ is positive semi-definite. Lemma~\ref{lem:PIP_basic_properties} says that $\widehat{\mat A}_{1}$ and $\widehat{\mat B}_{1}$ are positive semi-definite, and they will typically be positive definite for $N$ sufficiently large (when $\widehat{\mat C}_N$ is positive definite). But $\widehat{\mat C}_{N} - \widehat{\sigma}_1 \widehat{\mat A}_{1} \ct \widehat{\mat B}_{1}$ is necessarily indefinite, and hence we cannot say anything about the remaining terms.
However, $\widehat{\mat C}_{R,N}$ is a consistent estimator of the true $\mat C$ for sufficiently large values of $R$ (see Section~\ref{sec:analysis} for a discussion on the rate of convergence of this estimator depending on $R$). So, for a large enough sample size and an appropriate degree $R$, $\widehat{\mat C}_{R,N}$ cannot be far from being positive semi-definite. To eliminate practical anomalies, we will \emph{positivize} the estimator.

Due to \eqref{eq:fast_application}, the power iteration method can be used to find the leading eigenvalue $\lambda_{\text{max}}$ of $\widehat{\mat C}_{R,N}$ in $\O(K^3)$ operations. We can then find the smallest eigenvalue $\lambda_{\text{min}}$ of $\widehat{\mat C}_{R,N}$ by applying the power iteration method to $\lambda_{\text{max}} \mat I - \widehat{\mat C}_{R,N}$, where $\mat I \in \R^{K_1 \times K_2 \times K_1 \times K_2}$ is the identity. Subsequently, if $\lambda_{\text{min}}<0$, we can perturb $\widehat{\mat C}_{R,N}$ to obtain its positive semi-definite version:
$
\widehat{\mat C}_{R,N}^+ = \widehat{\mat C}_{R,N} + (\epsilon - \lambda_{\text{min}}) \mat I
$,
where $\epsilon \geq 0$ is a potential further regularization.

The positivized estimator $\widehat{\mat C}_{R,N}^+$ is $(R+1)$-separable, so it can still be approached in the same spirit, with one exception. If the inverse problem is ill-conditioned and regularization is used, the preconditioner discussed above is no longer effective, since $\mat A_1$ or $\mat B_1$ may not be invertible. In this case, we use the preconditioner $\mat P = \widehat{\sigma}_1 \widehat{\mat A}_1 \ct \widehat{\mat B}_1 + (\epsilon - \lambda_{\text{min}}) \mat I$. This is a preconditioning via the discrete Stein equation, see \cite{masak2019}. The effectiveness of the proposed inversion algorithm is demonstrated in Appendix~\ref{sec:further_simulations}.

\begin{remark}\label{rem:regularization}
We stress here that the potential need to regularize an estimator of a low degree-of-separability arises in a very different way compared to the necessity to regularize a (more typical) low-rank estimator. When the truncated eigendecomposition is used as an estimator, the spectrum is by construction singular and regularization is thus necessary for the purposes of prediction. Contrarily, when an estimator of low degree-of-separability is used, the spectrum of the estimator mimics that of $C$ more closely. If $C$ itself is well-conditioned, there may be no need to regularize, regardless of what degree $R$ is used as a cut-off. However, in the case of functional data observed on a dense grid, regularization may be necessary due to the spectral decay of $C$~itself.
\end{remark}

As an important application, we use the $R$-separable estimator \eqref{eq:estimator} to predict the missing values of a datum $\mat X \in \R^{K_1 \times K_2}$. In the case of a random vector $\mat x$, such that
\[
\mat x = \t{\big(\mat x_{\rm mis}^\top, \mat x_{\rm obs}^\top \big)} \sim \left(0,
\begin{pmatrix}
\Sigma_{11} & \Sigma_{12} \\
\Sigma_{12} & \Sigma_{22}
\end{pmatrix}
\right),
\]
where $\mat x_{\rm mis}$ is missing and $\mat x_{\rm obs}$ is observed, the best linear unbiased predictor of $\mat x_{\rm mis}$ given $\mat x_{\rm obs}$ is $
\widehat{\mat x}_{\rm mis} = \Sigma_{12} \Sigma_{22}^{-1} \mat x_{\rm obs}
$.
The goal here is to show that this predictor is calculable within the set computational limits, which prevents us from naively vectorizing $\mat X \in \R^{K_1 \times K_2}$, as above, and using the matricization of $\widehat{\mat C}_{R,N}$ in place~of~$\Sigma$.

Assume initially that $\mat X$ is observed up to whole columns indexed by the set $I$ and whole rows indexed by the set $J$. We can assume that $I = \{1,\ldots, m_1\}$ and $J = \{1,\ldots, m_2\}$ (otherwise we can permute the rows and columns to make it so). Denote the observed submatrix of $\mat X$ as $\mat X_\text{obs}$. If the covariance of $\mat X$ is $R$-separable, so is the covariance of $\mat X_\text{obs}$, specifically
$
\cov(\mat X_\text{obs}) = \sum_{r=1}^R \sigma_r \mat A_{r,22} \ct \mat B_{r,22}
$,
where $\mat A_{r,22}$ (resp. $\mat B_{r,22}$) are the bottom-right sub-matrices of $\mat A_r$ (resp. $\mat B_r$) of appropriate dimensions. Hence the inversion algorithm discussed above can be used to efficiently calculate $\Sigma_{22}^{-1} \mat x_{\rm obs}$, where $\mat x_{\rm obs}$ is the vectorization of $\mat X_\text{obs}$. It remains to apply the cross-covariance $\Sigma_{12}$ to this element. Since $\mat z = \Sigma_{12} \mat y$ can be calculated as
\begin{equation}\label{eq:trick_zeros}
\begin{pmatrix}
\mat z \\ \star
\end{pmatrix}
= \begin{pmatrix}
\Sigma_{11} & \Sigma_{12} \\
\Sigma_{12} & \Sigma_{22}
\end{pmatrix}
\begin{pmatrix}
0 \\ \mat y
\end{pmatrix},
\end{equation}
we can apply the entire $\widehat{\mat C}_{R,N}$ to $\Sigma_{22}^{-1} \mat x_{\rm obs}$ enlarged  to the appropriate dimensions by the suitable adjunction of zeros.

If an arbitrary pattern $\Omega$ in $\mat X$ is missing (i.e.\ $\Omega$ is a bivariate index set), we make use of the previous trick also when calculating $\Sigma_{22}^{-1} \mat x_{\rm obs}$. The preconditioned conjugate gradient algorithm discussed above only requires a fast application of $\Sigma_{22}$. This can be achieved by applying the whole $\widehat{\mat C}_{R,N}$ to $\widetilde{\mat X}$, with entries $\widetilde{\mat X}[i,j] = \mat X[i,j] \1\{(i,j) \in \Omega\}$.
The same trick can be used to apply the cross-covariance to the solution of the inverse problem. Hence the predictor can be calculated efficiently for an arbitrary missing pattern in $\mat X$.

\subsection{Degree-of-separability Selection}\label{sec:rank}

Given $R\in \N$, we have demonstrated how to construct and invert an $R$-separable proxy of the covariance. It now remains to discuss how to choose the degree-of-separability $R$. Recall that we do not assume that the covariance in question is $R$-separable, so there is no a-priori ``correct" choice of $R$. In this context, $R$ can be seen as governing the effective number of parameters being estimated from the data (the ``degrees of freedom", see Remark~\ref{remark:choice_of_R}), serving in effect as a~regularization parameter. So, one can
seek a positive integer $R$ that minimizes the mean squared error
\begin{equation}\label{eq:mean_squared_error}
\E \verti{\widehat{\mat C}_{R,N} - \mat C}_2^2.
\end{equation}
The underlying bias-variance trade-off is precisely the reason why small values of $R$ can lead to improved mean squared error compared to the empirical covariance estimator. %Intuitively, with increasing $R$, the estimator $\widehat{C}_{R,N}$ has increasing number of degrees of freedom, and more observations are needed to estimate it reliably (see Remark~\ref{remark:choice_of_R})

To determine an empirical surrogate of \eqref{eq:mean_squared_error} and devise a cross-validation strategy, note first that
\[
\E \verti{\widehat{\mat C}_{R,N} - \mat C}_2^2 = \E \verti{\widehat{\mat C}_{R,N}}_2^2 - 2 \E\langle \widehat{\mat C}_{R,N} , \mat C \rangle + \verti{\mat C}_2^2.
\]
Of the three terms on the right-hand side of the previous equation, the first term is estimated straightforwardly by dropping the expectation, while the final term does not depend on $R$ and hence it does not affect the minimization. Hence, it remains to estimate the middle term. Denote by $\widehat{\mat C}_{R,N-1}^{(j)}$ the $R$-separable estimator constructed excluding the $j$-th datum $\mat X_j$. Due to independence between samples, we have
\[
\E \langle \mat X_j, \widehat{\mat C}_{R,N-1}^{(j)} \mat X_j \rangle = \E \langle \widehat{\mat C}_{R,N-1}^{(j)}, \mat X_j \otimes \mat X_j \rangle = \langle \E \widehat{\mat C}_{R,N-1}^{(j)}, \E \mat X_j \otimes \mat X_j \rangle \approx \E \langle \widehat{\mat C}_{R,N} , C\rangle ,
\]
and hence the quantity $N^{-1} \sum_{j=1}^N \langle \mat X_j, \widehat{\mat C}_{R,N-1}^{(j)} \mat X_j \rangle$ approximates $\E\langle \widehat{\mat C}_{R,N} , \mat C \rangle$. In summary, the strategy is to choose $R$ as
\begin{equation}\label{eq:CV_objective}
\argmin_{R} \quad \verti{\widehat{\mat C}_{R,N}}_2^2 - \frac{2}{N} \sum_{j=1}^N \langle \mat X_j, \widehat{\mat C}_{R,N-1}^{(j)} \mat X_j \rangle.
\end{equation}
This procedure corresponds to leave-one-out cross-validation. In practice, we perform e.g.\ a~10-fold cross-validation. The objective of \eqref{eq:CV_objective} can be evaluated efficiently, since $\vertj \widehat{\mat C}_{R,N}\vertj_2^2 = \sum_{r=1}^R \widehat{\sigma}^2$, and $\widehat{\mat C}_{R,N-1}^{(j)} \mat X_j$ is calculated via \eqref{eq:fast_application}. Hence the proposed procedure is within our computational limits.

The choice \eqref{eq:CV_objective} is analogous to the classical cross-validated bandwidth selection scheme in kernel density estimation \citep{wand1994}. The typical scheme for principal component analysis is based on finding a low-dimensional plane fitting the data cloud well -- the low-dimensional plane being tied to the principal components \citep[e.g.][]{Jolliffe:1986}. Such a scheme is not applicable for $R$, because it degenerates: the first term in the separable expansion alone may very well span the whole ambient space (cf. Remark~\ref{rem:regularization}), thus projecting a datum on a subspace generated by a varying number of separable terms will not be informative. On the other hand, one can adopt some well-known semi-automated procedures such as the scree plot or percentage of variance explained \citep{Jolliffe:1986}, since the separable expansion scores allow for this interpretation.
%and the total variance can be calculated efficiently as $\vertj \widehat{\mat C}_{N}\vertj_2^2 = \frac{1}{N^2} \sum_{n=1}^N \sum_{j=1}^N \langle X_n , X_j \rangle^2$

\section{Asymptotic Theory}\label{sec:analysis}
%Here, we establish the consistency of our estimator and derive its rate of convergence. We separately consider two cases, the case of fully observed data and the case of discretely observed data.

\subsection{Fully Observed Functional Data}

In the fully observed case, our sample consists of i.i.d.\ observations $\ope X_1,\ldots,\ope X_N \sim \ope X$, where $\ope X$ is a random element on $\spa H = \spa H_1 \otimes \spa H_2$. Recall that our estimator $\Chat_{R,N}$ given in \eqref{eq:estimator} is the best $R$-separable approximation of the sample covariance operator $\Chat_N$. %= N^{-1}\sum_{n=1}^N (X_n - \overline{X}) \otimes (X_n - \overline{X})$, where $\overline{X} = N^{-1}\sum_{n=1}^N X_n$ is the sample mean. 
Under the assumption of finite fourth moment, we get the following rate of convergence for our estimator.
\begin{theorem}\label{thm:fully_observed}
Let $\ope X_1,\ldots,\ope X_N \sim \ope X$ be a collection of i.i.d.\ random elements of \mbox{$\spa H = \spa H_1 \otimes \spa H_2$} with $\E(\|\ope X\|^4) < \infty$. Further assume that the covariance of $\ope X$ has separable expansion $\ope C = \sum_{i=1}^\infty \sigma_i \ope A_i \ct \ope B_i$, with $\sigma_1 > \cdots > \sigma_R > \sigma_{R+1} \ge \cdots \ge 0$. Define \mbox{$\alpha_i = \min\{\sigma_{i-1}^2 - \sigma_i^2, \sigma_i^2 - \sigma_{i+1}^2\}$} for $i=1,\ldots,R$, and let $a_R = \verti{\ope C}_2 \sum_{i=1}^R (\sigma_i/\alpha_i)$. Then,
\begin{equation}\label{eq:bias_variance_split}
\verti{\Chat_{R,N} - \ope C}_{2} \le \left(\,\sum_{i=R+1}^\infty \sigma_i^2\right)^{1/2} + {\cal O}_{\Prob}\bigg(\frac{a_R}{\sqrt{N}}\bigg).
\end{equation}
\end{theorem}

The first term on the right-hand side of \eqref{eq:bias_variance_split} can be viewed as the {\em bias} of our estimator, which appears because we estimate an $R$-separable approximation of a general $\ope C$. Since $\ope C$ is a Hilbert-Schmidt operator, this term converges to $0$ as $R$ increases. If $\ope C$ is actually $R$-separable then this vanishes. The second term signifies the estimation error of the $R$-separable approximation. Similarly to principal component analysis \citep[see][Section~5.2]{hsing2015}, the error depends on the {\em spectral gap} $\alpha_i$ of the covariance operator. In our case, the spectral gap is in terms of the sequence of squared scores $(\sigma_i^2)_{i \ge 1}$.

The derived rate clearly shows the bias-variance trade-off. While the bias term is a decreasing function of $R$, generally the variance term (governed by $a_R$) is an increasing function of $R$. This emphasizes the need to choose an appropriate $R$ in practice. The actual trade-off depends on the decay of the separable expansion scores. In particular, if the scores decay slowly then we can estimate a relatively large number of terms in the separable expansion, but the estimation error will be high. On the other hand, when we have a fast decay in the scores, we can estimate a rather small number of terms, but with much better precision (see Appendix~\ref{subsec:optimal_R_convex}). In practice, we expect only a few scores to be significant and $\ope C$ to have a relatively low degree-of-separability. The theorem shows that in such situations, our estimator enjoys a convergence rate of $\O_{\Prob}(N^{-1/2})$, which is the same as that of the empirical estimator.

\begin{remark}\label{remark:choice_of_R}
The theorem also gives us convergence rates when we allow $R = R_N$ to increase as a function of $N$. For that we need to assume some structure on the decay of the separable expansion scores. For instance, if we assume that the scores satisfy a convexity condition \mbox{\citep{jirak2016}}, then it turns out that $a_R = {\cal O}(R^2/\sigma_R)$, while $\sum_{i > R} \sigma_i^2 = {\cal O}(R^{-1/2})$. In that case, for consistency, we need $R_N \to \infty$ as $N \to \infty$ with $\sigma_{R_N}^{-1}R_N^2 = {\scriptstyle \mathcal{O}}(N^{1/2})$. Also, the optimal rate is achieved upon choosing $R_N$ such that $\sigma_{R_N}^{-4}R_N^3 \asymp N$ (c.f.\ Appendix~\ref{subsec:optimal_R_convex}). Thus, the {\em admissible} and {\em optimal} rates of $R$ depend on the decay rate of the scores. In particular, if the scores exhibit an exponential decay, i.e.\ $\sigma_R \sim R^{-\tau}$ with $\tau > 1$, then for consistency we need $R_N = {\scriptstyle \mathcal O}(N^{1/(2\tau+4)})$, with the optimal rate being $R_N \asymp N^{1/(4\tau+3)}$, leading to 
\[
\verti{\Chat_{R,N} - \ope C}_{2} = {\cal O}_{\Prob}\Big(N^{-\frac{2\tau-1}{4\tau+3}}\Big).
\]
While the optimal rate for $R_N$ is a decreasing function of $\tau$, the rate of convergence of the error is an increasing function. Thus, our estimator is expected to have better performance when the scores have a fast decay (i.e.\ $\tau$ is large), which means that the true covariance is nearly separable.

Under polynomial decay of the scores, i.e.\ $\sigma_R \sim R^\tau \rho^{-R}$ with $0<\rho<1, \tau \in \mathbb{R}$, consistency is achieved when $R_N^{2-\tau}\rho^{R_N} = {\scriptstyle \mathcal O}(N^{1/2})$, while the optimal rate is obtained by solving $R_N^{3-4\tau}\rho^{4R_N} \asymp N$. So, in this case, we cannot hope to estimate more than $\log(N)$ terms in the separable expansion. These corroborate our previous claim that we can only expect to consistently estimate a small number of terms in the separable expansion.
\end{remark}

Our derived rates are genuinely nonparametric, in the sense that we have not assumed any structure whatsoever on the true covariance. We have only assumed that $\ope X$ has finite fourth moment, which is standard for covariance estimation. We can further relax that condition if we assume that $\ope C$ is actually $R$-separable. From the proof of the theorem, it is easy to see that if we assume $\rm{DoS}(\ope C) \le R$, then our estimator is consistent under the very mild condition of $\E(\|\ope X\|^2) < \infty$.

\subsection{Data Observed on a Regular Grid}

In practice, the data are observed discretely with possible noise contamination. We now develop asymptotic theory in that context. Specifically, assume that $\ope X = \{\fun X(t,s): t \in {\set T}, s \in {\set S}\}$ is a random surface taking values in $\spa H = \spa L_2(\set T \times \set S)$, where $\set T$ and $\set S$ are compact sets. To simplify notation, we assume without loss of generality that ${\set T} = {\set S} = [0,1]$. We observe the data at $K_1 \times K_2$ regular grid points or \emph{pixels}. Let $\{T_1^{K_1},\ldots,T_{K_1}^{K_1}\}$ and $\{S_1^{K_2},\ldots,S_{K_2}^{K_2}\}$ denote regular partitions of $[0,1]$ of lengths $1/K_1$ and $1/K_2$, respectively. We denote by $I_{i,j}^K = T_i^{K_1} \times S_j^{K_2}$ the $(i,j)$-th pixel for $i=1,\ldots,K_1,j=1,\ldots,K_2$. The pixels are non-overlapping (i.e., $I_{i,j}^K \cap I_{i^\prime,j^\prime}^K = \emptyset$ for $(i,j) \ne (i^\prime,j^\prime)$) and have the same volume $|I_{i,j}^K| = 1/(K_1K_2)$. For each surface $\ope X_n,n=1,\ldots,N$, we make one measurement at each of the pixels. The measurements at $I_{i,j}^K$ are assumed to be of the form
\begin{equation}\label{eq:measurement_with_noise_grid}
    Y_n^K[i,j] = X_n^K[i,j] + E_n^K[i,j], \qquad n=1,\ldots,N,i=1,\ldots,K_1,j=1,\ldots,K_2.
\end{equation}
Here, $\mat X_n^K = (X_n^K[i,j])$ is a discretized version of the surface $\ope X_n$ and $\mat E_n^K = (E_n^K[i,j])$ is measurement error or noise. The measurements for the $n$-th surface can be represented by the $K_1 \times K_2$ matrix $\mat Y_n^K = \mat X_n^K + \mat E_n^K$. We will consider two different schemes, which relate the surfaces $\ope X_n$ to their discrete versions $\mat{X}_n^K$.
\begin{enumerate}
\item[({\bf S1})] Pointwise evaluation within each pixel, i.e.
$
\mat{X}_n^K[i,j] = \fun X_n(t_i^{K_1},s_j^{K_2})$, for $i=1,\ldots,K_1$, $j=1,\ldots,K_2
$,
where $(t_i^{K_1},s_j^{K_2}) \in I_{i,j}^K$ are spatio-temporal locations. Note that the square integrability of $\ope X$ is not sufficient for such pointwise evaluations to be meaningful, so we will assume that $\ope X$ has continuous sample paths in this case \citep[see][]{hsing2015}.
\item[({\bf S2})] Averaged measurements over each pixel, i.e.
\[
\mat{X}_n^K[i,j] = \frac{1}{|I_{i,j}^K|} \iint_{I_{i,j}^K} \fun X_n(t,s) dt ds, \quad i=1,\ldots,K_1, \; j=1,\ldots,K_2.
\]
\end{enumerate}
For the measurement errors $E_n^K[i,j],n=1,\ldots,N,i=1,\ldots,K_1,j=1,\ldots,K_2$, we assume that they are i.i.d., independent of $\mat X_1^K,\ldots,\mat X_N^K$, have mean zero and finite fourth order moment.

We write $\ope Y_n^K = (\fun Y_n^K(t,s): t \in \set T, s \in \set S)$ with $\fun Y_n^K(t,s) = \sum_{i=1}^{K_1}\sum_{j=1}^{K_2} Y_n^K[i,j] \,\1_{\{(t,s) \in I_{i,j}^K\}}$ to denote the \emph{pixelated version} of $\mat Y_n^K$. Let $\widehat{\ope C}_N^K$ be the empirical covariance of $\ope Y_1^K,\ldots,\ope Y_N^K$. It is easy to see that $\widehat{\ope C}_N^K$ is the pixel-wise continuation of $\widehat{\mat C}_N^K$, the empirical covariance of $\mat Y_1^K,\ldots,\mat Y_N^K$. In the discrete observation scenario, our estimator is the best $R$-separable approximation of $\widehat{\ope C}_N^K$
\begin{equation}\label{eq:discrete_measurement_estimator}
\widehat{\ope C}_{R,N}^K = \argmin_{\ope G} \verti{\widehat{\ope C}_N^K - \ope G}_2^2 \quad\text{s.t.}\quad \mathrm{DoS}(\ope G) \le R.    
\end{equation}
%\soham{To derive the rate of convergence of $\widehat{\ope C}_{R,N}^K$, we need some \emph{continuity} assumption relating the random surface $\ope X$ and  its pixelated version $\ope X^K$.}
The following theorem gives the rate of convergence of this estimator when the underlying covariance is {\em Lipschitz continuous}.

\begin{theorem}\label{thm:discrete_observation}
Let $\ope X_1,\ldots,\ope X_N \sim \ope X$ be a collection of random surfaces on $[0,1]^2$ with $\E(\|\ope X\|^4) < \infty$, where the covariance of $\ope X$ has separable expansion $\ope C = \sum_{i=1}^\infty \sigma_i \ope A_i \ct \ope B_i$, with $\sigma_1 > \cdots > \sigma_R > \sigma_{R+1} \ge \cdots \ge 0$. Further assume that the kernel $c(t,s,t^\prime,s^\prime)$ of $\ope C$ is $L$-Lipschitz continuous on $[0,1]^4$. Suppose that the surfaces are measured on a $K_1 \times K_2$ grid according to \eqref{eq:measurement_with_noise_grid}, where the measurement errors $E_n^K[i,j],n=1,\ldots,N$,$ i=1,\ldots,K_1$, \mbox{$j=1,\ldots,K_2$} are i.i.d., independent of $\mat X_1^K,\ldots,\mat X_N^K$, with $\E(E_n^K[i,j]) = 0$, $\var(E_n^K[i,j]) = \sigma^2$ and $\E(|E_n^K[i,j]|^4)<\infty$. Also suppose that one of the following holds.
\begin{enumerate}
\item $\ope X$ has almost surely continuous sample paths and $\mat{X}_1^K,\ldots,\mat{X}_N^K$ are obtained from $\ope X_1,\ldots,\ope X_N$ under the measurement scheme (S1).
\item $\mat{X}_1^K,\ldots,\mat{X}_N^K$ are obtained from $\ope X_1,\ldots,\ope X_N$ under the measurement scheme (S2). 
\end{enumerate}
If the estimator $\widehat{\ope C}_{R,N}^K$ is defined as in \eqref{eq:discrete_measurement_estimator}, then
\begin{align*}
&\verti{\Chat_{R,N}^K - \ope C}_{2} = \left(\sum_{i=R+1}^\infty \sigma_i^2\right)^{1/2} + {\cal O}_{\Prob}\bigg(\frac{a_R}{\sqrt{N}}\bigg) \\
&\kern17ex + \big(16 a_R + \sqrt{2}\big)L \left(\frac{1}{K_1^2} + \frac{1}{K_2^2}\right)^{1/2} + \frac{8\sqrt{2}L^2}{\verti{\ope C}_2}\bigg(\frac{1}{K_1^2} + \frac{1}{K_2^2}\bigg)a_R \\
&\kern17ex +\frac{(8\sqrt{2}a_R+1)\sigma^2}{\sqrt{K_1K_2}} + \frac{4\sqrt{2}\sigma^4a_R}{\verti{\ope C}_2K_1K_2} + \frac{16L\sigma^2a_R}{\verti{\ope C}_2\sqrt{K_1K_2}} \bigg(\frac{1}{K_1^2}+\frac{1}{K_2^2}\bigg)^{1/2},
\end{align*}
where the ${\cal O}_{\Prob}$ term is uniform in $K_1,K_2$ and $a_R = \verti{\ope C}_2\sum_{i=1}^R (\sigma_i/\alpha_i)$ is as in Theorem~\ref{thm:fully_observed}.
\end{theorem}
The theorem shows that in the case of discretely observed data, we get the same rate of convergence as in the fully observed case, plus additional terms reflecting the estimation error of $R$ terms at finite resolution from noise-contamintated data. In particular, if $K_1 \approx K_2 \approx K$, the additional terms are $\O(a_R\sigma^2/K)$. We assumed Lipschitz continuity to quantifiably control those error terms and derive rates of convergence, but the condition is not necessary if we merely seek consistency, which can be established assuming continuity alone.

\subsection{Unbalanced Design}

The method that we have described so far is suitable for surfaces observed on the same regular grid. However, it is possible to have irregular and uneven observations. Next, we describe a way to adapt our proposed method to this setup. Suppose that for the $n$-th random surface $\ope X_n = (\fun X_n(t,s): t \in T, s \in S)$, we make $K_n$ measurements at bivariate locations $(t_1,s_1),\ldots,(t_{K_n},s_{K_n})$. To fix ideas, we consider the \emph{point-wise evaluation with additive noise} model
\begin{equation}\label{eq:measurement_with_noise_irregular}
Y_{ni} = \fun X_n(t_{ni},s_{ni}) + E_{ni}, \qquad i=1,\ldots,K_n,n=1,\ldots,N,
\end{equation}
where $E_{ni}$'s are i.i.d.\ with mean zero and variance $\sigma^2$, independent of $\ope X_n$. To tackle this situation, we take the classical approach, where we first smooth the observations and then apply our method on the smoothed surfaces \citep[e.g.,][]{ramsay2002,ramsay2005}. To be precise, based on the measurements for the $n$-th surface, first we construct its \emph{smoothed} or functional version
\begin{equation}\label{eq:Xsmooth}
\widetilde{\ope X}_n = \mathcal S(Y_{n1},\ldots,Y_{nK_n}), \qquad n=1,\ldots,N,
\end{equation}
where $\mathcal S$ is a \emph{smoothing operator} that takes as input the discrete measurements and outputs a~function. Then, we apply our methodology to the surfaces $\widetilde{\ope X}_1,\ldots,\widetilde{\ope X}_N$, i.e., use the estimator
\[
\widetilde{\ope C}_{R,N} = \argmin_{\ope G} \verti{\widetilde{\ope C}_N-\ope G}_2^2 \quad\text{s.t.}\quad \mathrm{DoS}(\ope G) \le R,
\]
where $\widetilde{\ope C}_N$ is the empirical covariance based on $\widetilde{\ope X}_1,\ldots,\widetilde{\ope X}_N$. The rate of convergence of this estimator depends crucially on the smoother, as shown in the following theorem.

\begin{theorem}\label{thm:rate_of_convergence_smoother}
Let $\ope X_1,\ldots,\ope X_N \sim \ope X$ be a collection of i.i.d.\ random elements of \mbox{$\spa H = \spa H_1 \otimes \spa H_2$} with $\E(\|\ope X\|^4) < \infty$, and the covariance of $\ope X$ has separable expansion \mbox{$\ope C = \sum_{i=1}^\infty \sigma_i \ope A_i \ct \ope B_i$}, with $\sigma_1 > \cdots > \sigma_R > \sigma_{R+1} \ge \cdots \ge 0$. Suppose that the surfaces are observed discretely as \eqref{eq:measurement_with_noise_irregular} and let $\widetilde{\ope X}_n$ denote the smoothed surface from \eqref{eq:Xsmooth}. Let $\widehat{\ope C}_N$ and $\widetilde{\ope C}_N$ be the empirical covariance operators based on $\ope X_1,\ldots,\ope X_N$ and $\widetilde{\ope X}_1,\ldots,\widetilde{\ope X}_N$, respectively, and suppose that $\verti{\widehat{\ope C}_N - \widetilde{\ope C}_N}_2 = \O_{\Prob}(b_N)$. Then,
\[
\verti{\widetilde{\ope C}_{R,N} - \ope C}_2 \le \left(\,\sum_{i=R+1}^\infty \sigma_i^2\right)^{1/2} + \O_{\Prob}\left(a_R b_N\right),
\]
where $a_R = \verti{\ope C}_2\sum_{i=1}^R (\sigma_i/\alpha_i)$ is as in Theorem~\ref{thm:fully_observed}.
\end{theorem}

The theorem shows that the rate of convergence of the estimator in this case remains similar to the case of fully observed surfaces except $N^{-1/2}$ is replaced by $b_N$, the rate of convergence of $\vertj\widehat{\ope C}_N - \widetilde{\ope C}_N\vertj_2$. Note that $b_N$ depends on the interplay between (i) the denseness of the measurements (via $K_1,\ldots,K_N$), (ii) the degree of noise (via the noise variance $\sigma^2$), (iii) the smoother used, and (iv) the smoothness of the underlying surfaces $\ope X_1,\ldots,\ope X_N$, and in certain scenarios is equal to the optimal rate of $N^{-1/2}$ \citep[e.g.,][Theorem~3]{hall2006}.

In practice, one does not construct the whole functions $\widetilde{\ope X}_1,\ldots,\widetilde{\ope X}_N$, but rather evaluate them at a fixed number of points. It is possible to develop the asymptotic properties of our estimator in this setup similar to Theorem~\ref{thm:discrete_observation}. But we avoid doing that for the sake of brevity.

\section{Empirical Demonstration}

\subsection{Simulation Study}\label{sec:simulations}

We now explore the finite sample performance of the proposed estimation and prediction methodology by means of a simulation study. We consider two scenarios for the covariance $\mat C$. Firstly, we consider the parametric covariance of Example~\ref{ex:gneiting}. We choose the parameters as before to have the ground truth $\mat C \in \R^{K \times K \times K \times K}$ whose eigenvalues and separable expansion scores are plotted in Figure~\ref{fig:Irish_Wind}. Note that this covariance is \emph{not} $R$-separable for any finite $R$. In the second scenario, we construct an $R$-separable covariance with $R=3$, i.e.\ $\mat C = \sum_{r=1}^3 \sigma_r \mat A_r \ct \mat B_r$, where $\mat A_1, \mat A_2, \mat A_3,\mat B_1, \mat B_2, \mat B_3 \in \R^{K \times K}$ are all covariances on a univariate domain with the Fourier basis eigenfunctions and a power decay of the eigenvalues. Those univariate covariances are distinct, however, since the ordering of the basis is different in all cases. Still, having the same set of eigenfunctions for all the constituents makes the covariance weakly separable, hence we also consider the model of \cite{lynch2018} in the second scenario. The scores are chosen as $\sigma_1=8$, $\sigma_2=4$, $\sigma_3=2$. See Appendix~\ref{sec:setup_for_simulations} for a full specification of both the scenarios.

In both scenarios, we fix $K=50$ and generate $\mat X_1,\ldots,\mat X_N \in \R^{K \times K}$ as independent zero-mean Gaussians with covariance $\mat C$. Then we fit the respective estimator $\widehat{\mat C}$ using the data and calculate the relative Frobenius error defined as
$
\| \widehat{\mat C} - \mat C \|_F / \| \mat C \|_F .
$
This is done for different sample sizes $N$, and the reported results are averages over 25 independent Monte Carlo runs.

Figure~\ref{fig:gneiting} shows how the relative error decreases as a function of $N$ for several competing methods. First we focus on the $R$-separable estimator, for which the choices $R=1,2,3$ are considered. According to Theorem~\ref{thm:fully_observed}, the relative error converges as $N \to \infty$ to the relative bias, whose square is given by $\sum_{r=R+1}^\infty \sigma_r^2 / \sum_{r=1}^\infty \sigma_r^2$.
% \begin{equation}\label{eq:bias}
% \left(\sum_{r=R+1}^\infty \sigma_r^2 \right)^{1/2} \Bigg/ \left(\sum_{r=1}^\infty \sigma_r^2 \right)^{1/2} .
% \end{equation}
This is the minimal achievable error by means of an $R$-separable approximation, even if we knew $\mat C$, and it is depicted by a dashed horizontal line (an asymptote) in Figure~\ref{fig:gneiting} for every considered $R$. As expected, for $R=1$ the relative error converges fast to its asymptote (i.e.\ the \emph{variance} converges to zero), which is higher than the asymptotes for larger values of $R$ (higher values of $R$ reduce the bias). However, the speed of convergence to these lower asymptotes is substantially slower (the variance goes to zero more slowly). Hence choosing $R$ to be larger does not necessarily lead to a smaller error: there is a~genuine bias-variance trade-off. This can be observed both in the parametric setup (Figure~\ref{fig:gneiting}, left) and the non-parametric setup (Figure~\ref{fig:gneiting}, right). In the parametric setup, the empirical estimator $\widehat{\mat C}_N$ is the only unbiased estimator we consider, and it is substantially worse than any of the $R$-separable approximations. Even though it is the only estimator for which the error will eventually converge to 0 for $N \to \infty$, $N$ would have to be extremely large for the empirical estimator to beat the $R$-separable estimators with reasonably chosen degree-of-separability $R$.

The situation is similar in the non-parametric setup (Figure~\ref{fig:gneiting}, right) with two exceptions. First, the ground truth is $3$-separable in this case and hence the error of the estimator with $R=3$ converges to zero. Second, the decay of separable expansion scores is slower in the nonparametric case (one can see this from the asymptotes in Figure~\ref{fig:gneiting}), which means the differences between different values of $R$ manifest themselves for smaller $N$ already.

\begin{figure}[!t]
   \centering
   \begin{tabular}{cc}
   \includegraphics[width=0.47\textwidth]{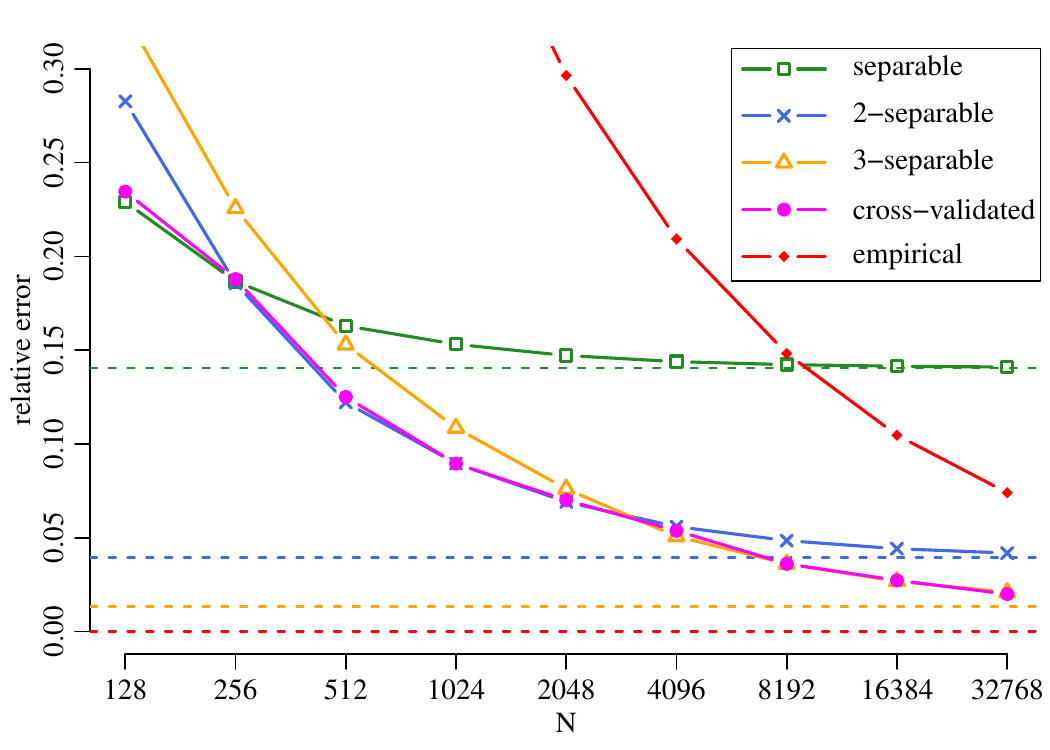} &
   \includegraphics[width=0.47\textwidth]{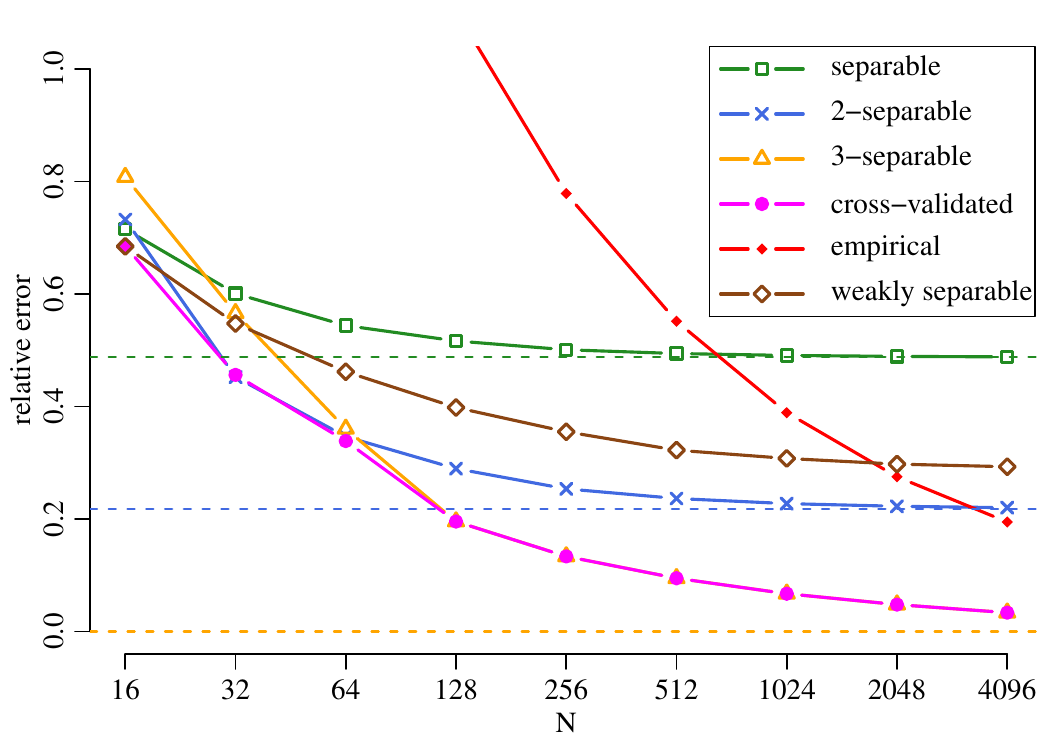}
   \end{tabular}  
   \caption{Relative estimation error depending on sample size $N$. Considered estimators are the separable estimator (R=1), $R$-separable estimators with $R=2$ and $R=3$, $R$-separable estimator with cross-validated $R$, the weakly separable estimator, and the empirical covariance estimator. Straight horizontal lines (asymptotes) show the bias of $R$-separable estimators for $R=1,2,3$.}
    \label{fig:gneiting} 
\end{figure}

In practice, we naturally do not know $\mat C$, %and hence we cannot calculate the errors as in Figure~\ref{fig:gneiting}. Therefore an automatic procedure to choose $R$ is essential. 
and hence cannot choose $R$ to optimally balance bias and variance as a visual inspection of Figure~\ref{fig:gneiting} might allow one to. Instead, we need to use the cross-validation strategy described in Section~\ref{sec:rank}. Figure~\ref{fig:gneiting} also shows the relative errors achieved with a cross-validated choice of $R$. Cross-validation appears to work rather well, with the cross-validated error curve forming almost a lower envelope of the curves for $R=1,2,3$, i.e.\ leading to an error that is always near optimal.

In the remainder of this section, we probe the performance of our prediction algorithm. We focus on the non-parametric setup (the results for the parametric setup are deferred to the appendix). The training sample is generated as above together with an additional test sample. The covariance is fitted on the training sample only, and then the final row and column of every observation in the test sample are predicted using the fitted covariance together with the remainder of the given observation (i.e.\ we perform one-step ahead prediction both in space and in time, at once). We use a small amount of ridge regularization ($\epsilon = 10^{-3}$) for all the competing methods.

Figure~\ref{fig:prediction} (left) shows relative prediction errors achieved by different estimators of the covariance. As before, higher values of $R$ lead to better performance as $N$ increases, and the empirical estimator does not perform well. Moreover, prediction with the empirical estimator is computationally more costly, c.f.\ the costs of inversion in Table~\ref{tab:complexities}. To demonstrate this in practice, Figure~\ref{fig:prediction} (right) shows the runtime of a single prediction task, run on a standard laptop. The memory complexity of constructing the empirical covariance restricts us to modest grid sizes  (up to $K=140$, we ran out of memory at $K=170$). Moreover, the runtimes explode for the empirical covariance. At the edge of feasibility ($K=140$), our inversion algorithm with $R=3$ runs about 1200 times faster compared to the empirical covariance, while it runs only 29 times slower compared to the separable model. It is shown empirically in Appendix~\ref{sec:further_simulations} that the number of iterations of our inversion algorithm does not depend on the grid size and hence the theoretical complexity of prediction for the $R$-separable model is the same as for the separable model. Nonetheless, the cost in practice is higher due to the iterative nature and some overhead calculations. The numerical behavior of the inversion algorithm is  discussed further in Appendix~\ref{sec:further_simulations}.

Finally, the non-parametric setup is by construction weakly separable, and hence we also compare our methodology with the weakly separable model of \cite{lynch2018}. There we use the 95\% cut-off for the variance explained in space and time, as suggested by the authors. As displayed in Figures~\ref{fig:gneiting} (right) and~\ref{fig:prediction} (left), the weakly separable model does better than the separable model and the empirical estimator, but is outperformed eventually by the proposed $R$-separable model.

\begin{figure}[!t]
   \centering
   \begin{tabular}{cc}
   \includegraphics[width=0.47\textwidth]{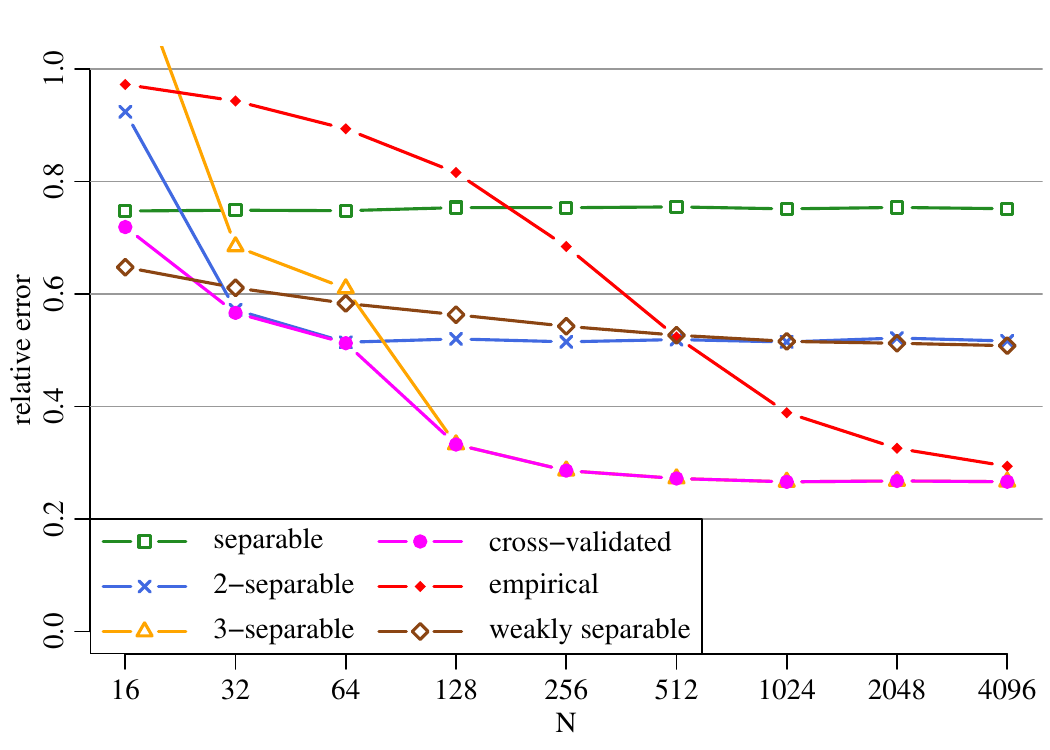} &
   \includegraphics[width=0.47\textwidth]{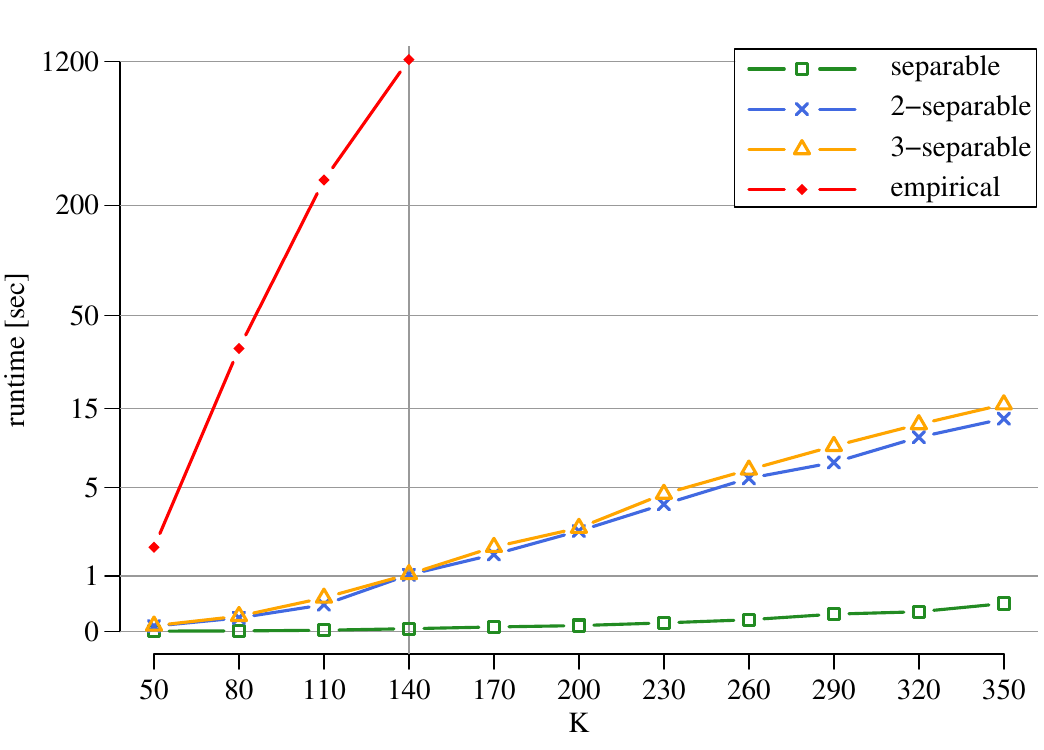}
   \end{tabular}  
   \caption{ Left: Relative prediction error depending on sample size $N$. Considered estimators are the separable estimator (R=1), $R$-separable estimators with $R=2$ and $R=3$, $R$-separable estimator with cross-validated $R$, the weakly separable estimator, and the empirical covariance estimator. Right: Log-runtimes of calculating a single prediction with a separable covariance (R=1), $R$-separable covariance with $R=2$ and $R=3$, and an unstructured covariance.}
    \label{fig:prediction} 
\end{figure}

\subsection{Classification of EEG signals}\label{sec:eeg}

Electroencephalography (EEG) is dominantly utilized for screening and diagnosis of various mental disorders, such as epilepsy or autism. It is an effective monitoring procedure for studying various brain activities. Since the responsive capacity of the brain is severely affected by alcoholism, EEG can also be used for detection and diagnosis of alcoholism. EEG signals are acquired using numerous electrodes placed on the scalp, and each electrode produces a time series of measurements, sampled over a specific interval. It is natural to consider all measurements per a single patient as a random surface, with time being one dimension and space (electrode placement) being the second dimension. We consider detection of alcoholism as a classification problem based on random surfaces, which are event-related EEG signals.
We work with a data set of 77 alcoholic and 45 control subjects, which is freely available from the University of California, Irvine machine learning database\footnote{https://archive.ics.uci.edu/ml/datasets/eeg+database, downloaded on 14 May 2021.}. Each subject was repeatedly exposed to either a~single stimulus (Condition~1) or two matching stimuli (Condition~2) or two non-matching stimuli (Condition~3). There was a total of 120 of these 1-second-long trials sampled at 256 Hz, but many were discarded right after the acquisition due to artifacts such as blinking. We discarded one of the control-group subjects, since this was a clear outlier with only 19 successful trials accross all three conditions. After that, we took averages of 10 (the maximum possible number for the remaining subjects) random available trials per condition, resulting in a data set $\mat X_1, \ldots, \mat X_{121} \in \R^{256 \times 64}$ of $N=121$ subjects, $K_1=256$ time points and $K_2=64$ spatial locations for each of the three conditions. Also, the class membership variables $Y_1,\ldots,Y_{121} \in \{ 0,1 \}$ (control or alcoholic) for all the subjects are available. 

Classification of subjects into their respective classes (control vs.\ alcoholic) using the EEG data set described above has been conducted many times before -- a literature search reveals dozens of papers published this year only. These attempts differ in many aspects, e.g.\ in how the data are sub-sampled, pre-processed or filtered, or which features are selected and which type of classifier is used. For example, \citet{prabhakar2020} compare over 50 different classification approaches with their accuracy varying between 80\% and 99\%. However, our goal is \emph{not} to build an additional classifier to compete with the state-of-the-art in EEG classification. It is merely to demonstrate how covariance estimation beyond separability can improve classification accuracy. For this purpose, we consider the functional linear discriminant analysis (fLDA) classifier \citep{oxford2011}. To the best of our knowledge, fLDA classifier has not been used before with the EEG data set.

Specifically, we utilize the centroid classifier of \citet{delaigle2012}. Assuming that the control group has a Gaussian distribution with mean $\mu_0$ and covariance $C$ while the alcoholic group has a Gaussian distribution with mean $\mu_1$ and covariance $C$, the optimal classifier (i.e.~a~predictor of the class membership $Y$) based on a one-dimensional projection is given by
\begin{equation}\label{eq:classifier}
\widehat Y := \1_{\{\langle X -\mu_0, \psi \rangle > \langle X -\mu_1, \psi \rangle\}},
\end{equation}
where $\psi$ is a solution to the linear problem involving the covariance, namely
\begin{equation}\label{eq:classification_inverse_problem}
    C \psi = \mu_1 - \mu_0,
\end{equation}
provided this solution exists. If the solution does not exist, neither does the optimal centroid classifier. Regardless,  $\mu_0$, $\mu_1$ and $C$ are unknown in practice and have to be estimated from the data. The estimator of $C$ is typically regularized by adding a ridge \citep{oxford2011} so the inverse problem can be solved and an estimator of $\psi$ is obtained. Then, one obtains the classifier by plugging in $\widehat \psi$ as well as $\widehat \mu_0$ and $\widehat \mu_1$ into \eqref{eq:classifier}.

While $\mu_0$ and $\mu_1$ can be easily estimated as the empirical means of the respective classes, estimation of the covariance poses an issue here. Even though the empirical covariance \mbox{$\widehat{\mat C}_N \in \R^{256 \times 64 \times 256 \times 64}$} can be evaluated in principle, we cannot expect it to be a good estimator, given the evidence in the simulation study. Moreover, and more importantly, the inverse problem \eqref{eq:classification_inverse_problem} is not computationally feasible with $C$ estimated empirically. 

Instead, one may use the separable estimator or the proposed $R$-separable estimator for $C$. In this particular case, the cross-validation strategy of Section~\ref{sec:rank} suggests the choice $R=2$ for all three conditions. Then, the inverse algorithm of Section~\ref{sec:inversion} can be used to solve \eqref{eq:classification_inverse_problem} efficiently, and the classifier is obtained easily.

Table~\ref{tab:accuracies} demonstrates the gains acquired by estimation beyond separability. To compare the resulting two classifiers (one for the separable estimator and other for the $2$-separable estimator of the covariance), we split the data set for every condition into folds $\mathcal F_k$, $k=1,\ldots,24$ of size $5$. Let $Y_{k,j}$, $k=1,\ldots,24$, $j=1,\ldots,5$, denote the class membership of the $j$-th observation in the $k$-th fold, and $\widehat Y_{k,j}^{(-k)}(R,\epsilon)$ denote the predicted class membership of $Y_{k,j}$ obtained by the fLDA classifier trained solely on folds $\mathcal{F}_{k'}$, $k'\neq k$, using the $R$-separable estimator of the covariance and $\epsilon I$ as the ridge regularizer. Out-of-sample cross-validated classification accuracy for a given classifier is then calculated as
\[
\mathrm{ACC}(R,\epsilon) = \sum_{k=1}^{24} \sum_{j=1}^5 \left| \widehat Y_{k,j}^{(-k)}(R,\epsilon) - Y_{k,j} \right| \Big/ 120.
\]
The maximum accuracies over a grid of ridge constants $\epsilon$ are reported in Table~\ref{tab:accuracies}. For every condition, the proposed $R$-separable estimator with $R=2$ (which is the degree-of-separability suggested by cross-validation) clearly outperforms the separable alternative of $R=1$.

\begin{table}[t]
    \centering
    \caption{Out-of-sample cross-validated classification accuracy (i.e.\ the ratio between the correctly classified and all subjects) for the fLDA classifier with the separable ($R=1$) or $R$-separable ($R=2$) estimator of the covariance, and for three separate data sets given by different conditions.}
    \begin{tabular}{cccc}
    \toprule
    Degree-of-separability & Condition 1 & Condition 2 & Condition 3 \\
    \midrule
    R=1  & 78 \% & 78 \% & 77 \% \\
    R=2  & 90 \% & 84 \% & 91 \% \\
    \bottomrule
    \end{tabular}
    \label{tab:accuracies}
\end{table}

\newpage

%%%%%%%%%%%%%%%%%%%%%%%%%%%%%%%%%%%%%%%%%%%%%%%%%%%%%%%%%%%%%%%%%%%%%%%%
%%%%%%%%%%%%%%%%%%%%%%%%%%%%%%%%%%%%%%%%%%%%%%%%%%%%%%%%%%%%%%%%%%%%%%%%
%%%%%%%%%%%%%%%%%%%%%%%%%%%%%%%%%%%%%%%%%%%%%%%%%%%%%%%%%%%%%%%%%%%%%%%%

\section{Appendices}
%\addtocontents{toc}{\protect\setcounter{tocdepth}{1}}
\renewcommand{\thesubsection}{\Alph{subsection}}
\numberwithin{lemma}{subsection}
\numberwithin{definition}{subsection}
\numberwithin{proposition}{subsection}
\numberwithin{equation}{subsection}
\numberwithin{corollary}{subsection}
\numberwithin{remark}{subsection}
\numberwithin{theorem}{subsection}
\numberwithin{example}{subsection}
\numberwithin{figure}{subsection}

These appendices contain proofs of the formal statements in the main text, additional simulation results, and some details which were left out of the main text. References to the paper are made in the standard way, while equations, lemmas, etc., that are introduced in the appendices are labelled and referred to with a section letter (from `A' to `E'). In Section~\ref{append:proofs_Tomas}, we give proofs of the results in Section~\ref{sec:methodology}. Perturbation bounds related to the separable expansion are derived in Section~\ref{append:perturbation}. In Section~\ref{append:proofs}, we give proofs of the rate of convergence results in Section~\ref{sec:analysis}, and derive some results on the rates under special decay structures on the separable component scores. Finally, details about the simulations in Section~\ref{sec:simulations} are given in Section~\ref{sec:setup_for_simulations}, while additional empirical results are provided in Section~\ref{sec:further_simulations}.

\subsection{Proofs of the Results in Section~\ref{sec:methodology}}\label{append:proofs_Tomas}

In this section, we give proofs of the results in Section~\ref{sec:methodology} of the main text. We start with the uniqueness of the partial inner products.
\begin{lemma}\label{lemma:PIP_unique}
Let $\spa H=\spa H_1 \otimes \spa H_2$. Then, there exist two unique bi-linear operators $T_1: \spa H \times \spa H_1 \to \spa H_2$ and $T_2: \spa H \times \spa H_2 \to \spa H_1$ defined by
\[
\begin{split}
T_1(x \otimes y, e) &= (x \otimes_1 y) e , \quad x,e \in \spa H_1, \, y \in \spa H_2,\\
T_2(x \otimes y, f) &= (x \otimes_2 y) f , \quad x \in \spa H_1, \, y,f \in \spa H_2.
\end{split}
\]
\end{lemma}

\begin{proof}
The set $\spa H_0 = \big\{\sum_{j=1}^m x_j \otimes y_j: m \in \N, x_j \in \spa H_1, y_j \in \spa H_2\big\}$ is dense in $\spa H$. Consider an element $\ope G = \sum_{j=1}^m x_j \otimes y_j \in \spa H_0$. Then, for $e \in \spa H_1$ and $f \in \spa H_2$, we have
\[
\begin{split}
\langle e, T_2(\ope G,f) \rangle_{\spa H_1} &= \big\langle e, \sum_{j=1}^m \langle y_j,f \rangle_{\spa H_2} x_j \big\rangle_{\spa H_1} = \sum_{j=1}^m \langle x_j, e \rangle_{\spa H_1} \langle y_j,f \rangle_{\spa H_2} \\
&=
 \sum_{j=1}^m \langle x_j \otimes y_j, e \otimes f \rangle_{\spa H} = \langle \ope G, e\otimes f \rangle_{\spa H}.
 \end{split}
\]
Choosing $e = T_2(\ope G,f)$ and using the Cauchy-Schwarz inequality, we obtain
\begin{equation*}
\|T_2(\ope G,f) \|_{\spa H_1}^2 = \langle T_2(\ope G,f) , T_2(\ope G,f) \rangle_{\spa H_1} = \langle \ope G , T_2(\ope G,f) \otimes f \rangle_{\spa H} \leq \| \ope G \|_{\spa H} \| T_2(\ope G,f) \|_{\spa H_1} \| f \|_{\spa H_2} .    
\end{equation*}
Hence $\|T_2(\ope G,f) \|_{\spa H_1} \leq \| \ope G \|_{\spa H} \| f \|_{\spa H_2}$. Therefore, $T_2(\cdot, f)$ can be continuously extended from the set $\spa H_0$ to the whole space $\spa H$, and the uniqueness holds true.
\end{proof}

\begin{proof}[Proof of Proposition~\ref{prop:integral_rep_PIP}]
Since $g$ is an $\spa L_2$ kernel, there is a Hilbert-Schmidt operator $\ope G:\spa H_2 \to \spa H_1$ with SVD $G = \sum \sigma_j e_j \otimes_2 f_j$ with kernel given by $g(t,s) = \sum \sigma_j e_j(t) f_j(s)$ (note that the sum only converges in the $\spa L_2$ sense, not uniformly). By isometry, from the SVD we have $g = \sum \sigma_j e_j \otimes f_j$, so
\[
u = T_2(g,v) = \sum_{j=1}^\infty \sigma_j T_2(e_j \otimes f_j,v) = \sum_{j=1}^\infty \sigma_j \langle f_j,v \rangle e_j = \sum_{j=1}^\infty \sigma_j e_j \int_{E_2} f_j(s)v(s) d \mu_2(s) 
\]
and hence by Fubini's theorem:
\[
u(t) = \sum_{j=1}^\infty \sigma_j e_j(t) \int_{E_2} f_j(s)v(s) d \mu_2(s)   = \int_{E_2} \left[ \sum_{j=1}^\infty \sigma_j e_j(t) f_j(s) \right] v(s) d \mu_2(s) 
\]
from which the claim follows.
\end{proof}

\begin{proof}[Proof of Lemma~\ref{lem:PIP_basic_properties}]
The first claim follows directly from the definition of the partial inner product.

In the second part, we want to show that if both $\ope C$ and the respective weighting $\ope W_1$ or $\ope W_2$ have the property of positive semi-definiteness, then the partial inner product will retain that property. Consider the eigendecompositions $\ope C = \sum \lambda_j g_j \otimes g_j$ and $\ope W_2 = \sum \alpha_j h_j \otimes h_j$. Then we have
\[
T_2(\ope C,\ope W_2) = \sum_{j=1}^\infty \sum_{i=1}^{\infty} \lambda_j \alpha_i T_2(g_j \otimes g_j, h_i \otimes h_i) = \sum_{j=1}^\infty \sum_{i=1}^{\infty} \lambda_j \alpha_i T_2(g_j, h_i) \otimes T_2(g_j, h_i),
\]
where the last equality can be verified on rank-1 elements $g_j$. Thus $T_2(\ope C,\ope W_2)$ is a weighted sum of quadratic forms with weights given by the non-negative eigenvalues of $\ope C$ and $\ope W_2$. As such, $T_2(\ope C,\ope W_2)$ must be positive semi-definite.
\end{proof}

\begin{proof}[Proof of Proposition~\ref{prop:convergence}]
Let $\ope C_1 := \sum_{j=1}^\infty \sigma_j^2 \ope A_j \otimes \ope A_j$ and $\ope C_2 := \sum_{j=1}^\infty \sigma_j^2 \ope B_j \otimes \ope B_j$. We will abuse the notation slightly and denote for any $k \in \N$
\[
\ope C_1^k := \sum_{j=1}^\infty \sigma_j^{2k} \ope A_j \otimes \ope A_j \quad \& \quad \ope C_2^k := \sum_{j=1}^\infty \sigma_j^{2k} \ope B_j \otimes \ope B_j .
\]
Note that if $\ope C$ was an operator, it would hold that $\ope C_1 = \ope C {\ope C}^\ast$ and $\ope C_2 = {\ope C}^\ast \ope C$, while $\ope C_1^k$ and $\ope C_2^k$ would denote the powers as usual. However, we aim for a more general statement, forcing us to view $\ope C$ as an element of a product space rather than an operator, so the ``powers'' serve just as a notational convenience in this proof. Also, the proportionality sign is used here to avoid the necessity of writing down the scaling constants for unit norm elements.

From the recurrence relation, and the definition of the partial inner product, we have
\[
\begin{split}
\ope V^{(1)} \propto T_1\big(\ope C,T_2(\ope C,\ope V^{(0)})\big) &= T_1\left(\sum_{j=1}^\infty \sigma_j \ope A_j \otimes \ope B_j, \sum_{i=1}^\infty \sigma_i \langle \ope B_i, \ope V^{(0)} \rangle \ope A_i\right) \\
&= \sum_{i=1}^\infty \sum_{j=1}^\infty \sigma_i \sigma_j  \langle \ope B_i, \ope V^{(0)} \rangle T_1(\ope A_j \otimes \ope B_j, \ope A_i).
\end{split}
\]
Since $T_1(\ope A_j \otimes \ope B_j, \ope A_i) = \langle \ope A_j,\ope A_i \rangle \ope B_j = \1_{[i=j]} \ope B_j$, we have
\[
\ope V^{(1)} \propto \sum_{j=1}^\infty \sigma_j^2 \langle \ope B_i, \ope V^{(0)} \rangle \ope B_j = T_2(\ope C_2, \ope V^{(0)}).
\]
By the same token we have $\ope V^{(2)} \propto T_2(\ope C_2, \ope V^{(1)})$, from which we obtain similarly
\[
\begin{split}
\ope V^{(2)} &\propto T_2\left(\sum_{j=1}^\infty \sigma_j^2 \ope B_j \otimes \ope B_j, \sum_{i=1}^\infty \sigma_i^2 \langle \ope B_i, \ope V^{(0)} \rangle \ope B_i\right) = \sum_{i=1}^\infty \sum_{j=1}^\infty \sigma_i^2 \sigma_j^2  \langle \ope B_i, \ope V^{(0)} \rangle T_2(\ope B_j \otimes \ope B_j, \ope B_i)\\
&=\sum_{i=1}^\infty \sum_{j=1}^\infty \sigma_i^2 \sigma_j^2  \langle \ope B_i, \ope V^{(0)} \rangle \langle \ope B_j,\ope B_i \rangle \ope B_j = \sum_{j=1}^\infty \sigma_j^4 \langle \ope B_j, \ope V^{(0)} \rangle \ope B_j = T_2\left(\sum_{j=1}^\infty \sigma_j^4 \ope B_j \otimes \ope B_j, \ope V^{(0)}\right) \\
&= T_2(\ope C_2^2,\ope V^{(0)}).
\end{split}
\]
By induction, we have $\ope V^{(k)} \propto T_2(\ope C_2^k,\ope V^{(0)})$ for $k=1,2,\ldots$.

Now we express the starting point $\ope V^{(0)}$ in terms of the orthonormal basis $\{ \ope B_j \}$. Let $\ope V^{(0)} = \sum_{j=1}^\infty \beta_j \ope B_j$, where we have $\beta_j = \langle \ope B_j,\ope V^{(0)} \rangle$. Then we can express
\[
\ope V^{(k)} \propto T_2(\ope C_2^2,\ope V^{(k)}) = T_2\left(\sum_{j=1}^\infty \sigma_j^{2k} \ope B_j \otimes \ope B_j, \sum_{i=1}^\infty \beta_i \ope B_i\right) = \sum_{i=1}^\infty \sum_{j=1}^\infty \sigma_j^{2k} \beta_i \langle \ope B_i, \ope B_j \rangle \ope B_j = \sum_{j=1}^\infty \sigma_j^{2k} \beta_j \ope B_j.
\]

Hence we have an explicit formula for the $k$-th step:
\[
\ope V^{(k)} = \frac{\ope B_1 + \ope R^{(k)}}{\|\ope B_1 + \ope R^{(k)} \|},
\]
where $\ope R^{(k)} = \sum_{j=2}^\infty \left( \frac{\sigma_j}{\sigma_1} \right)^{2k} \frac{\beta_j}{\beta_1} \ope B_j$.

It remains to show that $\| \ope R^{(k)} \| \to 0$ for $k \to \infty$. Due to the decreasing ordering of the scores $\{ \sigma_j \}$, we have
\[
\begin{split}
\| \ope R^{(k)} \|^2 &= \sum_{j=2}^\infty \left( \frac{\sigma_j}{\sigma_1} \right)^{4k} \frac{\beta_j^2}{\beta_1^2} \leq \left( \frac{\sigma_2}{\sigma_1} \right)^{4k-2} \sum_{j=2}^\infty \left( \frac{\sigma_j}{\sigma_1} \right)^{2} \frac{\beta_j^2}{\beta_1^2} \\
&=  \left( \frac{\sigma_2}{\sigma_1} \right)^{4k-2} \frac{1}{\sigma_1^2 \beta_1^2} \sum_{j=1}^\infty \sigma_j^2 \beta_j \leq \left( \frac{\sigma_2}{\sigma_1} \right)^{4k-2} \frac{1}{\sigma_1^2 \beta_1^2} \| \ope C \|,
\end{split}
\]
where in the last inequality we used that $|\beta_j|\leq 1$. Since $\sigma_1 > \sigma_j$ for $j \geq 2$, we see that the remainder $\ope R^{(k)}$ goes to zero.

The proof of the statement concerning the sequence $\{ \ope U^{(k)} \}$ follows the same steps.
\end{proof}

\subsubsection{More General Definition of the Partial Inner Product}\label{appendix:PIP_general}

In the main text, we have defined the partial inner products $T_1$ and $T_2$ for the tensor product of two Hilbert spaces $\spa H_1$ and $\spa H_2$. We can define the partial inner product more generally for the tensor product of more than two Hilbert spaces as follows. For $d$ separable Hilbert spaces $\spa H_1,\ldots,\spa H_d$, let
\[
\spa H := \spa H_1 \otimes \ldots \otimes \spa H_d = {\displaystyle \bigotimes_{j=1}^d} \spa H_j
\]
be their tensor product Hilbert space, defined as the completion of
\[
\Big\{\sum_{j=1}^m x_{1j} \otimes \cdots \otimes x_{dj} : m \in \N, x_{ij} \in \spa H_i, i=1,\ldots,d, j=1,\ldots,m\Big\},
\]
under the inner product $\langle x_1 \otimes \cdots \otimes x_d, y_1 \otimes \cdots \otimes y_d \rangle = \langle x_1,y_1 \rangle_{\spa H_1} \times \cdots \times \langle x_d,y_d \rangle_{\spa H_d}$. For $J \subset \{ 1, \ldots,d \}$, let
\[
\spa H_J := {\displaystyle \bigotimes_{j \in J}} \spa H_j \quad \text{ and } \quad \spa H_{-J} := {\displaystyle \bigotimes_{j \notin J}} \spa H_j.
\]
We can define the partial inner product $T_J : \spa H \times \spa H_{J} \to \spa H_{-J} $ to be the unique bi-linear operator such that
\[
T_{J}(\ope X \otimes \ope Y, \ope A) = \langle \ope X,\ope A\rangle_{\spa H_{J}} \ope Y, \quad \forall\, \ope A,\ope X \in \spa H_J, \ope Y \in \spa H_{-J}.
\]
Note that $\spa H_1 \otimes \spa H_2$ is isomorphic to $\spa H_2 \otimes \spa H_1$ with the isomorphism given by $\Phi(x \otimes y) = y \otimes x$, $\forall x \in \spa H_1, y \in \spa H_2$, and the same holds for products of multiple spaces. Hence we can always permute dimensions as we wish. Thus we can w.l.o.g.\ assume that $J=\{1,\ldots,d' \}$ for some $d' < d$. Proving uniqueness is now the same as it was in the case of Lemma~\ref{lemma:PIP_unique}, and Proposition~\ref{prop:integral_rep_PIP} (of the main text) can be generalized to more than two dimensions as well.

\subsection{Perturbation Bounds}\label{append:perturbation}

From the main text, $\ope C \in {\spa S}_2(\spa H_1 \otimes \spa H_2)$ has the separable expansion
\[
\ope C = \sum_{i=1}^{\infty} \sigma_i \ope A_i \ct \ope B_i,
\]
where $\{\ope A_i\}_{i\ge 1}$ (resp., $\{\ope B_i\}_{i \ge 1}$) is an orthonormal basis of ${\spa S}_2(\spa H_1)$ (resp., ${\spa S}_2(\spa H_2)$), and $|\sigma_1| \ge |\sigma_2| \ge \cdots \ge 0$. We use the notation $\mathscr C$ to indicate the element of $\spa S_2(\spa H_2 \otimes \spa H_2, \spa H_1 \otimes \spa H_1)$ that is isomorphic to $\ope C$ (see Section~\ref{sec:SCD}). The following lemma gives perturbation bounds for the components of the separable expansion.
\begin{lemma}[Perturbation Bounds for Separable Expansion]\label{lemma:perturbation_bound}
Let $\ope C$ and $\widetilde{\ope C}$ be two operators in ${\spa S}_2(\spa H_1 \otimes \spa H_2)$ with separable expansions $\ope C = \sum_{i=1}^\infty \sigma_i \ope A_i \ct \ope B_i$ and $\widetilde{\ope C} = \sum_{i=1}^\infty \widetilde\sigma_i \widetilde{\ope A}_i \ct \widetilde{\ope B}_i$, respectively. Also suppose that $\sigma_1 > \sigma_2 > \cdots \ge 0$, and $\langle \ope A_i,\widetilde{\ope A}_i\rangle_{{\spa S}_2(\spa H_1)},\langle \ope B_i,\widetilde{\ope B}_i \rangle_{{\spa S}_2(\spa H_2)} \ge 0$ for every $i=1,2,\ldots$ (adjust the sign of $\widetilde\sigma_i$ as required). Then,
\begin{itemize}
\item[(a)] $\sup_{i \ge 1} \big|\sigma_i - \widetilde\sigma_i\big| \le \verti{\ope C - \widetilde{\ope C}}_{2}$.
\item[(b)] For every $i \ge 1$,
\begin{align*}
\verti{\ope A_i - \widetilde{\ope A}_i}_{{\spa S}_2(\spa H_1)} &\le \frac{2\sqrt{2}}{\alpha_i} \verti{\ope C - \widetilde{\ope C}}_{2} \bigg(\verti{\ope C}_{2} + \verti{\widetilde{\ope C}}_{2} \bigg),\\
\verti{\ope B_i - \widetilde{\ope B_i}}_{{\spa S}_2(\spa H_2)} &\le \frac{2\sqrt{2}}{\alpha_i} \verti{\ope C - \widetilde{\ope C}}_{2} \bigg(\verti{\ope C}_{2} + \verti{\widetilde{\ope C}}_{2}\bigg),
\end{align*}
where $\alpha_i = \min\{\sigma_{i-1}^2-\sigma_i^2, \sigma_i^2 - \sigma_{i+1}^2\}$. Here, $\verti{\cdot}_{2}$ denotes the Hilbert-Schmidt norm.
\end{itemize}
\end{lemma}
\begin{proof}
Note that $\sigma_i$ (resp., $\widetilde\sigma_i$) is the $i$-th singular value of the operator $\mathscr C$ (resp., $\widetilde{\mathscr C}$). Following \citet[][Lemma 4.2]{bosq2012}, we get that $\sup_{i \ge 1} \big|\sigma_i - \widetilde\sigma_i\big| \le \verti{{\mathscr C} - \widetilde{\mathscr C}}_{\infty} \le \verti{{\mathscr C} - \widetilde{\mathscr C}}_{2}$. Part (a) now follows by noting that $\verti{{\mathscr C} - \widetilde{\mathscr C}}_{2} = \verti{\ope C - \widetilde{\ope C}}_{2}$ because of the isometry between ${\spa S}_2(\spa H_1 \otimes \spa H_2)$ and ${\spa S}_2(\spa H_2 \otimes \spa H_2, \spa H_1 \otimes \spa H_1)$.

For part (b), recall that $\ope A_i$ (resp., $\widetilde{\ope A}_i$) is isomorphic to $e_i$ (resp., $\widetilde e_i$), the $i$-th right singular element of $\mathscr C$ (resp., of $\widetilde{\mathscr C}$), see Section~\ref{sec:SCD}. Now, $e_i$ (resp., $\widetilde e_i$) is the $i$-th eigenfunction of ${\mathscr C}{\mathscr C}^\ast$ (resp., $\widetilde{\mathscr C}{\widetilde{\mathscr C}}^\ast$) with corresponding eigenvalue $\lambda_i = \sigma_i^2$ (resp., $\widetilde\lambda_i = \widetilde\sigma_i^2$). Here, ${\mathscr C}^\ast$ (resp., ${\widetilde{\mathscr C}}^\ast$) denotes the adjoint of $\mathscr C$ (resp., $\widetilde{\mathscr C}$). Also, $\langle e_i,\widetilde e_i\rangle_{\spa H_1 \otimes \spa H_1} = \langle \ope A_i,\widetilde{\ope A}_i \rangle_{{\spa S}_2(\spa H_1)} \ge 0$. Now, using a perturbation bound on the eigenfunctions of operators \citep[][Lemma 4.3]{bosq2012}, we get
\[
\|e_i - \widetilde e_i\|_{\spa H_1 \otimes \spa H_1} \le \frac{2\sqrt{2}}{\alpha_i} \verti{{\mathscr C}{\mathscr C}^\ast - \widetilde{\mathscr C} {\widetilde{\mathscr C}}^\ast}_{\infty} \le \frac{2\sqrt{2}}{\alpha_i} \verti{{\mathscr C}{\mathscr C}^\ast - \widetilde{\mathscr C} {\widetilde{\mathscr C}}^\ast}_{2},
\]
where $\alpha_i = \min\{\sigma_{i-1}^2 - \sigma_i^2, \sigma_i^2 - \sigma_{i+1}^2\}$. Now,
\begin{align*}
\verti{{\mathscr C}{\mathscr C}^\ast - \widetilde{\mathscr C}{\widetilde{\mathscr C}}^\ast}_{2} = \verti{{\mathscr C}\big({\mathscr C}-\widetilde{\mathscr C}\big)^\ast + \big({\mathscr C}-\widetilde{\mathscr C}\big){\widetilde{\mathscr C}}^\ast}_{2} &\le \verti{{\mathscr C}-\widetilde{\mathscr C}}_{2} \bigg(\verti{\mathscr C}_{2}+\verti{\widetilde{\mathscr C}}_{2}\bigg) \\
&= \verti{\ope C-\widetilde{\ope C}}_{2} \bigg(\verti{\ope C}_{2}+\verti{\widetilde{\ope C}}_{2}\bigg),
\end{align*}
where we have used: (i) the triangle inequality for the Hilbert-Schmidt norm, (ii) the fact that the Hilbert-Schmidt norm of an operator and its adjoint are the same, and (iii) the isometry between ${\spa S}_2(\spa H_1 \otimes \spa H_2)$ and ${\spa S}_2(\spa H_2 \otimes \spa H_2, \spa H_1 \otimes \spa H_1)$. By noting that 
\[
\verti{\ope A_i - \widetilde{\ope A}_i}_{{\spa S}_2(\spa H_1)} = \|e_i - \widetilde e_i\|_{\spa H_1 \otimes \spa H_1},
\]
the upper bound on $\verti{\ope A_i - \widetilde{\ope A}_i}_{{\spa S}_2(\spa H_1)}$ follows. The bound on $\verti{\ope B_i - \widetilde{\ope B}_i}_{{\spa S}_2(\spa H_2)}$ can be proved similarly.
\end{proof}

The following lemma gives us perturbation bound for the best $R$-separable approximation of Hilbert-Schmidt operators.

\begin{lemma}[Perturbation Bound for Best $R$-separable Approximation]\label{lemma:perturbation_R-sep}
Let $\ope C$ and $\widetilde{\ope C}$ be two Hilbert-Schmidt operators on $\spa H_1 \otimes \spa H_2$, with separable expansions $\ope C=\sum_{r=1}^\infty \sigma_r \ope A_r \ct \ope B_r$ and $\widetilde{\ope C}=\sum_{r=1}^\infty \tilde\sigma_r \widetilde{\ope A}_r \ct \widetilde{\ope B}_r$, respectively. Let $\ope C_R$ and $\widetilde{\ope C}_R$ be the best $R$-separable approximations of $\ope C$ and $\widetilde{\ope C}$, respectively. Then,
\[
\verti{\ope C_R - \widetilde{\ope C}_R}_{2} \le \bigg\{4\sqrt{2}\Big(\verti{\ope C}_{2} + \verti{\widetilde{\ope C}}_{2}\Big) \sum_{r=1}^R \frac{\sigma_r}{\alpha_r} + 1\bigg\} \verti{\ope C - \widetilde{\ope C}}_{2},
\]
where $\alpha_r = \min\{\sigma_{r-1}^2 - \sigma_r^2, \sigma_r^2 - \sigma_{r+1}^2\}$.
\end{lemma}
\begin{proof}
Note that $\ope C_R$ and $\widetilde{\ope C}_R$ have separable expansions $\ope C_R=\sum_{i=1}^R \sigma_i \ope A_i \ct \ope B_i$ and $\widetilde{\ope C}_R=\sum_{i=1}^R \widetilde\sigma_i \widetilde{\ope A}_i \ct \widetilde{\ope B}_i$, respectively. W.l.o.g. we assume that $\langle \ope A_i,\widetilde{\ope A}_i \rangle_{{\spa S}_2(\spa H_1)}, \langle \ope B_i,\widetilde{\ope B}_i \rangle_{{\spa S}_2(\spa H_2)} \ge 0$ for every $i=1,\ldots,R$ (if not, one can change the sign of $\widetilde{\ope A}_i$ or $\widetilde{\ope B}_i$, and adjust the sign of $\widetilde\sigma_i$ as required). Thus,
\begin{align}\label{eq:perturbation_R-sep_Eq1}
\verti{\ope C_R - \widetilde{\ope C}_R}_{2} &= \verti{\sum_{i=1}^R \sigma_i \ope A_i \ct \ope B_i - \sum_{i=1}^R \widetilde\sigma_i \widetilde{\ope A}_i \ct \widetilde{\ope B}_i}_{2} \nonumber \\
&= \verti{\sum_{i=1}^R \sigma_i \Big(\ope A_i \ct \ope B_i - \widetilde{\ope A}_i \ct \widetilde{\ope B}_i\Big) + \sum_{i=1}^R \big(\sigma_i - \widetilde\sigma_i) \widetilde{\ope A}_i \ct \widetilde{\ope B}_i}_{2} \nonumber \\
&\le \verti{\sum_{i=1}^R \sigma_i \Big(\ope A_i \ct \ope B_i - \widetilde{\ope A}_i \ct \widetilde{\ope B}_i\Big)}_{2} + \verti{\sum_{i=1}^R \big(\sigma_i - \widetilde\sigma_i) \widetilde{\ope A}_i \ct \widetilde{\ope B}_i}_{2}.
\end{align}
Now, $\verti{\sum_{i=1}^R \big(\sigma_i - \widetilde\sigma_i\big) \widetilde{\ope A}_i \ct \widetilde{\ope B}_i}_{2}^2 = \sum_{i=1}^R \big(\sigma_i - \widetilde\sigma_i\big)^2 \le \sum_{i=1}^\infty \big(\sigma_i - \widetilde\sigma_i\big)^2$ which, by von Neumann's trace inequality, is bounded by $\verti{\ope C - \widetilde{\ope C}}_{2}^2$ \citep[][Theorem~4.5.3]{hsing2015}. On the other hand,
\begin{align*}
\verti{\sum_{i=1}^R \sigma_i \Big(\ope A_i \ct \ope B_i - \widetilde{\ope A}_i \ct \widetilde{\ope B}_i\Big)}_{2} &\le \sum_{i=1}^R \sigma_i \verti{\ope A_i \ct \ope B_i - \widetilde{\ope A}_i \ct \widetilde{\ope B}_i}_{2} \\
&= \sum_{i=1}^R \sigma_i \verti{\ope A_i \ct \Big(\ope B_i - \widetilde{\ope B}_i\Big) + \Big(\ope A_i - \widetilde{\ope A}_i\Big) \ct \widetilde{\ope B}_i}_{2} \\
&\le \sum_{i=1}^R \sigma_i \bigg(\verti{\ope A_i - \widetilde{\ope A}_i}_{{\spa S}_2(\spa H_1)} + \verti{\ope B_i - \widetilde{\ope B}_i}_{{\spa S}_2(\spa H_2)}\bigg) \\
&\le 4\sqrt{2} \verti{\ope C - \widetilde{\ope C}}_{2} \bigg(\verti{\ope C}_{2} + \verti{\widetilde{\ope C}}_{2}\bigg) \sum_{i=1}^R \frac{\sigma_i}{\alpha_i}
\end{align*}
where the last inequality follows by part (b) of Lemma~\ref{lemma:perturbation_bound}. The proof is complete upon using these inequalities in conjunction with \eqref{eq:perturbation_R-sep_Eq1}.
\end{proof}

\subsection{Proofs of Results in Section~\ref{sec:analysis} and related discussions}\label{append:proofs}

\subsubsection{Fully Observed Functional Data}
\begin{proof}[Proof of Theorem~\ref{thm:fully_observed}]
To bound the error of the estimator, we use the following {\em bias-variance-type} decomposition
\begin{equation}\label{eq:full_observation_bias-variance}
\verti{\Chat_{R,N} - \ope C}_{2} \le \verti{\Chat_{R,N} - \ope C_R}_{2} + \verti{\ope C_R - \ope C}_{2},
\end{equation}
where $\ope C_R$ is the best $R$-separable approximation of $\ope C$. If $\ope C$ has separable expansion $\ope C=\sum_{i=1}^\infty \sigma_i \ope A_i \ct \ope B_i$, then $\ope C_R$ has separable expansion $\ope C_R=\sum_{i=1}^R \sigma_i \ope A_i \ct \ope B_i$, and
\begin{equation}\label{eq:full_observation_bias}
\verti{\ope C - \ope C_R}_{2}^2 = \sum_{i=R+1}^\infty \sigma_i^2.
\end{equation}
For the first part, we use the perturbation bound from Lemma~\ref{lemma:perturbation_R-sep} (see Appendix~\ref{append:perturbation}), to get
\begin{equation}\label{eq:full_observation_variance}
\verti{\Chat_{R,N} - \ope C_R}_{2} \le \bigg\{4\sqrt{2}\Big(\verti{\Chat_N}_{2} + \verti{\ope C}_{2}\Big) \sum_{r=1}^R \frac{\sigma_r}{\alpha_r} + 1\bigg\} \verti{\Chat_N - \ope C}_{2},
\end{equation}
where $\alpha_r = \min\{\sigma_{r-1}^2 - \sigma_r^2, \sigma_r^2 - \sigma_{r+1}^2\}$. 

Since $\E(\|\ope X\|^4)$ is finite, $\verti{\Chat_N - \ope C}_{2} = {\cal O}_{\Prob}(N^{-1/2})$ and $\verti{\Chat_N}_{2} = \verti{\ope C}_{2} + {\scriptstyle\cal O}_{\Prob}(1)$. Using these in \eqref{eq:full_observation_variance}, we get that
\begin{equation}\label{eq:full_observation_variance_bound}
\verti{\Chat_{R,N} - \ope C_R}_{2} = {\cal O}_{\Prob}\bigg(\frac{a_R}{\sqrt{N}}\bigg),
\end{equation}
where $a_R = \verti{\ope C}_2 \sum_{i=1}^R (\sigma_i/\alpha_i)$. The theorem follows by combining \eqref{eq:full_observation_bias} and \eqref{eq:full_observation_variance_bound} in \eqref{eq:full_observation_bias-variance}.
\end{proof}

\subsubsection{Discussion on the Convergence Rates}\label{subsec:optimal_R_convex}
Theorem~\ref{thm:fully_observed} shows that the convergence rate of the estimator depends on the decay of the separable component scores. Here we discuss about the rate under some special decay structures, in particular that of convexity. Assume that the separable component scores of $\ope C$ are convex, in the sense that the linearly interpolated scree plot $t \mapsto \sigma_t$  is convex, where $\sigma_t=(\lceil t \rceil-t)\sigma_{\lfloor t\rfloor} + (t-\lfloor t\rfloor)\sigma_{\lceil t \rceil}$ whenever $t$ is not an integer \citep{jirak2016}. Then it also holds that $t \mapsto \sigma_t^2$ is convex. Now, following \citet[][Eq.~(7.25)]{jirak2016}, we get that for $j > k$, 
\[
k\sigma_k \ge j\sigma_j \text{ and } \sigma_k - \sigma_j \ge (1 - k/j) \sigma_k.
\]
Also, because of the convexity of $t \mapsto \sigma_t^2$, it follows that $\sigma_{i-1}^2 - \sigma_i^2 \ge \sigma_i^2 - \sigma_{i+1}^2$ for every $i$, showing 
\[
\alpha_i = \min\{\sigma_{i-1}^2 - \sigma_i^2, \sigma_i^2 - \sigma_{i+1}^2\} = \sigma_i^2 - \sigma_{i+1}^2.
\]
So,
\[
\frac{\sigma_i}{\alpha_i} = \frac{\sigma_i}{\sigma_i^2 - \sigma_{i+1}^2} = \frac{\sigma_i}{(\sigma_i - \sigma_{i+1})\,(\sigma_i + \sigma_{i+1})} \ge \frac{1}{2\sigma_1} \frac{\sigma_i}{(\sigma_i - \sigma_{i+1})}.
\]
Now, for $0 < t < 1$, $(1-t)^{-1} > 1+t$. Since $\sigma_i > \sigma_{i+1}$ for $i=1,\ldots,R$, this shows that
\[
\sum_{i=1}^R \frac{\sigma_i}{\sigma_i - \sigma_{i+1}} = \sum_{i=1}^R \frac{1}{1 - \frac{\sigma_{i+1}}{\sigma_i}} > \sum_{i=1}^R \bigg(1 + \frac{\sigma_{i+1}}{\sigma_i}\bigg) > R.
\]
So, $a_R \gtrsim R$. Again, since $t \mapsto \sigma_t^2$ is convex, $\sigma_i^2 - \sigma_{i+1}^2 \ge \sigma_i^2/(i+1)$. Using this we get
\[
\begin{split}
a_R &\propto \sum_{i=1}^R \frac{\sigma_i}{\sigma_i^2 - \sigma_{i+1}^2} \le \sum_{i=1}^R \frac{\sigma_i^2}{\sigma_i(\sigma_i^2 - \sigma_{i+1}^2)} \le \sum_{i=1}^R \frac{i+1}{\sigma_i} \\
&\le \frac{1}{R\sigma_R} \sum_{i=1}^R i(i+1) = \frac{1}{\sigma_R} \frac{(R+1)(R+2)}{3} \asymp \frac{R^2}{\sigma_R},
\end{split}
\]
where we have used $i\sigma_i \ge R\sigma_R$ on the fourth step. It follows that, $a_R = {\cal O}(R^2/\sigma_R)$. Also, following \citet[][Eq.~(7.25)]{jirak2016}, we have $\sum_{i > R} \sigma_i^2 \le (R+1) \sigma_R^2$. Thus, from Theorem~\ref{thm:fully_observed},
\[
\verti{\Chat_{R,N} - \ope C}_{2} = {\cal O}(\sqrt{R} \sigma_R) + {\cal O}_{\Prob}\bigg(\frac{R^2}{\sigma_R\sqrt{N}}\bigg).
\]
On the other hand, $R\sigma_R \le \sigma_1$ implies that $\sqrt{R}\sigma_R \le \sigma_1/\sqrt{R}$. This finally shows that
\[
\verti{\Chat_{R,N} - \ope C}_{2} = {\cal O}\bigg(\frac{1}{\sqrt{R}}\bigg) + {\cal O}_{\Prob}\bigg(\frac{R^2}{\sigma_R\sqrt{N}}\bigg).
\]
So, for consistency, we need $R = R_N \to \infty$ as $N \to \infty$ while $\sigma_{R_N}^{-1} R_N^2 = {\scriptstyle\O}(\sqrt{N})$. The optimal rate of $R$ is obtained by solving $\sigma_{R_N}^{-4}R_N^3 \asymp N$. Clearly, the rate of decay of the separable component scores plays an important role in determining the {\em admissible} and {\em optimal} rates of $R$. For instance, if the scores have an exponential decay, i.e., $\sigma_R \sim R^{-\tau}$ with $\tau > 1$, we need $R_N = {\scriptstyle\O}(N^{1/(2\tau+4)})$ for consistency. The optimal rate is achieved by taking $R_N \asymp N^{1/(4\tau+3)}$, which gives
\[
\verti{\Chat_{R,N} - \ope C}_{2} = {\cal O}_{\Prob}\Big(N^{-\frac{2\tau-1}{4\tau+3}}\Big).
\]
The derived rates show a trade-off between the number of estimated components and the error. While the optimal rate for $R_N$ is a decreasing function of $\tau$, the rate of convergence of the error is an increasing function. In particular, if the scores decay slowly (i.e., $\tau$ is close to $1$), then we can estimate a relatively large number of components in the separable expansion, but this will likely not lead to a lower estimation error (since the scores which are cut off are still substantial). On the other hand, when we have a fast decay in the scores (i.e., $\tau$ is large), we can estimate a rather small number of components in the separable expansion, but with much better precision. 

Similar rates can be derived assuming polynomial decay of the scores, i.e. when $\sigma_R \sim R^\tau \rho^{-R}$ with $0<\rho<1, \tau \in \mathbb{R}$. In this case, consistency is achieved when $R_N^{2-\tau}\rho^{R_N} = {\scriptstyle \O}(\sqrt{N})$, while to obtain the optimal rate, one needs to solve $R_N^{3-4\tau}\rho^{4R_N} \asymp N$. Thus, in the case of polynomial decay of the scores (which is considerably slower than the exponential decay), we cannot expect to reliably estimate more than $\log N$ many components in the separable expansion.

\subsubsection{Data Observed on a Regular Grid}

Next, we prove Theorem~\ref{thm:discrete_observation}. Before doing that, we introduce a notational convention: when the operator $\ope A$ has a kernel that is piecewise constant on the rectangles $\{I_{i,j}^K\}$ (i.e. a ``pixelated kernel"), we will write $\|\mat{A}\|_{\rm F}$ for the Frobenius norm of the corresponding tensor of pixel coefficients. This is proportional to the Hilbert-Schmidt norm $\verti{\ope A}_{2}$ of $\ope A$. We summarise this in the lemma below, whose straightforward proof we omit (see \citealp[][Lemma~3]{masak2019} for more details).
\begin{lemma}\label{lemma:HS_to_Frobenius}
Let $\ope A$ be an operator with a pixelated kernel 
\[
a(t,s,t^\prime,s^\prime) = \sum_{i=1}^{K_1}\sum_{j=1}^{K_2}\sum_{k=1}^{K_1}\sum_{l=1}^{K_2} \mat{A}[i,j,k,l] \, \1\{(t,s) \in I_{i,j}^K, (t^\prime,s^\prime) \in I_{k,l}^K\},
\]
where $\mat{A} = (\mat{A}[i,j,k,l])_{i,j,k,l} \in \mathbb{R}^{K_1 \times K_2 \times K_1 \times K_2}$ is the tensorized version of $\ope A$. Then,
\[
\verti{\ope A}_{2} = \frac{1}{K_1K_2} \|\mat{A}\|_{\rm F},
\] 
where $\|\mat{A}\|_{\rm F} = \sqrt{\sum_{i=1}^{K_1}\sum_{j=1}^{K_2}\sum_{k=1}^{K_1}\sum_{l=1}^{K_2} \mat{A}^2[i,j,k,l]}$ is the Frobenius norm of $\mat{A}$. 
\end{lemma}

\smallskip

\begin{proof}[Proof of Theorem~\ref{thm:discrete_observation}]
We decompose the error of our estimator as
\begin{equation}\label{eq:discrete_bias_variance}
\verti{\Chat_{R,N}^K - \ope C}_{2} \le  \verti{\Chat_{R,N}^K - \ope C_R}_{2} + \verti{\ope C_R - \ope C}_{2}.
\end{equation}
The second term $\verti{\ope C_R - \ope C}_{2}$ equals $\sqrt{\sum_{r=R+1}^\infty \sigma_r^2}$ (see \eqref{eq:full_observation_bias}). For the first term, we observe that $\Chat_{R,N}^K$ and $\ope C_R$ are the best $R$-separable approximations of $\Chat_N^K$ and $\ope C$, respectively. So, using Lemma~\ref{lemma:perturbation_R-sep}, we get
\begin{equation}\label{eq:discrete_variance}
\verti{\Chat_{R,N}^K - \ope C_R}_{2} \le \bigg\{4\sqrt{2}\Big(\verti{\ope C}_{2} + \verti{\Chat_N^K}_{2}\Big) \sum_{r=1}^R \frac{\sigma_r}{\alpha_r} + 1\bigg\} \verti{\Chat_N^K - \ope C}_{2},
\end{equation}
with $\alpha_r = \min\{\sigma_{r-1}^2 - \sigma_r^2, \sigma_r^2 - \sigma_{r+1}^2\}$. Next, we derive bounds on $\verti{\Chat_N^K - \ope C}_{2}$. Note that $\widehat{\ope C}_N^K$ is the empirical covariance of $\ope Y_1^K,\ldots,\ope Y_N^K$. Under the assumptions of the theorem, $\ope Y_1^K,\ldots,\ope Y_N^K$ are i.i.d.\ with mean zero and covariance $\widetilde{\ope C}^K = \E(\ope Y_1^K \otimes \ope Y_1^K)$. Let $\ope X_n^K = (\fun X_n^K(t,s): t \in T,s \in S)$ with $\fun X_n^K(t,s) = \sum_{i=1}^{K_1}\sum_{j=1}^{K_2} X_n^K[i,j]\,\1\big\{(t,s) \in I_{i,j}^K\big\}$ be the pixel-wise continuations of $\mat X_n^K$. It follows that $\ope X_1^K,\ldots,\ope X_N^K$ are i.i.d.\ with mean zero and covariance $\ope C^K = \E(\ope X_1^K \otimes \ope X_1^K)$. We use the general bound
\begin{equation}\label{eq:discrete_variance_decomposition}
\verti{\Chat_N^K - \ope C}_{2} \le \verti{\Chat_N^K - \widetilde{\ope C}^K}_{2} + \verti{\widetilde{\ope C}^K - \ope C^K}_{2} + \verti{\ope C^K - \ope C}_2.
\end{equation}
Note that $\widetilde{\ope C}^K$ and $\ope C^K$ are pixelated operators with discrete versions $\widetilde{\mat C}^K$ and $\mat C^K$, respectively, where $\widetilde{\mat C}^K[i,j,k,l] = \cov\big\{Y_1^K[i,j],Y_1^K[k,l]\big\}$ and $\mat C^K[i,j,k,l] = \cov\big\{X_1^K[i,j],X_1^K[k,l]\big\}$. Also, under the assumptions of the theorem, $\widetilde{\mat C}^K[i,j,k,l] - \mat C^K[i,j,k,l] = \sigma^2\,\1\{(i,j)=(k,l)\}$. Thus, by Lemma~\ref{lemma:HS_to_Frobenius}, we get
\begin{equation}\label{eq:discrete_variance_error}
\verti{\widetilde{\ope C}^K - \ope C^K}_2 = \frac{1}{K_1K_2} \Big\|\widetilde{\mat C}^K - \mat C^K\Big\|_{\rm F} = \frac{\sigma^2}{\sqrt{K_1K_2}}.
\end{equation}
We will now bound the remaining two terms in \eqref{eq:discrete_variance_decomposition}, separately under (S1) and (S2).

\medskip

Under (S1), it is easy to see that $\ope C^K$ is the integral operator with kernel
\begin{equation*}
c^K(t,s,t^\prime,s^\prime) = \sum_{i=1}^{K_1}\sum_{j=1}^{K_2}\sum_{k=1}^{K_1}\sum_{l=1}^{K_2} c(t_i^{K_1},s_j^{K_2},t_k^{K_1},s_l^{K_2})\, \1\{(t,s) \in I_{i,j}^K, (t^\prime,s^\prime) \in I_{k,l}^K\}.
\end{equation*}
Using this, we get
\begin{align}\label{eq:pixel_bias_M1}
\verti{\ope C^K - \ope C}_{2}^2 &= \iiiint \big\{c^K(t,s,t^\prime,s^\prime) - c(t,s,t^\prime,s^\prime)\big\}^2 dt ds dt^\prime ds^\prime \nonumber\\
&= \sum_{i,k=1}^{K_1}\sum_{j,l=1}^{K_2} \iiiint_{I_{i,j}^K \times I_{k,l}^K} \big\{c(t_i^{K_1},s_j^{K_2},t_k^{K_1},s_l^{K_2}) - c(t,s,t^\prime,s^\prime)\big\}^2 dt ds dt^\prime d s^\prime.
\end{align}
Because of the Lipschitz condition, for $(t,s) \in I_{i,j}^K,(t^\prime,s^\prime) \in I_{k,l}^K$,
\begin{align*}
\big|c(t_i^{K_1},s_j^{K_2},t_k^{K_1},s_l^{K_2}) &- c(t,s,t^\prime,s^\prime)\big|^2  \\
& \le L^2 \{(t-t_i^{K_1})^2 + (s-s_j^{K_2})^2 + (t^\prime - t_k^{K_1})^2 + (s^\prime - s_l^{K_2})^2\} \\
& \le L^2 \bigg(\frac{1}{K_1^2} + \frac{1}{K_2^2} + \frac{1}{K_1^2} + \frac{1}{K_2^2}\bigg) = 2L^2\bigg(\frac{1}{K_1^2} + \frac{1}{K_2^2}\bigg).
\end{align*}
Plugging this into \eqref{eq:pixel_bias_M1} yields
\begin{align}\label{eq:pixel_bias_bound_M1}
\verti{\ope C^K - \ope C}_{2}^2 \le 2L^2\bigg(\frac{1}{K_1^2} + \frac{1}{K_2^2}\bigg).
\end{align}

For the first part in \eqref{eq:discrete_variance_decomposition}, we observe that $\Chat_N^K$ is the sample covariance of $\ope Y_1^K,\ldots,\ope Y_N^K$, which are i.i.d. with $\E(\ope Y_1^K) = 0$ and $\var(\ope Y_1^K) = \widetilde{\ope C}^K$. Also, $\Chat_N^K$ and $\widetilde{\ope C}^K$ are pixelated operators with discrete versions $\widehat{\mat C}_N^K$ and $\widetilde{\mat C}^K$, respectively, where $\widehat{\mat C}_N^K = N^{-1} \sum_{n=1}^N (\mat{X}_n^K - \bar{\mat{X}}_N^K) \otimes (\mat{X}_n^K - \bar{\mat{X}}_N^K)$ is the sample variance based on $\mat Y_1^K,\ldots,\mat Y_N^K$ and $\widetilde{\mat C}^K$ is as defined before. So, by Lemma~\ref{lemma:HS_to_Frobenius}, 
\begin{equation}\label{eq:pixel_HS_to_Frob}
\verti{\Chat_N^K - \ope C^K}_{2} = \frac{1}{K_1K_2} \left\|\widehat{\mat C}_N^K - \widetilde{\mat C}^K\right\|_{\rm F}.
\end{equation}
For the Frobenius norm, we write
\begin{align}\label{eq:pixel_Frob_bound}
\left\|\widehat{\mat C}_N^K - \widetilde{\mat C}^K\right\|_{\rm F} &=\left\|\frac{1}{N}\sum_{n=1}^N \mat Y_n^K \otimes \mat Y_n^K - \bar{\mat Y}_N^K \otimes \bar{\mat Y}_N^K - \widetilde{\mat C}^K\right\|_{\rm F} \nonumber\\
&\le \left\|\frac{1}{N}\sum_{n=1}^N \mat Y_n^K \otimes \mat Y_n^K - \widetilde{\mat C}^K\right\|_{\rm F} + \left\|\bar{\mat Y}_N^K \otimes \bar{\mat Y}_N^K\right\|_{\rm F}.
\end{align}
Now, $\left\|\bar{\mat Y}_N^K \otimes \bar{\mat Y}_N^K\right\|_{\rm F} = \left\|\bar{\mat Y}_N^K\right\|_{\rm F}^2 = \left\|N^{-1}\sum_{n=1}^N \mat Y_n^K\right\|_F^2$, where $\mat Y_n^K$ are i.i.d., zero-mean elements. So,
\[
\E\left\|\bar{\mat Y}_N^K \otimes \bar{\mat Y}_N^K\right\|_{\rm F} = \E\left\|\frac{1}{N} \sum_{n=1}^N \mat Y_n^K\right\|_{\rm F}^2 = \frac{1}{N}\E\left\|\mat Y_1^K\right\|_{\rm F}^2.
\]
Again, $\left\|\mat Y_1^K\right\|_{\rm F}^2 = \sum_{i=1}^{K_1}\sum_{j=1}^{K_2} \{Y_1^K[i,j]\}^2$, so
\[
\E\left\|\mat Y_1^K\right\|_{\rm F}^2 = \sum_{i=1}^{K_1}\sum_{j=1}^{K_2} \E\big(Y_1^K[i,j]\big)^2 = \sum_{i=1}^{K_1}\sum_{j=1}^{K_2} \var\big(Y_1^K[i,j]\big).
\]
By the independence of $\mat X_1^K$ and $\mat E_1^K$, it follows that
\begin{equation}\label{eq:discrete_measurement_var1}
\var\big(Y_1^K[i,j]\big) = \var\big(X_1^K[i,j]\big) + \var\big(E_1^K[i,j]\big) = \var\big(X_1^K[i,j]\big) + \sigma^2.
\end{equation}
Under the measurement scheme (S1), $X_1^K[i,j] = \fun X_1(t_i^{K_1},s_j^{K_2})$, so 
\[
\var\big(X_1^K[i,j]\big) = c(t_i^{K_1},s_j^{K_2},t_i^{K_1},s_j^{K_2}) \le \sup_{(t,s) \in [0,1]^2} c(t,s,t,s) =: S_1,
\]
where $S_1$ is finite since we assume that $\ope X$ has continuous sample paths. This shows that 
\begin{equation}\label{eq:pixel_Frob_bound_2}
\E\left\|\bar{\mat Y}_N^K \otimes \bar{\mat Y}_N^K\right\|_{\rm F} \le \frac{K_1K_2(S_1+\sigma^2)}{N}.
\end{equation}

Next, we define $\mat Z_n^K = \mat Y_n^K \otimes \mat Y_n^K - \widetilde{\mat C}^K$. Then, $\mat Z_1^K,\ldots,\mat Z_N^K$ are i.i.d., mean centered, which gives
\begin{align*}
\E\left\|\frac{1}{N}\sum_{n=1}^N \mat Y_n^K \otimes \mat Y_n^K - \widetilde{\mat C}^K\right\|_{\rm F}^2 = \E\left\|\frac{1}{N}\sum_{n=1}^N \mat Z_n^K\right\|_{\rm F}^2 = \frac{1}{N} \E\left\|\mat Z_1^K\right\|_{\rm F}^2.
\end{align*}
Now,
\[
\E\left\|\mat Z_1^K\right\|_{\rm F}^2 = \sum_{i=1}^{K_1}\sum_{j=1}^{K_2}\sum_{k=1}^{K_1}\sum_{l=1}^{K_2} \E\big(Y_1^K[i,j]Y_1^{K}[k,l] - \widetilde C^K[i,j,k,l]\big)^2
\]
For the summands, we have
\begin{align}\label{eq:discrete_measurement_var2}
&\E\Big\{\big(Y_1^K[i,j]Y_1^K[k,l] - \widetilde C^K[i,j,k,l]\big)^2\Big\} \nonumber\\
&=\E\Big\{\big(X_1^K[i,j]X_1^K[k,l] - C^K[i,j,k,l] + X_1^K[i,j]E_1^K[k,l] + E_1^K[i,j]X_1^K[k,l] \nonumber\\
&\kern45ex + E_1^K[i,j]E_1^K[k,l] - \sigma^2\1_{(i,j)=(k,l)}\big)^2\Big\} \nonumber\\
&\le 4\bigg[\E\Big\{\big(X_1^K[i,j]X_1^K[k,l] - C^K[i,j,k,l]\big)^2\Big\} + \E\Big\{\big(X_1^K[i,j]E_1^K[k,l]\big)^2\Big\} \nonumber\\
&\kern15ex + \E\Big\{\big(E_1^K[i,j]X_1^K[k,l]\big)^2\Big\} + \E\Big\{\big(E_1^K[i,j]E_1^K[k,l] - \sigma^2\1_{(i,j)=(k,l)}\big)^2\Big\}\bigg] \nonumber\\
&= 4\bigg[\var\big(X_1^K[i,j]X_1^K[k,l]\big) + \E\Big\{\big(X_1^K[i,j]\big)^2\Big\} \E\Big\{\big(E_1^K[k,l]\big)^2\Big\} \nonumber\\
&\kern15ex + \E\Big\{\big(E_1^K[i,j]\big)^2\Big\} \E\Big\{\big(X_1^K[k,l]\big)^2\Big\} + \var\big(E_1^K[i,j]E_1^K[k,l]\big)\bigg] \nonumber\\
&= 4\bigg[\var\big(X_1^K[i,j]X_1^K[k,l]\big) + \sigma^2 \var\big(X_1^K[i,j]\big) \\
&\kern15ex+ \sigma^2 \var\big(X_1^K[k,l]\big) + \var\big(E_1^K[i,j]E_1^K[k,l]\big)\bigg].
\end{align}
Now,
\begin{align*}
\begin{split}
\var\big(X_1^K[i,j]X_1^K[k,l]\big) &= \var\big(\fun X_1(t_i^{K_1},s_j^{K_2}) \fun X_1(t_k^{K_1}s_l^{K_2})\big) \\
&\le \sup_{(t,s,t^\prime,s^\prime) \in [0,1]^4} \var\big(\fun X_1(t,s)\fun X_1(t^\prime,s^\prime)\big) =:S_2,
\end{split}
\end{align*}
where $S_2$ is finite since $\ope X$ has continuous sample paths and finite fourth moment. On the other hand,
\begin{align*}
\var\big(E_1^K[i,j]E_1^K[k,l]\big) = \begin{cases}\gamma^2 - \sigma^4 & \text{if } (i,j) = (k,l) \\ \sigma^4 & \text{otherwise} \end{cases}.
\end{align*}
where $\gamma^2:=E(|E_1^K[i,j]|^4)<\infty$.
Combining these, we get that
\begin{align*}
&\sum_{i=1}^{K_1}\sum_{j=1}^{K_2}\sum_{k=1}^{K_1}\sum_{l=1}^{K_2} \E\Big\{\big(Y_1^K[i,j]Y_1^K[k,l] - \widetilde C^K[i,j,k,l]\big)^2\Big\} \\
&= 4\bigg[\sum_{i=1}^{K_1}\sum_{j=1}^{K_2}\sum_{k=1}^{K_1}\sum_{l=1}^{K_2} \var\big(X_1^K[i,j]X_1^K[k,l]\big) + K_1K_2\sigma^2\sum_{i=1}^{K_1}\sum_{j=1}^{K_2} \var\big(X_1^K[i,j]\big) \\
&\kern10ex+ K_1K_2\sigma^2\sum_{k=1}^{K_1}\sum_{l=1}^{K_2} \var\big(X_1^K[k,l]\big) + \sum_{i=1}^{K_1}\sum_{j=1}^{K_2}\sum_{k=1}^{K_1}\sum_{l=1}^{K_2} \var\big(E_1^K[i,j]E_1^K[k,l]\big)\bigg] \\
&\le 4\big\{K_1^2K_2^2S_2 + 2K_1^2K_2^2S_1\sigma^2 + K_1K_2(\gamma^2-\sigma^4) + (K_1^2K_2^2-K_1K_2)\sigma^4\big\} \\
&=4K_1^2K_2^2\bigg(S_2 + 2S_1\sigma^2 + \sigma^4 + \frac{\gamma^2-\sigma^4}{K_1K_2}\bigg).
\end{align*}
Thus,
\[
\E\left\|\frac{1}{N}\sum_{n=1}^N \mat Y_n^K \otimes \mat Y_n^K - \widetilde{\mat C}^K\right\|_{\rm F}^2 \le \frac{4K_1^2K_2^2}{N}\bigg(S_2 + 2S_1\sigma^2 + \sigma^4 + \frac{\gamma^2-\sigma^4}{K_1K_2}\bigg),
\]
which implies that
\begin{align}\label{eq:pixel_Frob_bound_1}
\E\left\|\frac{1}{N}\sum_{n=1}^N \mat Y_n^K \otimes \mat Y_n^K - \widetilde{\mat C}^K\right\|_{\rm F} &\le \sqrt{\E\left\|\frac{1}{N}\sum_{n=1}^N \mat Y_n^K \otimes \mat Y_n^K - \widetilde{\mat C}^K\right\|_{\rm F}^2} \nonumber\\
&\le \frac{2K_1K_2}{\sqrt{N}}\sqrt{S_2 + 2S_1\sigma^2 + \sigma^4 + \frac{\gamma^2-\sigma^4}{K_1K_2}}.
\end{align}
Combining \eqref{eq:pixel_Frob_bound}, \eqref{eq:pixel_Frob_bound_2} and \eqref{eq:pixel_Frob_bound_1}, we obtain
\[
\E\left\|\widehat{\mat C}_N^K - \widetilde{\mat C}^K\right\|_{\rm F} \le \frac{2K_1K_2}{\sqrt{N}}\sqrt{S_2 + 2S_1\sigma^2 + \sigma^4 + \frac{\gamma^2-\sigma^4}{K_1K_2}} + \frac{K_1K_2(S_1+\sigma^2)}{N}.
\]
Finally, \eqref{eq:pixel_HS_to_Frob} yields
\[
\E\verti{\Chat_N^K - \widetilde{\ope C}^K}_{2} \le \frac{2}{\sqrt{N}}\sqrt{S_2 + 2S_1\sigma^2 + \sigma^4 + \frac{\gamma^2-\sigma^4}{K_1K_2}} + \frac{S_1+\sigma^2}{N} = \O(N^{-1/2}),
\]
uniformly in $K_1,K_2$. This shows
\begin{equation}\label{eq:pixel_variance_bound_M1}
\verti{\Chat_N^K - \widetilde{\ope C}^K}_{2} = \O_{\Prob}(N^{-1/2}),
\end{equation}
and the ${\cal O}_{\Prob}$ term is uniform in $K_1,K_2$. Using \eqref{eq:discrete_variance_error}, \eqref{eq:pixel_bias_bound_M1} and \eqref{eq:pixel_variance_bound_M1} in \eqref{eq:discrete_variance_decomposition}, we have
\begin{equation}\label{eq:pixel_bound_M1}
\verti{\Chat_N^K - \ope C}_{2} = \O_{\Prob}(N^{-1/2}) + L \sqrt{\frac{2}{K_1^2}+\frac{2}{K_2^2}} + \frac{\sigma^2}{\sqrt{K_1K_2}},
\end{equation}
where the ${\cal O}_{\Prob}$ term is uniform in $K_1,K_2$.

Next, we consider the measurement scheme (S2). Observe that, under this scheme, $\ope C^K$ is the integral operator with kernel
\begin{align}\label{eq:M2_kernel}
c^K(t,s,t^\prime,s^\prime) &= \cov\big\{\fun X_1^K(t,s),\fun X_1^K(t^\prime,s^\prime)\big\} \\
&= \sum_{i=1}^{K_1}\sum_{j=1}^{K_2}\sum_{k=1}^{K_1}\sum_{l=1}^{K_2} \bar c^K[i,j,k,l]\, \1\{(t,s) \in I_{i,j}^K, (t^\prime,s^\prime) \in I_{k,l}^K\}, 
\end{align}
where
\begin{equation*}
\bar c^K[i,j,k,l] = \frac{1}{|I_{i,j}^K|\,|I_{k,l}^K|} \iint_{I_{i,j}^K \times I_{k,l}^K} c(u,v,u^\prime,v^\prime) du dv du^\prime dv^\prime.
\end{equation*}
Now, as in \eqref{eq:pixel_bias_M1}, we get
\begin{align}\label{eq:pixel_bias_M2}
\verti{\ope C^K - \ope C}_{2}^2 &= \iiiint \big\{c^K(t,s,t^\prime,s^\prime) - c(t,s,t^\prime,s^\prime)\big\}^2 dt ds dt^\prime ds^\prime \nonumber\\
&= \sum_{i=1}^{K_1}\sum_{j=1}^{K_2}\sum_{k=1}^{K_1}\sum_{l=1}^{K_2} \iint_{I_{i,j}^K \times I_{k,l}^K} \big\{\bar c^K[i,j,k,l] - c(t,s,t^\prime,s^\prime)\big\}^2 dt ds dt^\prime ds^\prime.
\end{align}
Using \eqref{eq:M2_kernel} and the Lipschitz condition on $c$, given $(t,s) \in I_{i,j}^K, (t^\prime,s^\prime) \in I_{k,l}^K$, one has
\begin{align*}
\big|\bar c^K[i,j,k,l] &- c(t,s,t^\prime,s^\prime)\big| \\
&= \left|\frac{1}{|I_{i,j}^K|\,|I_{k,l}^K|} \iiiint_{I_{i,j}^K \times I_{k,l}^K} \big\{c(u,v,u^\prime,v^\prime) - c(t,s,t^\prime,s^\prime)\big\} du dv du^\prime dv^\prime \right|\\
&\le \frac{1}{|I_{i,j}^K|\,|I_{k,l}^K|} \iiiint_{I_{i,j}^K \times I_{k,l}^K} \big|c(u,v,u^\prime,v^\prime) - c(t,s,t^\prime,s^\prime)\big| du dv du^\prime dv^\prime \\
&\le \frac{1}{|I_{i,j}^K|\,|I_{k,l}^K|} \iiiint_{I_{i,j}^K \times I_{k,l}^K} L \sqrt{\frac{1}{K_1^2} + \frac{1}{K_2^2} + \frac{1}{K_1^2} + \frac{1}{K_2^2}} ~du dv du^\prime dv^\prime \\
&= L \sqrt{\frac{2}{K_1^2} + \frac{2}{K_2^2}}.
\end{align*}
Using this in \eqref{eq:pixel_bias_M2} yields
\begin{align}\label{eq:pixel_bias_bound_M2}
\verti{\ope C^K - \ope C}_{2}^2 \le 2L^2\bigg(\frac{1}{K_1^2} + \frac{1}{K_2^2}\bigg).
\end{align}
For $\verti{\Chat_K^N - \widetilde{\ope C}^K}_{2}$, we proceed similarly as under (S1). We need to bound $\E\left\|\mat Y_1^K\right\|_{\rm F}^2$ and $\E\left\|\mat{Z}_1^K\right\|_{\rm F}^2$ for which we need bounds on $\sum_{i=1}^{K_1}\sum_{j=1}^{K_2} \var\big(X_1^K[i,j]\big)$ and $\sum_{i=1}^{K_1}\sum_{j=1}^{K_2}\sum_{k=1}^{K_1}\sum_{l=1}^{K_2} \var\big(\mat{X}_1^K[i,j]\mat{X}_1^{K}[k,l]\big)$ (see \eqref{eq:discrete_measurement_var1}, \eqref{eq:discrete_measurement_var2}). Recall that under measurement scheme (S2),
\[
X_1^K[i,j] = \frac{1}{|I_{i,j}^K|} \int_{I_{i,j}^K} \fun X_1(t,s) dt ds = {\sqrt{K_1K_2}} \langle \ope X_1,g_{i,j}^K\rangle,
\]
where $g_{i,j}^K(t,s) = \sqrt{K_1K_2} \, \1\{(t,s) \in I_{i,j}^K\}$. So, $\var\big(X_1^K[i,j]\big) = K_1K_2 \var\big(\langle \ope X_1,g_{i,j}^K\rangle\big) = K_1K_2 \langle \ope Cg_{i,j}^K,g_{i,j}^K \rangle$. Observe that $g_{i,j}^K,i=1,\ldots,K_1,j=1,\ldots,K_2$ are orthonormal in ${\spa L}_2([0,1]^2)$ (i.e., $\langle g_{i,j}^K,g_{k,l}^K \rangle = \1\{(i,j) = (k,l)\}$). Thus, we can extend them to form a basis of ${\spa L}_2([0,1]^2)$ (c.f. the proof of Theorem~2 in \citealp{masak2019}). We again denote this extended basis by $(g_{i,j}^K)_{i,j=1}^\infty$. Since $\ope C$ is positive semi-definite, $\langle \ope Cg_{i,j}^K,g_{i,j}^K \rangle \ge 0$ for every $i,j \ge 1$. Thus,
\begin{align}\label{eq:pixel_M2_Frob_bound_2}
\sum_{i=1}^{K_1}\sum_{j=1}^{K_2} \var\big(X_1^K[i,j]\big) = K_1K_2 \sum_{i=1}^{K_1}\sum_{j=1}^{K_2} \langle \ope Cg_{i,j}^K,g_{i,j}^K \rangle \le K_1K_2 \sum_{i=1}^{\infty}\sum_{j=1}^{\infty} \langle \ope Cg_{i,j}^K,g_{i,j}^K \rangle = K_1K_2 \verti{\ope C}_1.
\end{align}
Again, $X_1[i,j]X_1[k,l] = K_1K_2 \langle \ope X_1 \otimes \ope X_1, g_{i,j}^K \otimes g_{k,l}^K \rangle$. This shows that 
\[
\var\big(X_1^K[i,j]X_1^{K}[k,l]\big) = K_1^2K_2^2 \langle \Gamma(g_{i,j}^K \otimes g_{k,l}^K),g_{i,j}^K \otimes g_{k,l}^K \rangle,
\]
where $\Gamma = \E(\ope X_1 \otimes \ope X_1 \otimes \ope X_1 \otimes \ope X_1) - \ope C \otimes \ope C$ is the covariance operator of $\ope X_1 \otimes \ope X_1$. Note that the assumption $\E(\|\ope X_1\|^4) < \infty$ ensures the existence of $\Gamma$, and further assures that $\Gamma$ is a trace-class operator. Since the $(g_{i,j}^K)$ are orthornormal in ${\spa L}_2([0,1]^2)$, the $\big(g_{i,j}^K \otimes g_{k,l}^K \big)_{i,j,k,l}$ are orthonormal in ${\spa L}_2([0,1]^4)$. So, we can extend them to a basis $(g_{i,j}^K \otimes g_{k,l}^K)_{i,j,k,l=1}^\infty$ of ${\spa L}_2([0,1]^4)$. Since $\Gamma$ is positive semi-definite
\begin{align}\label{eq:pixel_M2_Frob_bound_1}
\sum_{i=1}^{K_1}\sum_{j=1}^{K_2}\sum_{k=1}^{K_1}\sum_{l=1}^{K_2} \var\big(\mat{X}_1^K[i,j]\mat{X}_1^{K}[k,l]\big) &= K_1^2K_2^2 \sum_{i=1}^{K_1}\sum_{j=1}^{K_2}\sum_{k=1}^{K_1}\sum_{l=1}^{K_2} \langle \Gamma(g_{i,j}^K \otimes g_{k,l}^K), g_{i,j}^K \otimes g_{k,l}^K \rangle \nonumber\\
&\le K_1^2K_2^2 \sum_{i=1}^{\infty}\sum_{j=1}^{\infty}\sum_{k=1}^{\infty}\sum_{l=1}^{\infty} \langle \Gamma(g_{i,j}^K \otimes g_{k,l}^K), g_{i,j}^K \otimes g_{k,l}^K \rangle \nonumber\\
&= K_1^2 K_2^2 \verti{\Gamma}_{1}.
\end{align}
Using \eqref{eq:pixel_M2_Frob_bound_2} and \eqref{eq:pixel_M2_Frob_bound_1}, and proceeding as in the case of (S1), we get
\[
\E\verti{\Chat_N^K - \widetilde{\ope C}^K}_{2} \le \frac{2}{\sqrt{N}}\sqrt{\verti{\Gamma}_1 + 2\verti{\ope C}_1\sigma^2 + \sigma^4 + \frac{\gamma^2-\sigma^4}{K_1K_2}} + \frac{\verti{\ope C}_1+\sigma^2}{N},
\]
that is,
\begin{equation}\label{eq:pixel_variance_bound_M2}
\verti{\Chat_N^K - \widetilde{\ope C}^K}_{2} = \O_{\Prob}(N^{-1/2}),
\end{equation}
where the $\O_{\Prob}$ term is uniform in $K_1,K_2$. Finally, using \eqref{eq:discrete_variance_error}, \eqref{eq:pixel_bias_bound_M2} and \eqref{eq:pixel_variance_bound_M2} in \eqref{eq:discrete_variance_decomposition}, we get
\begin{equation}\label{eq:pixel_bound_M2}
\verti{\Chat_N^K - \ope C}_{2} = \O_{\Prob}(N^{-1/2}) + L \sqrt{\frac{2}{K_1^2}+\frac{2}{K_2^2}} + \frac{\sigma^2}{\sqrt{K_1K_2}},
\end{equation}
where the ${\cal O}_{\Prob}$ term is uniform in $K_1,K_2$. Observe that we obtain the same rate under both (S1) and (S2). Also observe that, under both the schemes,
\[
\verti{\Chat_N^K}_{2} = \verti{\ope C}_{2} + \O_{\Prob}(N^{-1/2}) + L \sqrt{\frac{2}{K_1^2}+\frac{2}{K_2^2}} + \frac{\sigma^2}{\sqrt{K_1K_2}}.
\]
Using these in \eqref{eq:discrete_variance}, we get
\begin{align*}
&\verti{\Chat_{R,N}^K - \ope C_R}_{2} \\
&\le \left\{4\sqrt{2}\left(\verti{\ope C}_{2} + \verti{\ope C}_{2} + \O_{\Prob}(N^{-1/2}) + L \sqrt{\frac{2}{K_1^2}+\frac{2}{K_2^2}} + \frac{\sigma^2}{\sqrt{K_1K_2}}\right) \sum_{r=1}^R \frac{\sigma_r}{\alpha_r} + 1\right\} \\
&\kern40ex\times \left\{\O_{\Prob}(N^{-1/2}) + L \sqrt{\frac{2}{K_1^2}+\frac{2}{K_2^2}} + \frac{\sigma^2}{\sqrt{K_1K_2}}\right\} \\
&= \O_{\Prob}\bigg(\frac{a_R}{\sqrt{N}}\bigg) + \big(16 a_R + \sqrt{2}\big)L \sqrt{\frac{1}{K_1^2} + \frac{1}{K_2^2}} + \frac{8\sqrt{2}L^2}{\verti{\ope C}_2}\bigg(\frac{1}{K_1^2} + \frac{1}{K_2^2}\bigg) a_R \\
&\kern20ex+\frac{(8\sqrt{2}a_R+1)\sigma^2}{\sqrt{K_1K_2}} + \frac{4\sqrt{2}\sigma^4a_R}{\verti{\ope C}_2K_1K_2} + \frac{16L\sigma^2a_R}{\verti{\ope C}_2\sqrt{K_1K_2}} \sqrt{\frac{1}{K_1^2}+\frac{1}{K_2^2}},
\end{align*}
where the $\O_{\Prob}$ term is uniform in $K_1,K_2$.
\end{proof}

\subsubsection{Unbalanced Design}

Finally, we consider the case where the surfaces are observed at irregular (and possibly different number of) locations. Here, we provide a proof of Theorem~\ref{thm:rate_of_convergence_smoother} in the main text.

\begin{proof}[Proof of Theorem~\ref{thm:rate_of_convergence_smoother}]
We start by using the inequality
\begin{equation}\label{eq:rate_smoother_eq1}
\verti{\widetilde{\ope C}_{R,N} - \ope C}_2 \le \verti{\widetilde{\ope C}_{R,N} - \ope C_R}_2 + \verti{\ope C_R - \ope C}_2.
\end{equation}
For the second part on the right-hand side, using \eqref{eq:full_observation_bias}, we get that $\verti{\ope C_R - \ope C}_2 = \big(\sum_{r=1}^R \sigma_r^2\big)^{1/2}$. For the first part, we use Lemma~\ref{lemma:perturbation_R-sep} to get
\begin{equation}\label{eq:rate_smoother_eq2}
\verti{\widetilde{\ope C}_{R,N} - \ope C_R}_2 \le \bigg\{4\sqrt{2}\big(\verti{\ope C}_2 + \verti{\widetilde{\ope C}_N}_2\big)\sum_{r=1}^R \frac{\sigma_r}{\alpha_r}+1\bigg\} \verti{\widetilde{\ope C}_N - \ope C}_2.
\end{equation}
Now,
\[
\verti{\widetilde{\ope C}_N - \ope C}_2 \le \verti{\widetilde{\ope C}_N - \widehat{\ope C}_N}_2 + \verti{\widehat{\ope C}_N - \ope C}_2.
\]
The first part on the right is $\O_{\Prob}(b_N)$ by the assumption of the theorem, while the second part is $\O_{\Prob}(N^{-1/2})$ since $\E(\|\ope X\|^4) < \infty$. Hence, we get $\verti{\widetilde{\ope C}_N - \ope C}_2 \le \O_{\Prob}(b_N)$, which also shows
\[
\verti{\widetilde{\ope C}_N}_2 \le \verti{\ope C}_2 + \verti{\widetilde{\ope C}_N - \ope C}_2 \le \verti{\ope C}_2 + \O_{\Prob}(b_N).
\]
Using these in \eqref{eq:rate_smoother_eq2}, we get $\verti{\widetilde{\ope C}_{R,N} - \ope C_R}_2 = \O_{\Prob}(a_R b_N)$ where $a_R = \verti{\ope C}_2 \sum_{r=1}^R \sigma_r/\alpha_r$. The result now follows by substituting this in \eqref{eq:rate_smoother_eq1}.
\end{proof}
%\soham{WE ARE ASSUMING THAT $b_N \ge N^{-1/2}$, WHICH I BELIEVE IS REASONABLE. MORE PRECISELY, THE RESULT WOULD HAVE $\tilde b_N = b_N \vee N^{-1/2}$ in place of $b_N$.}

\subsection{Setup for Simulations}\label{sec:setup_for_simulations}

The parametric covariance used in Example~\ref{ex:gneiting} and Section~\ref{sec:simulations} has kernel
\[
c(t,s,t',s') = \frac{\sigma^2}{(a^2|t-t'|^{2 \alpha}+1)^\tau} \exp\left( \frac{b^2 |s -s' |^{2 \gamma}}{ (a^2|t-t'|^{2 \alpha}+1)^{\beta\gamma} } \right).
\]
This covariance was introduced by \cite{gneiting2002}. Among the various parameters, $a$, resp. $b$, control the temporal, resp. spatial, domain scaling, and $\beta \in [0,1]$ controls the departure from separability with $\beta=0$ corresponding to a separable model. We fix $\beta=0.7$, which seems to be as a rather high degree of non-separability given the range of $\beta$, but one should note that non-separability is rather small for this model regardless of the choice of $\beta$ \citep{genton2007}. The remaining parameters are set $\alpha=\gamma=\sigma^2=\tau=1$, and the scaling parameters are set as $a=b=20$ on the $[0,1]$ interval, which corresponds to considering the domain as $[0,20]$ interval with $a=b=1$. We stretch the domain this way in order to strengthen the non-separability of the model. Even though the covariance is stationary, this fact is completely ignored. None of the methods presented in this paper make use of stationarity, and the reported results are not affected by it.

The non-parametric covariance, i.e. the second considered setup in our simulation study, is constructed as follows. Let $\{ \phi_1, \ldots, \phi_{50} \}$ denote the first fifty functions of the trigonometric basis, and let $J_1=\{ 1,4,\ldots, 49\}$, $J_2=\{ 2,5,\ldots, 50\}$ and $J_3=\{ 3,6,\ldots, 48\}$ be three index sets. $A_1$, $A_2$ and $A_3$ are constructed to have norm one, power decay of the eigenvalues, and the trigonometric system as their eigenbasis. However, the leading eigenfunctions in $\mat A_r$ are those trigonometric functions $\phi_l$ with indices in $J_r$, $r=1,\ldots,3$. $\mat B_1$, $\mat B_2$ and $\mat B_3$ are chosen in the same way, and the resulting covariance is chosen as
\[
\mat C = \sum_{r=1}^3 \sigma_r \mat A_r \ct \mat B_r
\]
with $\sigma_1 = 8$, $\sigma_2=4$ and $\sigma_3=2$. Note that these are not the separable component scores, since $\mat A$'s are not orthogonal, and the same is true for $\mat B$'s. In other words, the previous equation is \emph{not} a separable expansion. Still, the covariance is $3$-separable, i.e. it is a superposition of three separable terms. Moreover, all the separable terms have the same eigenfunctions, which are outer products of the trigonometric functions. Hence the covariance is by construction weakly separable \citep{lynch2018}.

In the main paper, we compare the proposed methodology with the weakly separable estimator of \cite{lynch2018}. Since the codes for the weakly separable estimator are only available in Matlab, we use our own \textsf{R} implementation.
A weakly separable estimator is obtained in two steps. The product eigenbasis is estimated from the data via partial tracing of \cite{aston2017}, and the eigenvalues are estimated in the subsequent step. \cite{lynch2018} propose to enforce low-rankness by only retaining a part of the basis, such that at least 95\% of variance is explained in both dimensions. We follow this suggestion.

In the following section, we work with a third covariance, which is also by construction $R$-separable and allows a perfect control of the separable component scores. Here we describe how exactly the $\mat A$'s, $\mat B$'s and $\sigma$'s of equation \eqref{eq:simul_cov} are chosen.

In the first half of Section~\ref{sec:further_simulations} leading up to Figure \ref{fig:rancov}, our goal is to show how the constant $a_R$ of Theorem~\ref{thm:fully_observed}, which is related to the decay of the separable component scores, affects the convergence speed, i.e. the ``variance'' part of the error. To this end, we generate a random ONB, and split this basis into $R=4$ sets of vectors, say $\mathcal{D}_1, \ldots, \mathcal{D}_4$. To generate $\mat A_r$ for $r=1,\ldots,4$, we use the vectors from $\mathcal{D}_r$ as the eigenvectors, while the non-zero eigenvalues are set as $|\mathcal{D}_r|,\ldots,1$. Hence $\mat A_r$ is singular with random eigenvectors and the eigenvalues, which are non-zero, are linearly decaying. The procedure is the same for $\mat B_1,\ldots,\mat B_4$. Note that the exact form of the covariances $\mat A$ and $\mat B$ or the fact that they are low-rank themselves is probably not affecting the results in any way. However, it is not easy to come up with a generation procedure for $\mat C$, which allows for both the control of its separable component scores, which we require, and an efficient data simulation. One basically needs $\mat A_1,\ldots,\mat A_4$ to be positive semi-definite and orthogonal at the same time. And the only way to achieve this is to have $\mat A_1,\ldots,\mat A_4$ low-rank with orthogonal eigenspaces.

In the second half of Section~\ref{sec:further_simulations} leading up to Figure~\ref{fig:inverse}, the situation is slightly simpler. In order to emulate more realistic inversion problems, we first simulate the data, and find the $R$-separable estimate. Since we are only interested in controlling the separable component scores after the data were already simulated, we have complete freedom in the choice of $\mat A$'s, $\mat B$'s and $\sigma$'s in \eqref{eq:simul_cov}. We set both $\mat A$'s and $\mat B$'s in the similar way as an output of a procedure describe as follows, with the exception of $\mat A_1$ and $\mat B_1$, which are chosen equal to the covariance of Brownian motion, standardized to have Hilbert-Schmidt norm equal to one. Since we keep all the $\sigma$'s equal to one, ordering of the covariance is immaterial. For a fixed $r=2,\ldots,R$, we generate $\mat A_r$ as follows (the procedure for $\mat B_r$ is again the same). We have a pre-specified list of functions, including polynomials of low order, trigonometric functions, and a B-spline basis. We choose a random number of these functions (evaluated on a grid), complement them with random vectors to span the whole space, and orthogonalize this collection to obtain an eigenbasis. The eigenvalues are chosen of to have a power decay with a randomly selected base. This procedure leads to a visually smooth covariance $\mat A_r$. 

Once the covariance is set, we generate the data as follows. For $n=1,\ldots,N$ and $r=1,\ldots,R$, we generate $\widetilde{\mat X}_n,r$ from the matrix-variate Gaussian distribution with mean zero and covariance $\mat A_r \ct \mat B_r$ \citep{gupta2000}. The sample is obtained as $\mat X_n = \widetilde{\mat X}_{n,1} + \ldots + \widetilde{\mat X}_{n,R}$ for $n=1,\ldots,N$. However, this is not possible for the parametric setup discussed in the main paper, since that covariance is not $R$-separable. For the parametric setup, we calculate the square-root of the covariance using the eigendecomposition, and then generate the data with this squre-root, which is computationally demanding.

\newpage
Finally, we note that the runtimes reported in the main paper were obtained on a personal laptop belonging to one of the authors, with Windows 10 (64-bit) operating system, Intel Core i7-7700HQ (2.8 GHz) processor, 16 GB RAM, and running \textsf{R} version 3.6.3 \citep{R}.

\subsection{Further Empirical Results}\label{sec:further_simulations}

In this section, we first report prediction performance in the parametric covariance setup. Then we move to the third covariance setup described in the previous section.

The relative prediction errors for the parametric covariance setup are reported in Table~\ref{tab:prediction}, because there are only small differences between different values of $R$ in this case. The latter is true due to the fact that the covariance is quite close to separable, but likely also due to the shape  of the parametric covariance. Nonetheless, the behaviour of the prediction error is qualitatively similar to Figure 3 of the main paper. Namely, for small values of $N$, the separable model works the best, and with increasing $N$ also the best choice of $R$ increases. The empirical covariance predicts again poorly.

\begin{table}[!t]
    \centering
    \caption{Relative prediction error in the parametric covariance setup for the $R$-separable estimators ($R=1,2,3$) and the empirical covariance.}
    \label{tab:prediction}
    \begin{tabular}{l ccc ccc ccc}
    \toprule
$N$ & 128 & 256 & 512 & 1024 & 2048 & 4096 & 8192 & 16384 & 32768 \\    
    \midrule
$R=1$ & \textbf{0.356} & 0.354 & 0.354 & 0.353 & 0.349 & 0.354 & 0.350 & 0.352 & 0.353   \\
$R=2$ & 0.363 & \textbf{0.349} & \textbf{0.345} & 0.342 & 0.337 & 0.341 & 0.338 & 0.340 & 0.340  \\
$R=3$ & 0.422 & 0.352 & 0.346 & \textbf{0.341} & \textbf{0.335} & \textbf{0.338} & \textbf{0.335} & \textbf{0.336} & \textbf{0.335}  \\
$EMP$ & 1.012 & 0.970 & 0.866 & 0.674 & 0.473 & 0.396 & 0.361 & 0.348 & 0.342  \\
\bottomrule
    \end{tabular}
\end{table}
  
Recall that the prediction results in the main paper concern the one-step-ahead prediction in both dimensions. But we can naturally predict also arbitrary missing patterns using the methodology of the main paper. Hence we show in Table~\ref{tab:prediction_random_pattern} and Figure~\ref{fig:prediction_random_pattern} results analogous to
Figure~\ref{fig:prediction} (left) of the main paper and Table~\ref{tab:prediction}, only with a random missing pattern of size 200 left out for prediction from every single surface. We can see that the results are qualitatively similar. Generally, it can be said that prediction of arbitrary patterns is statistically easier, since it corresponds more closely to interpolation.

\begin{table}[!t]
    \centering
    \caption{Relative prediction error in the parametric covariance setup with random missing patterns.}
    \label{tab:prediction_random_pattern}
    \begin{tabular}{l ccc ccc ccc}
    \toprule
$N$ & 128 & 256 & 512 & 1024 & 2048 & 4096 & 8192 & 16384 & 32768 \\    
    \midrule
$R=1$ & 0.123 & 0.120 & 0.121 & 0.120 & 0.120 & 0.120 & 0.120 & 0.120 & 0.120   \\
$R=2$ & \textbf{0.119} & 0.108 & 0.104 & 0.102 & 0.100 & 0.100 & 0.100 & 0.099 & 0.099   \\
$R=3$ & 0.133 & \textbf{0.107} & \textbf{0.103} & \textbf{0.100} & \textbf{0.096} & \textbf{0.093} & \textbf{0.091} & 0.090 & 0.090  \\
$EMP$ & 0.773 & 0.620 & 0.430 & 0.242 & 0.136 & 0.103 & 0.091 & \textbf{0.085} & \textbf{0.083} \\
\bottomrule
    \end{tabular}
\end{table}

\begin{figure}[!b]
  \centering
%   (a) $\widehat{R}=2$ & (b) $\widehat{R}=3$ \\
  \includegraphics[width=0.47\textwidth]{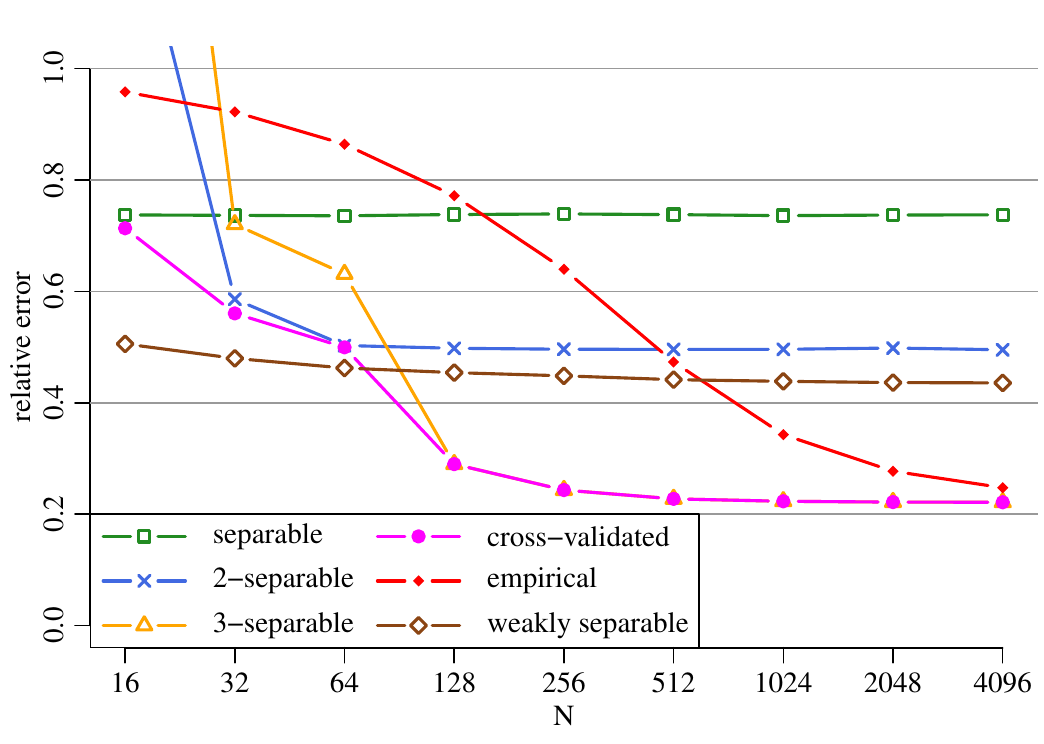} 
      \caption{Relative prediction error depending on the sample size $N$ in the non-parametric scenario with random missing patterns.}
    \label{fig:prediction_random_pattern} 
\end{figure}

In the remainder this section, we consider randomly generated $R$-separable covariances, i.e.
\begin{equation}\label{eq:simul_cov}
\mat C = \sum_{r=1}^R \sigma_r \mat A_r \ct \mat B_r .
\end{equation}
In our experience, the exact choice of the covariance (apart of the degree-of-separability and magnitudes of the scores) does not affect the result. Note that this is also supported by theory, where no specific assumptions on the covariance are made.

Firstly, we examine the role of decay of the scores $\{ \sigma_j \}$. We fix the true degree-of-separability at $R=4$ and choose $\sigma_r = \alpha^{R-r}$, $r=1,\ldots,R$, for different values of $\alpha$. Hence we have different polynomial decays for the scores. Higher $\alpha$'s correspond to faster decays and consequently to a higher value of $a_R$ from Theorem~\ref{thm:fully_observed}. Thus we expect slower convergence for higher values of $\alpha$.

% \begin{figure}[!t]
%   \centering
%   \begin{tabular}{ccc}
%   (a) $N=256$ & (b) $N=2048$ & (c) $N=16384$ \vspace*{-0.45mm}\\
%   \includegraphics[width=0.3\textwidth]{plots/scree_a.pdf} &
%   \includegraphics[width=0.3\textwidth]{plots/scree_b.pdf} &
%   \includegraphics[width=0.3\textwidth]{plots/scree_c.pdf}
%   \end{tabular}  
%   \caption{Scree plots for 3 different values of $N$ averaged over the 10 independent simulation runs.}
%     \label{fig:scree} 
% \end{figure}

Figure~\ref{fig:rancov} shows the results for $\widehat{R}=2$ and $\widehat{R}=3$ (i.e. for a wrongly chosen degree-of-separability, since the truth is $R=4$). The sample size $N$ is varied and the grid size is fixed again as $K=50$. To be able to visually compare the speed of convergence, we removed the bias
\[
\left(\sum_{r=R+1}^\infty \sigma_r^2 / \sum_{r=1}^\infty \sigma_r^2\right)^{1/2}
\]
from all the errors. For example, for $\widehat{R}=3$, the bias is equal to $\sigma_4$ (since the $\mat C$ is standardized to have norm equal to one). However, $\sigma_4$ varies for different $\alpha$'s, so we opt to remove $\sigma_4$ from the error corresponding to all the $\alpha$'s, in order for the curves in Figure~\ref{fig:rancov} to depict only the variance converging to zero. As expected, the convergence is faster for smaller $\alpha$'s corresponding to a slower decay of the separable component scores.

One can also notice certain transitions in Figure~\ref{fig:rancov}. For $\widehat{R}=2$ and $\alpha=6$, the drop in error between sample sizes $N=128$ and $N=256$ clearly stands out in the figure. This is because when $\alpha=6$, the scores decay so rapidly that the sample size $N=128$ is not enough for the second principal separable component to be estimated reliably. The situation is similar to Figure~\ref{fig:gneiting}, but since here the bias is subtracted from the error, choosing a higher $\widehat{R}$ than we can afford to estimate is even more striking. A similar behavior can be observed for multiple curves in Figure~\ref{fig:rancov} (right), when $\widehat{R}=3$. For example, one can observe an ``elbow'' at $N=512$ for the relatively slow decay of $\alpha=3$. This ``elbow'' is present because for smaller sample sizes, a smaller value of $\widehat{R}$ would have been better. From $N=512$ onwards, all 3 separable components are estimated reliably and the variance decays rather slowly and smoothly. This ``elbow'' exists for $\alpha=4,5$ as well, but manifests ``later" in terms of $N$. Finally, for $\alpha=6$, one can actually observe 3 different modes of convergence in Figure~\ref{fig:rancov} (right): before $N=256$, the degree-of-separability is overestimated by 2; between $N=256$ and $N=4096$, it is overestimated by 1; and it seems that for a larger $N$, the curve would finally enter the slowly converging mode.

\begin{figure}[!t]
  \centering
  \begin{tabular}{cc}
%   (a) $\widehat{R}=2$ & (b) $\widehat{R}=3$ \\
  \includegraphics[width=0.47\textwidth]{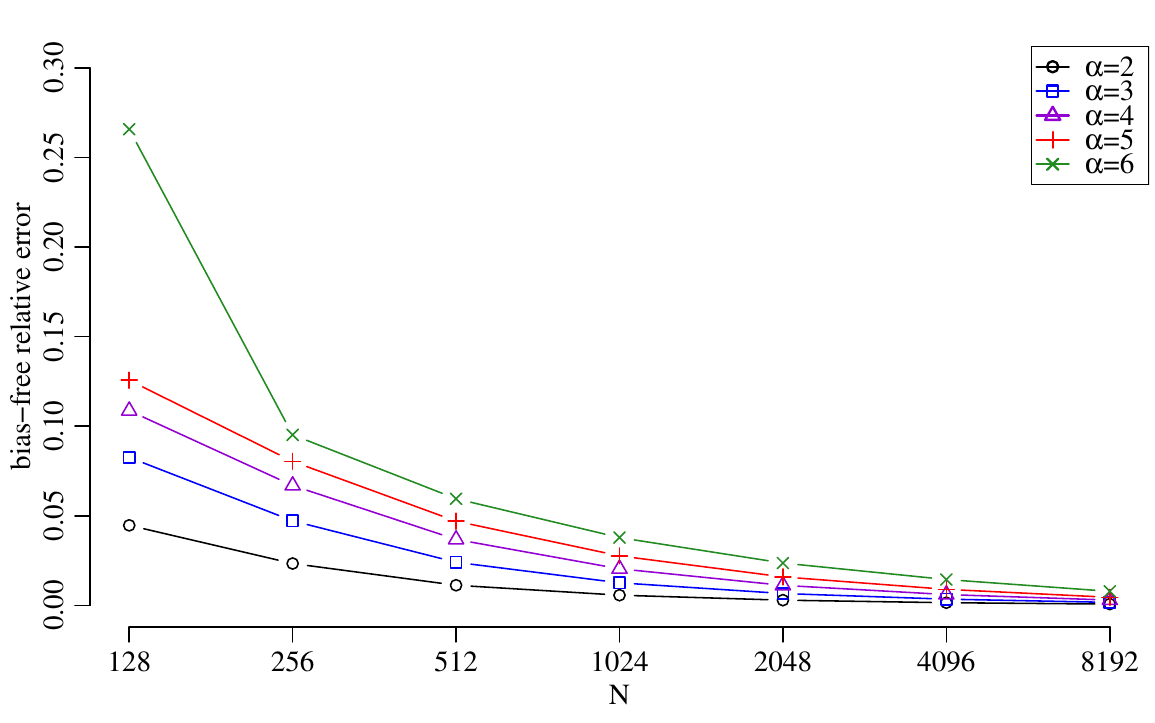} &
  \includegraphics[width=0.47\textwidth]{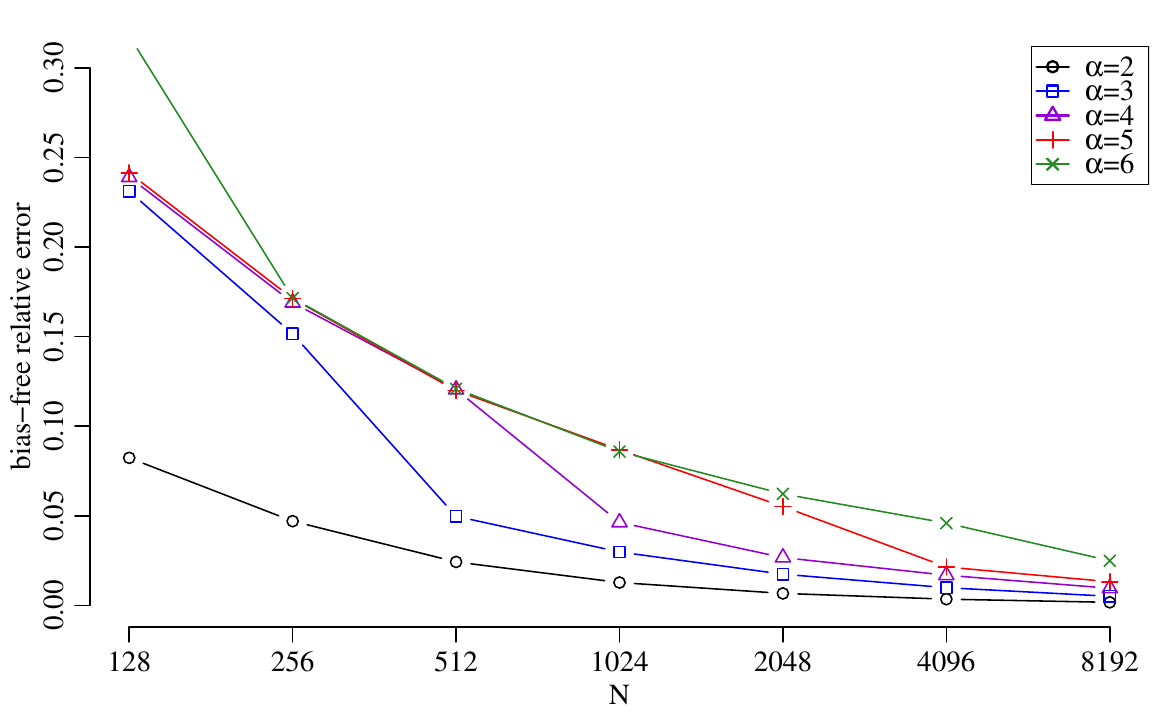}
  \end{tabular}  
      \caption{Relative estimation error curves for $\widehat{R}=2$ (left) and $\widehat{R}=3$ (right) when the true degree-of-separability is $R=4$. The reported errors are bias-free; bias was subtracted to make apparent the different speed of convergence of the variance to zero.}
    \label{fig:rancov} 
\end{figure}

In the second part of this section, we display the numerical performance of the inversion algorithm of Section~\ref{sec:inversion}. The number of observations and the degree-of-separability will be fixed now as $N=500$ and $R=5$. For different values of $K$, we find the estimator $\widehat{\mat C}_{R,N} = \sum_{r=1}^R \widehat{\sigma}_r \widehat{\mat A}_r \ct \widehat{\mat B}_r$.

We generate a random $\mat X \in \R^{K \times K}$, and calculate $\mat Y = \widehat{\mat C}_{R,N}^+ \mat X$, where $\widehat{\mat C}_{R,N}^+$ is a regularized estimator of Section~\ref{sec:inversion}. Then we use the inversion algorithm to recover $\mat X$ from the knowledge of $\widehat{\mat C}_{R,N}^+$ and $\mat Y$.

Figure~\ref{fig:inverse} (left) shows how the number of iterations required by the PCG inversion algorithm evolves as the grid size $K$ increases for different values of the regularizer $\epsilon$, leading to 3 fixed condition numbers $\kappa = 10,10^2,10^3$. The results are again averages over 10 independent simulation runs. For a fixed condition number $\kappa$, we always find $\epsilon$ such that the condition number of $\mat C_{R,N}^+$ is exactly $\kappa$. We want to control the condition number of the left-hand side matrix because it generally captures the difficulty of the inversion problem \citep{vanloan1983}. In the case of PCG, the number of iterations is expected to grow roughly as the square-root of the condition number. As seen in Figure~\ref{fig:inverse} (left), the number of iterations needed for convergence does depend on the condition number in the expected manner, while the grid size $K$ does not affect the required number of iterations. This fact allows us to claim that the computational complexity of the inversion algorithm is $\O(K^3)$, i.e. the same as for the separable model.

\begin{remark}\label{rem:PCG_works}
We ran the PCG algorithm with a relatively stringent tolerance $10^{-10}$ (i.e. stopping the algorithm only when two subsequent iterates are closer than $10^{-10}$ in the Frobenius norm). The maximum recovery error across all simulation runs and all setups of the parameters was $3\cdot 10^{-10}$. Hence there is no doubt that the inversion algorithm converges.
\end{remark}

Finally, we explore the claim that a nearly separable model leads to milder computational costs than a highly non-separable model. We take $\widehat{\mat C}_{R,N}$ estimated with $N=500$ and with different values of $R=3,5,7$, and we are going to change its scores $\widehat{\sigma}_r$, $r=1,\ldots,R$. Firstly, we fix $\widetilde{\sigma}_1 \in \{ 0.15 , 0.25 , \ldots, 0.95 \} \cap \{ \sigma; \sigma \geq 1/R \}$. Then, we generate $\widetilde{\sigma}_r$ for $r=2,\ldots,R$ as a random variable uniformly distributed on the interval
\[
\left( \max\left( 0, 1 - \sum_{j=1}^{r-1} \widetilde{\sigma}_j - (R-r)\widetilde{\sigma}_1 \right),
\min\left(\widetilde{\sigma}_1, 1 - \sum_{j=1}^{r-1} \widetilde{\sigma}_j\right)\right).
\]
This leads to a collection of scores which are smaller than or equal to $\widetilde{\sigma}_1$ and they sum up (together with $\widetilde{\sigma}_1$) to 1. Lastly, we set
\[
\widetilde{\mat C}_{R,N}^\epsilon = \sum_{r=1}^R \widetilde{\sigma}_r \widehat{\mat A}_r \ct \widehat{\mat B}_r + \epsilon \mat I.
\]
This time, we do not look for $\epsilon$ in a way such that the condition number of $\widetilde{\mat C}_{R,N}^\epsilon$ is fixed, because we want to explore how the size of $\widetilde{\sigma}_1$ affects the number of iterations. However, a part of this effect is how $\widetilde{\sigma}_1$ affects the condition number, so we just standardize with several different (but fixed) values of $\epsilon$. Finally, we generate a random $\mat X \in \R^{K \times K}$, calculate $\mat Y = \widetilde{\mat C}_{R,N}^\epsilon \mat X$, and use the inversion algorithm to recover $\mat X$ from the knowledge of $\widetilde{\mat C}_{R,N}^\epsilon$ and $\mat Y$. Remark~\ref{rem:PCG_works} applies here as well.

\begin{figure}[!t]
  \centering
  \begin{tabular}{cc}
  \includegraphics[width=0.47\textwidth]{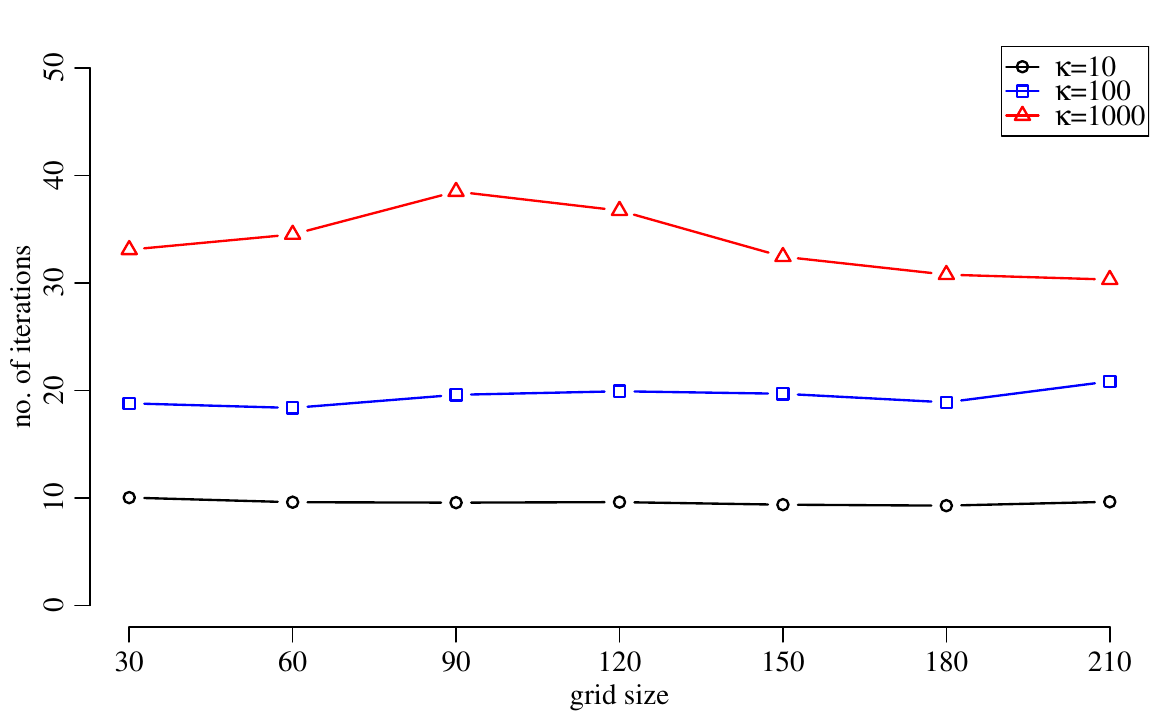} &
  \includegraphics[width=0.47\textwidth]{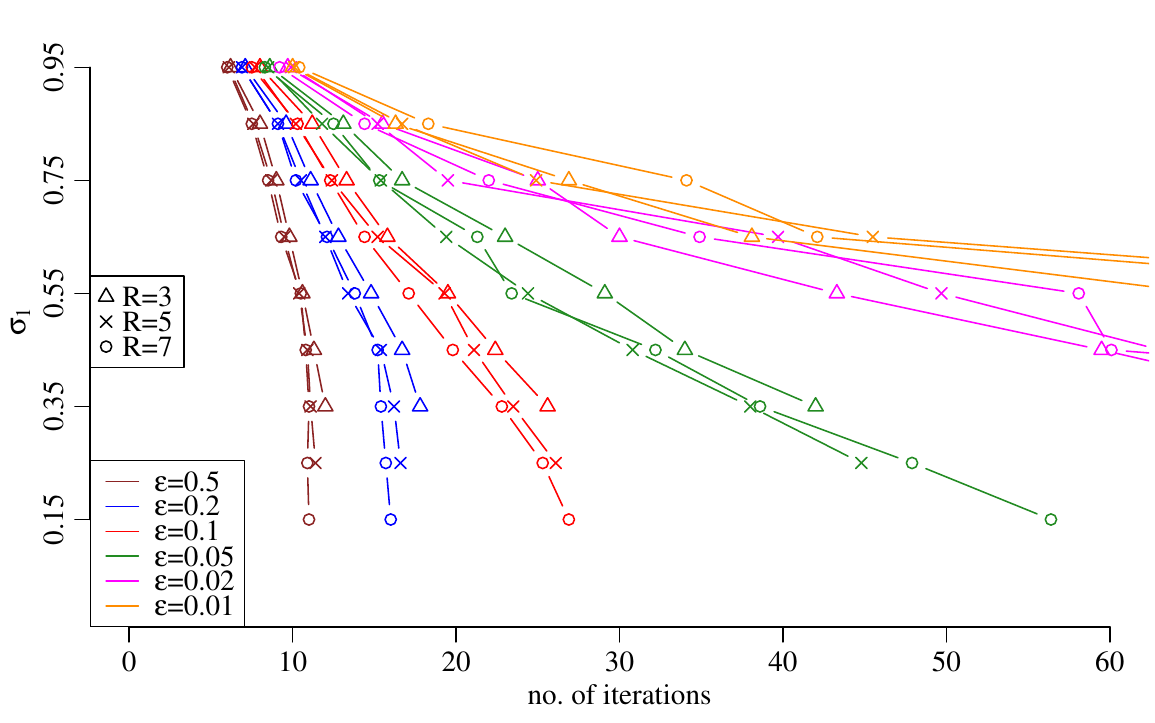}
  \end{tabular}  
         \caption{{\normalfont \textbf{Left:}} Number of iterations required by the PCG algorithm does not increase with increasing grid size $K$ when the condition number $\kappa$ is held fixed, $\kappa \in \{ 10, 10^2, 10^3 \}$. {\normalfont \textbf{Right:}} The smaller the leading score $\sigma_1$ is (relatively to other scores), the higher number of iterations is required by the PCG algorithm. Different regularization constants $\epsilon$ are distinguished by different colors, while different degrees-of-separability are distinguished by different symbols. Smaller degrees-of-separability prevent $\sigma_1$ from being too small, leading to naturally shorter curves, but otherwise the degree-of-separability is not affecting the results very much.}
    \label{fig:inverse} 
\end{figure}

The results are plotted in Figure~\ref{fig:inverse} (right). As expected, the better regularized problems with a larger $\epsilon$ generally require a smaller number of iterations. But more importantly, the number of iterations increases with decreasing $\widetilde{\sigma}_1$. This effect is milder for large $\epsilon$, but more severe for smaller regularization constants. This means that unless $C$ is well-posed (i.e. with a relatively large smallest eigenvalue $\lambda_{\min}$), we have to pay extra costs for very substantial departures from separability (i.e. when the largest score $\sigma_1$ is not much larger than the other scores).

\newpage
\bibliographystyle{imsart-nameyear}
\bibliography{biblio}

\begin{thebibliography}{34}
% BibTex style file: imsart-nameyear.bst, 2013-01-28
% Default style options (sort=1,type=nameyear).
% Used options (sort=1,type=nameyear).

\bibitem[\protect\citeauthoryear{Aston, Pigoli and Tavakoli}{2017}]{aston2017}
\begin{barticle}[author]
\bauthor{\bsnm{Aston},~\bfnm{John~AD}\binits{J.~A.}},
  \bauthor{\bsnm{Pigoli},~\bfnm{Davide}\binits{D.}} \AND
  \bauthor{\bsnm{Tavakoli},~\bfnm{Shahin}\binits{S.}}
(\byear{2017}).
\btitle{Tests for separability in nonparametric covariance operators of random
  surfaces}.
\bjournal{The Annals of Statistics}
\bvolume{45}
\bpages{1431--1461}.
\end{barticle}
\endbibitem

\bibitem[\protect\citeauthoryear{Bagchi and Dette}{2020}]{bagchi2020}
\begin{barticle}[author]
\bauthor{\bsnm{Bagchi},~\bfnm{Pramita}\binits{P.}} \AND
  \bauthor{\bsnm{Dette},~\bfnm{Holger}\binits{H.}}
(\byear{2020}).
\btitle{A test for separability in covariance operators of random surfaces}.
\bjournal{The Annals of Statistics}
\bvolume{48}
\bpages{2303--2322}.
\end{barticle}
\endbibitem

\bibitem[\protect\citeauthoryear{Ba{\'\i}llo, Cuevas and
  Fraiman}{2011}]{oxford2011}
\begin{barticle}[author]
\bauthor{\bsnm{Ba{\'\i}llo},~\bfnm{Amparo}\binits{A.}},
  \bauthor{\bsnm{Cuevas},~\bfnm{Antonio}\binits{A.}} \AND
  \bauthor{\bsnm{Fraiman},~\bfnm{Ricardo}\binits{R.}}
(\byear{2011}).
\btitle{Classification methods for functional data}.
\bjournal{The Oxford handbook of functional data analysis}.
\end{barticle}
\endbibitem

\bibitem[\protect\citeauthoryear{Bijma, De~Munck and
  Heethaar}{2005}]{munck2005}
\begin{barticle}[author]
\bauthor{\bsnm{Bijma},~\bfnm{Fetsje}\binits{F.}},
  \bauthor{\bsnm{De~Munck},~\bfnm{Jan~C}\binits{J.~C.}} \AND
  \bauthor{\bsnm{Heethaar},~\bfnm{Rob~M}\binits{R.~M.}}
(\byear{2005}).
\btitle{The spatiotemporal {MEG} covariance matrix modeled as a sum of
  {K}ronecker products}.
\bjournal{NeuroImage}
\bvolume{27}
\bpages{402--415}.
\end{barticle}
\endbibitem

\bibitem[\protect\citeauthoryear{Bosq}{2012}]{bosq2012}
\begin{bbook}[author]
\bauthor{\bsnm{Bosq},~\bfnm{Denis}\binits{D.}}
(\byear{2012}).
\btitle{Linear processes in function spaces: theory and applications}
\bvolume{149}.
\bpublisher{Springer Science \& Business Media}.
\end{bbook}
\endbibitem

\bibitem[\protect\citeauthoryear{Chen, Genton and Sun}{2021}]{chen2021}
\begin{barticle}[author]
\bauthor{\bsnm{Chen},~\bfnm{Wanfang}\binits{W.}},
  \bauthor{\bsnm{Genton},~\bfnm{Marc~G}\binits{M.~G.}} \AND
  \bauthor{\bsnm{Sun},~\bfnm{Ying}\binits{Y.}}
(\byear{2021}).
\btitle{Space-time covariance structures and models}.
\bjournal{Annual Review of Statistics and Its Application}
\bvolume{8}.
\end{barticle}
\endbibitem

\bibitem[\protect\citeauthoryear{Constantinou, Kokoszka and
  Reimherr}{2017}]{constantinou2017}
\begin{barticle}[author]
\bauthor{\bsnm{Constantinou},~\bfnm{Panayiotis}\binits{P.}},
  \bauthor{\bsnm{Kokoszka},~\bfnm{Piotr}\binits{P.}} \AND
  \bauthor{\bsnm{Reimherr},~\bfnm{Matthew}\binits{M.}}
(\byear{2017}).
\btitle{Testing separability of space-time functional processes}.
\bjournal{Biometrika}
\bvolume{104}
\bpages{425--437}.
\end{barticle}
\endbibitem

\bibitem[\protect\citeauthoryear{Delaigle and Hall}{2012}]{delaigle2012}
\begin{barticle}[author]
\bauthor{\bsnm{Delaigle},~\bfnm{Aurore}\binits{A.}} \AND
  \bauthor{\bsnm{Hall},~\bfnm{Peter}\binits{P.}}
(\byear{2012}).
\btitle{Achieving near perfect classification for functional data}.
\bjournal{Journal of the Royal Statistical Society: Series B (Statistical
  Methodology)}
\bvolume{74}
\bpages{267--286}.
\end{barticle}
\endbibitem

\bibitem[\protect\citeauthoryear{Dette, Dierickx and Kutta}{2020}]{dette2020}
\begin{barticle}[author]
\bauthor{\bsnm{Dette},~\bfnm{Holger}\binits{H.}},
  \bauthor{\bsnm{Dierickx},~\bfnm{Gauthier}\binits{G.}} \AND
  \bauthor{\bsnm{Kutta},~\bfnm{Tim}\binits{T.}}
(\byear{2020}).
\btitle{Quantifying deviations from separability in space-time functional
  processes}.
\bjournal{arXiv preprint arXiv:2003.12126}.
\end{barticle}
\endbibitem

\bibitem[\protect\citeauthoryear{Genton}{2007}]{genton2007}
\begin{barticle}[author]
\bauthor{\bsnm{Genton},~\bfnm{Marc~G}\binits{M.~G.}}
(\byear{2007}).
\btitle{Separable approximations of space-time covariance matrices}.
\bjournal{Environmetrics}
\bvolume{18}
\bpages{681--695}.
\end{barticle}
\endbibitem

\bibitem[\protect\citeauthoryear{Gneiting}{2002}]{gneiting2002}
\begin{barticle}[author]
\bauthor{\bsnm{Gneiting},~\bfnm{Tilmann}\binits{T.}}
(\byear{2002}).
\btitle{Nonseparable, stationary covariance functions for space--time data}.
\bjournal{Journal of the American Statistical Association}
\bvolume{97}
\bpages{590--600}.
\end{barticle}
\endbibitem

\bibitem[\protect\citeauthoryear{Gneiting, Genton and
  Guttorp}{2006}]{gneiting2006}
\begin{bincollection}[author]
\bauthor{\bsnm{Gneiting},~\bfnm{Tilmann}\binits{T.}},
  \bauthor{\bsnm{Genton},~\bfnm{Marc~G}\binits{M.~G.}} \AND
  \bauthor{\bsnm{Guttorp},~\bfnm{Peter}\binits{P.}}
(\byear{2006}).
\btitle{Geostatistical {S}pace-{T}ime {M}odels, {S}tationarity, {S}eparability,
  and {F}ull {S}ymmetry}.
In \bbooktitle{Statistical Methods for Spatio-Temporal Systems}
\bpages{151-175}.
\bpublisher{Chapman and Hall/CRC}.
\end{bincollection}
\endbibitem

\bibitem[\protect\citeauthoryear{Gupta and Nagar}{2000}]{gupta2000}
\begin{bmisc}[author]
\bauthor{\bsnm{Gupta},~\bfnm{A}\binits{A.}} \AND
  \bauthor{\bsnm{Nagar},~\bfnm{D}\binits{D.}}
(\byear{2000}).
\btitle{Matrix variate distributions. Monographs and surveys in pure and
  applied mathematics}.
\end{bmisc}
\endbibitem

\bibitem[\protect\citeauthoryear{Hall, M{\"u}ller and Wang}{2006}]{hall2006}
\begin{barticle}[author]
\bauthor{\bsnm{Hall},~\bfnm{Peter}\binits{P.}},
  \bauthor{\bsnm{M{\"u}ller},~\bfnm{Hans-Georg}\binits{H.-G.}} \AND
  \bauthor{\bsnm{Wang},~\bfnm{Jane-Ling}\binits{J.-L.}}
(\byear{2006}).
\btitle{Properties of principal component methods for functional and
  longitudinal data analysis}.
\bjournal{The Annals of Statistics}
\bpages{1493--1517}.
\end{barticle}
\endbibitem

\bibitem[\protect\citeauthoryear{Haslett and Raftery}{1989}]{haslett1989}
\begin{barticle}[author]
\bauthor{\bsnm{Haslett},~\bfnm{John}\binits{J.}} \AND
  \bauthor{\bsnm{Raftery},~\bfnm{Adrian~E}\binits{A.~E.}}
(\byear{1989}).
\btitle{Space-time modelling with long-memory dependence: {A}ssessing
  {I}reland's wind power resource}.
\bjournal{Journal of the Royal Statistical Society: Series C (Applied
  Statistics)}
\bvolume{38}
\bpages{1--21}.
\end{barticle}
\endbibitem

\bibitem[\protect\citeauthoryear{Hsing and Eubank}{2015}]{hsing2015}
\begin{bbook}[author]
\bauthor{\bsnm{Hsing},~\bfnm{Tailen}\binits{T.}} \AND
  \bauthor{\bsnm{Eubank},~\bfnm{Randall}\binits{R.}}
(\byear{2015}).
\btitle{Theoretical Foundations of Functional Data Analysis, With An
  Introduction to Linear Operators}.
\bpublisher{John Wiley \& Sons}.
\end{bbook}
\endbibitem

\bibitem[\protect\citeauthoryear{Huang and Sun}{2019}]{huang2019}
\begin{barticle}[author]
\bauthor{\bsnm{Huang},~\bfnm{Huang}\binits{H.}} \AND
  \bauthor{\bsnm{Sun},~\bfnm{Ying}\binits{Y.}}
(\byear{2019}).
\btitle{Visualization and assessment of spatio-temporal covariance properties}.
\bjournal{Spatial Statistics}
\bvolume{34}
\bpages{100272}.
\end{barticle}
\endbibitem

\bibitem[\protect\citeauthoryear{Jirak}{2016}]{jirak2016}
\begin{barticle}[author]
\bauthor{\bsnm{Jirak},~\bfnm{Moritz}\binits{M.}}
(\byear{2016}).
\btitle{Optimal eigen expansions and uniform bounds}.
\bjournal{Probability Theory and Related Fields}
\bvolume{166}
\bpages{753--799}.
\end{barticle}
\endbibitem

\bibitem[\protect\citeauthoryear{Jolliffe}{1986}]{Jolliffe:1986}
\begin{bbook}[author]
\bauthor{\bsnm{Jolliffe},~\bfnm{I.~T.}\binits{I.~T.}}
(\byear{1986}).
\btitle{Principal Component Analysis}.
\bpublisher{Springer, Verlag}.
\end{bbook}
\endbibitem

\bibitem[\protect\citeauthoryear{Lynch and Chen}{2018}]{lynch2018}
\begin{barticle}[author]
\bauthor{\bsnm{Lynch},~\bfnm{Brian}\binits{B.}} \AND
  \bauthor{\bsnm{Chen},~\bfnm{Kehui}\binits{K.}}
(\byear{2018}).
\btitle{A test of weak separability for multi-way functional data, with
  application to brain connectivity studies}.
\bjournal{Biometrika}
\bvolume{105}
\bpages{815--831}.
\end{barticle}
\endbibitem

\bibitem[\protect\citeauthoryear{Masak and Panaretos}{2019}]{masak2019}
\begin{barticle}[author]
\bauthor{\bsnm{Masak},~\bfnm{Tomas}\binits{T.}} \AND
  \bauthor{\bsnm{Panaretos},~\bfnm{Victor~M}\binits{V.~M.}}
(\byear{2019}).
\btitle{Spatiotemporal Covariance Estimation by Shifted Partial Tracing}.
\bjournal{arXiv preprint arXiv:1912.12870}.
\end{barticle}
\endbibitem

\bibitem[\protect\citeauthoryear{Pigoli et~al.}{2018}]{pigoli2018}
\begin{barticle}[author]
\bauthor{\bsnm{Pigoli},~\bfnm{Davide}\binits{D.}},
  \bauthor{\bsnm{Hadjipantelis},~\bfnm{Pantelis~Z}\binits{P.~Z.}},
  \bauthor{\bsnm{Coleman},~\bfnm{John~S}\binits{J.~S.}} \AND
  \bauthor{\bsnm{Aston},~\bfnm{John~AD}\binits{J.~A.}}
(\byear{2018}).
\btitle{The statistical analysis of acoustic phonetic data: exploring
  differences between spoken Romance languages}.
\bjournal{Journal of the Royal Statistical Society: Series C (Applied
  Statistics)}
\bvolume{67}
\bpages{1103--1145}.
\end{barticle}
\endbibitem

\bibitem[\protect\citeauthoryear{Prabhakar and Rajaguru}{2020}]{prabhakar2020}
\begin{barticle}[author]
\bauthor{\bsnm{Prabhakar},~\bfnm{Sunil~Kumar}\binits{S.~K.}} \AND
  \bauthor{\bsnm{Rajaguru},~\bfnm{Harikumar}\binits{H.}}
(\byear{2020}).
\btitle{Alcoholic {EEG} signal classification with {C}orrelation {D}imension
  based distance metrics approach and {M}odified {A}daboost classification}.
\bjournal{Heliyon}
\bvolume{6}
\bpages{e05689}.
\end{barticle}
\endbibitem

\bibitem[\protect\citeauthoryear{Ramsay and Silverman}{2002}]{ramsay2002}
\begin{bbook}[author]
\bauthor{\bsnm{Ramsay},~\bfnm{J.~O.}\binits{J.~O.}} \AND
  \bauthor{\bsnm{Silverman},~\bfnm{B.~W.}\binits{B.~W.}}
(\byear{2002}).
\btitle{Applied Functional Data Analysis: Methods and Case Studies}.
\bpublisher{Springer, New York}.
\end{bbook}
\endbibitem

\bibitem[\protect\citeauthoryear{Ramsay and Silverman}{2005}]{ramsay2005}
\begin{bbook}[author]
\bauthor{\bsnm{Ramsay},~\bfnm{J.~O.}\binits{J.~O.}} \AND
  \bauthor{\bsnm{Silverman},~\bfnm{B.~W.}\binits{B.~W.}}
(\byear{2005}).
\btitle{Functional data analysis}.
\bpublisher{Springer, New York}.
\end{bbook}
\endbibitem

\bibitem[\protect\citeauthoryear{Rougier}{2017}]{rougier2017}
\begin{barticle}[author]
\bauthor{\bsnm{Rougier},~\bfnm{Jonathan}\binits{J.}}
(\byear{2017}).
\btitle{A representation theorem for stochastic processes with separable
  covariance functions, and its implications for emulation}.
\bjournal{arXiv preprint arXiv:1702.05599}.
\end{barticle}
\endbibitem

\bibitem[\protect\citeauthoryear{{R Core Team}}{2020}]{R}
\begin{bmanual}[author]
\bauthor{\bsnm{{R Core Team}}}
(\byear{2020}).
\btitle{R: A Language and Environment for Statistical Computing}
\bpublisher{R Foundation for Statistical Computing},
\baddress{Vienna, Austria}.
\end{bmanual}
\endbibitem

\bibitem[\protect\citeauthoryear{Tsiligkaridis and Hero}{2013}]{hero2013a}
\begin{barticle}[author]
\bauthor{\bsnm{Tsiligkaridis},~\bfnm{Theodoros}\binits{T.}} \AND
  \bauthor{\bsnm{Hero},~\bfnm{Alfred~O}\binits{A.~O.}}
(\byear{2013}).
\btitle{Covariance estimation in high dimensions via {K}ronecker product
  expansions}.
\bjournal{IEEE Transactions on Signal Processing}
\bvolume{61}
\bpages{5347--5360}.
\end{barticle}
\endbibitem

\bibitem[\protect\citeauthoryear{Van~Loan}{2000}]{vanloan2000}
\begin{barticle}[author]
\bauthor{\bsnm{Van~Loan},~\bfnm{Charles~F}\binits{C.~F.}}
(\byear{2000}).
\btitle{The ubiquitous {K}ronecker product}.
\bjournal{Journal of Computational and Applied Mathematics}
\bvolume{123}
\bpages{85--100}.
\end{barticle}
\endbibitem

\bibitem[\protect\citeauthoryear{Van~Loan and Golub}{1983}]{vanloan1983}
\begin{bbook}[author]
\bauthor{\bsnm{Van~Loan},~\bfnm{Charles~F}\binits{C.~F.}} \AND
  \bauthor{\bsnm{Golub},~\bfnm{Gene~H}\binits{G.~H.}}
(\byear{1983}).
\btitle{Matrix computations}.
\bpublisher{Johns Hopkins University Press}.
\end{bbook}
\endbibitem

\bibitem[\protect\citeauthoryear{Van~Loan and Pitsianis}{1993}]{vanloan1993}
\begin{bincollection}[author]
\bauthor{\bsnm{Van~Loan},~\bfnm{Charles~F}\binits{C.~F.}} \AND
  \bauthor{\bsnm{Pitsianis},~\bfnm{Nikos}\binits{N.}}
(\byear{1993}).
\btitle{Approximation with {K}ronecker products}.
In \bbooktitle{Linear algebra for large scale and real-time applications}
\bpages{293--314}.
\bpublisher{Springer}.
\end{bincollection}
\endbibitem

\bibitem[\protect\citeauthoryear{Wand and Jones}{1994}]{wand1994}
\begin{bbook}[author]
\bauthor{\bsnm{Wand},~\bfnm{Matt~P}\binits{M.~P.}} \AND
  \bauthor{\bsnm{Jones},~\bfnm{M~Chris}\binits{M.~C.}}
(\byear{1994}).
\btitle{Kernel smoothing}.
\bpublisher{CRC press}.
\end{bbook}
\endbibitem

\bibitem[\protect\citeauthoryear{Wang, Chiou and M{\"u}ller}{2016}]{wang2016}
\begin{barticle}[author]
\bauthor{\bsnm{Wang},~\bfnm{Jane-Ling}\binits{J.-L.}},
  \bauthor{\bsnm{Chiou},~\bfnm{Jeng-Min}\binits{J.-M.}} \AND
  \bauthor{\bsnm{M{\"u}ller},~\bfnm{Hans-Georg}\binits{H.-G.}}
(\byear{2016}).
\btitle{Functional data analysis}.
\bjournal{Annual Review of Statistics and Its Application}
\bvolume{3}
\bpages{257--295}.
\end{barticle}
\endbibitem

\bibitem[\protect\citeauthoryear{Weidmann}{2012}]{weidmann2012}
\begin{bbook}[author]
\bauthor{\bsnm{Weidmann},~\bfnm{Joachim}\binits{J.}}
(\byear{2012}).
\btitle{Linear operators in Hilbert spaces}
\bvolume{68}.
\bpublisher{Springer Science \& Business Media}.
\end{bbook}
\endbibitem

\end{thebibliography}

\end{document}